\documentclass[11pt,reqno]{amsart}

\usepackage[T1]{fontenc}
\usepackage{lmodern}
\usepackage{microtype}
\usepackage[margin=1.08in]{geometry}
\usepackage{amsmath,amssymb,amsthm,mathtools,bm,mathrsfs}
\usepackage{booktabs,tabularx,array}
\usepackage{enumitem}
\usepackage{xcolor}
\usepackage{hyperref}
\usepackage[nameinlink,capitalise,noabbrev]{cleveref}

\crefname{theorem}{Theorem}{Theorems}
\Crefname{theorem}{Theorem}{Theorems}
\crefname{proposition}{Proposition}{Propositions}
\Crefname{proposition}{Proposition}{Propositions}
\crefname{lemma}{Lemma}{Lemmas}
\Crefname{lemma}{Lemma}{Lemmas}
\crefname{corollary}{Corollary}{Corollaries}
\Crefname{corollary}{Corollary}{Corollaries}
\crefname{definition}{Definition}{Definitions}
\Crefname{definition}{Definition}{Definitions}
\crefname{remark}{Remark}{Remarks}
\Crefname{remark}{Remark}{Remarks}
\crefname{conjecture}{Conjecture}{Conjectures}
\Crefname{conjecture}{Conjecture}{Conjectures}
\crefname{openproblem}{Open Problem}{Open Problems}
\Crefname{openproblem}{Open Problem}{Open Problems}

\hypersetup{
  colorlinks=true,
  linkcolor=blue!52!black,
  citecolor=blue!52!black,
  urlcolor=blue!60!black,
  pdftitle={Regularized Bulk Universality versus Bounded-Disorder Nonuniversality for Annealed Complexity of Spherical p-Spin Landscapes},
  pdfauthor={Taegyun Kim}
}

\numberwithin{equation}{section}

\newtheorem{theorem}{Theorem}[section]
\newtheorem{proposition}[theorem]{Proposition}
\newtheorem{lemma}[theorem]{Lemma}
\newtheorem{corollary}[theorem]{Corollary}
\newtheorem{conjecture}[theorem]{Conjecture}
\newtheorem{openproblem}[theorem]{Open Problem}
\theoremstyle{definition}
\newtheorem{definition}[theorem]{Definition}
\theoremstyle{remark}
\newtheorem{remark}[theorem]{Remark}

\newif\ifptrfextendedoutlook
\ptrfextendedoutlookfalse

\newcommand{\SN}{\mathbb S_N}
\newcommand{\ON}{\mathbb O(N)}
\newcommand{\Crt}{\operatorname{Crt}}
\newcommand{\Hess}{\operatorname{Hess}}

\newcommand{\Vol}{\operatorname{Vol}}
\newcommand{\Tr}{\operatorname{Tr}}

\newcommand{\supp}{\operatorname{supp}}

\newcommand{\E}{\mathbb E}
\newcommand{\Pp}{\mathbb P}
\newcommand{\R}{\mathbb R}

\newcommand{\1}{\mathbf 1}
\newcommand{\dd}{\,\mathrm d}
\newcommand{\eps}{\varepsilon}
\newcommand{\op}{\mathrm{op}}
\newcommand{\HS}{\mathrm{HS}}
\newcommand{\cN}{\mathcal N}
\newcommand{\Ncov}{\mathcal N_{\mathrm{cov}}}

\newcommand{\cF}{\mathcal F}
\newcommand{\cV}{\mathcal V}
\newcommand{\cI}{\mathcal I}
\newcommand{\cK}{\mathcal K}

\newcommand{\cO}{\mathcal O}
\newcommand{\cZ}{\mathcal Z}

\newcommand{\rhoSC}{\rho_{\mathrm{sc}}}
\newcommand{\norm}[1]{\left\lVert#1\right\rVert}
\newcommand{\abs}[1]{\left\lvert#1\right\rvert}
\newcommand{\ip}[2]{\left\langle#1,#2\right\rangle}
\newcommand{\tensor}{\otimes}

\title[Bulk universality versus nonuniversality]{Regularized Bulk Universality versus Bounded-Disorder Nonuniversality\\for Annealed Complexity of Spherical \(p\)-Spin Landscapes}
\author{Taegyun Kim}
\address{Department of Mathematical Sciences, KAIST, Daejeon, Republic of Korea}
\email{ktg11k@kaist.ac.kr}
\date{July 29, 2026}


\begin{document}
\raggedbottom

\begin{abstract}
Fix \(p\ge3\).  We establish a separation between regularized bulk
universality and unrestricted annealed complexity for the pure spherical
\(p\)-spin Hamiltonian with independent non-Gaussian tensor coordinates.
There is a symmetric, compactly supported disorder
law with a smooth density, matching the first \(2p\) Gaussian moments,
for which the expectation of a positive, frame-averaged regularization
of the one-point Kac--Rice functional is asymptotic to its Gaussian
counterpart, while the unrestricted critical-point count in a compact
energy window has a lower exponential rate strictly above the Gaussian
limit.  The obstruction is a coherent block involving on the order of
\(N^{1/p}\) coordinates, so neither finite moment matching nor a uniform
entry bound restores unregularized annealed universality.

For uniformly subexponential disorder matching the first \(m\) Gaussian
moments, the regularized functional on incoherent frames satisfies
\[
 \left|\log\frac{\E\cZ_{N,L}^J}{\E\cZ_{N,L}^G}\right|
 \le C Nq_N^{m-1},\qquad
 q_N=L^{p-1}N^{-(p-2)/2}(\log N)^{(p-1)/2}.
\]
Thus mean and variance matching imply regularized pressure universality
for every \(p\ge3\), and the expectation ratio tends to one when
\((m-1)(p-2)>2\).  We identify the Gaussian variational limit and prove
the Gaussian quadratic energy-excursion upper bound uniformly over profiles
asymptotically supported on \(o(N)\) coordinates under a Gaussian-rate
profile-tail condition.  Finally, we give a conditional reduction to
exact universality: once two one-sided de-regularization defects vanish,
an exact max formula makes control of the localized complement necessary
and sufficient.  The defects vanish in the Gaussian model.
\end{abstract}

\maketitle

\section{Introduction}

\subsection{The question and the obstruction}

Let
\[
 \SN=\mathbb S^{N-1}(\sqrt N)
 =\{\sigma\in\R^N:\norm{\sigma}_2^2=N\}.
\]
For a fixed integer \(p\ge3\), consider the ordered-tensor Hamiltonian
\begin{equation}\label{eq:model-intro}
 H_N^J(\sigma)
 =N^{-(p-1)/2}\sum_{i_1,\ldots,i_p=1}^N
 J_{i_1\cdots i_p}\sigma_{i_1}\cdots\sigma_{i_p},
 \qquad \sigma\in\SN,
\end{equation}
where the entries \(J_I\), \(I=(i_1,\ldots,i_p)\in[N]^p\), are independent, centered, and have variance one.  For a Borel set \(B\subset\R\), write
\begin{equation}\label{eq:Crt-intro}
 \Crt_N^J(B)=\#\{\sigma\in\SN:\nabla_SH_N^J(\sigma)=0,
 \ H_N^J(\sigma)/N\in B\}.
\end{equation}
Write \(\Crt_{N,\max}^J(B)\) and \(\Crt_{N,\min}^J(B)\) for the corresponding numbers of local maxima and local minima.
When the limit exists, the annealed complexity is the exponential rate
of \(\E\Crt_N^J(B)\); otherwise we use its upper and lower exponential
rates.

For Gaussian disorder, isotropy and the Kac--Rice formula reduce the problem to a shifted GOE determinant, leading to the explicit formulas of Auffinger, Ben Arous, and \v{C}ern\'y \cite{ABAC}.  A natural question is whether the same exponent holds for non-Gaussian entries under light-tail assumptions or finite moment matching.

The main conclusion of this paper is that two different answers coexist.

\begin{theorem}[Informal; bounded-disorder coexistence]
\label[theorem]{thm:intro-coexistence}
For every \(p\ge3\), there is a symmetric, compactly supported law
\(\mu\) with a \(C^\infty\) density, matching the standard Gaussian moments through
order \(2p\), for which the following hold simultaneously.  Every fixed
admissible incoherent-frame regularization has expectation asymptotic to
its Gaussian counterpart, whereas for some compact interval
\(B\subset(0,\infty)\),
\[
 \liminf_{N\to\infty}\frac1N\log\E_\mu\Crt_N^J(B)
 >
 \lim_{N\to\infty}\frac1N\log\E\Crt_N^G(B).
\]
The count on the left is unregularized and unrestricted in space.  The
same dichotomy holds in the standard symmetric-coordinate model.
\end{theorem}

The formal statement, including the admissible regularizers, is
\cref{thm:bounded-coexistence}.  Its two mechanisms are as follows.

\begin{enumerate}[label=\textbf{(\Alph*)},leftmargin=2.1em]
\item \emph{The delocalized regularized Kac--Rice pressure is universal.}  At a point whose coordinates are all \(O(\bigl(\log N/N\bigr)^{1/2})\), every tensor coefficient has small influence on a regularized Kac--Rice integrand, while the squared influences have total size only \(O(N)\).  A product-mixture Lindeberg interpolation preserves this aggregate gain.  Mean and variance matching gives pressure universality for every \(p\ge3\), and higher matching yields a strong ratio theorem.  This statement concerns the regularized observable, not an unregularized non-Gaussian critical-point count.

\item \emph{Finite moments do not force the Gaussian annealed
asymptotics of the unregularized energy-window count over all critical
points.}  Here ``unregularized'' means that neither a Kac--Rice
mollifier nor a spatial localization restriction is imposed.  A
single coefficient of size \(a\sqrt N\) contributes order \(N\) near a
coordinate pole.  Its probability is controlled by the quadratic-scale
tail of the entry law, which is invisible to any prescribed finite list
of moments.  Entry truncation alone does not remove the phenomenon:
a coherent block of fixed-size entries on \(N^{1/p}\) coordinates also
contributes order \(N\), at speed-\(N\) probability cost.  We construct
compactly supported smooth moment-matching laws for which this
mesoscopic lower rate exceeds the Gaussian energy-window complexity.
We do not assert existence of the non-Gaussian limit.
\end{enumerate}

The one-point Kac--Rice formula used throughout is classical; we refer to
Adler and Taylor \cite{AdlerTaylor} for a general account.  In the
Gaussian pure model, Auffinger, Ben Arous, and \v{C}ern\'y obtained the
exact GOE representation and annealed complexity \cite{ABAC}; related
formulas for general isotropic Gaussian fields were developed in
\cite{AuffingerBenArous}, while Gaussian typical-count results based on
second moments were established in \cite{Subag,SubagZeitouni}.  The
recent work of Aza\"is, Dalmao, and De Castro gives non-asymptotic
bracketing for the Gaussian KSS field \cite{AzaisDalmaoDeCastro}.
Chen, Lu, and Sen derive a variational complexity formula for a Gaussian
\(p\)-spin Hamiltonian with a non-rotationally invariant potential
\cite{ChenLuSen}.  Neither work treats replacement of the Gaussian tensor
coordinates or the bounded finite-moment obstruction proved here.
Complexity formulas for non-Gaussian empirical-risk landscapes and other
non-invariant models appear in \cite{MaillardBAB,BKMN}, while
deterministic spiked polynomial and tensor deformations were studied in
\cite{BAMMN,ABLspiked,PiccoloSpiked}.

Our replacement argument is multiplicative rather than additive, but is
conceptually related to Lindeberg invariance principles such as
\cite{Chatterjee}.  Random-determinant asymptotics beyond invariant
ensembles were developed in \cite{BBMdet,BKMN}.  Sawhney and Sellke
proved free-energy and ground-state universality for spherical spin
glasses under the sharp finite-\(2p\)-moment condition
\cite{SawhneySellke}; their result concerns thermodynamic observables,
not Kac--Rice critical-point counts.  A companion preprint by the author
studies the corresponding free-energy dichotomy and the localized
thermodynamic branch created by extremal couplings \cite{KimHeavy}; it
does not treat critical points or Kac--Rice complexity.  The
counterexamples below show that annealed critical-point counts are
sensitive both to quadratic-scale tails and to mesoscopic coherent
deviations, even within bounded disorder.

This distinction mirrors a broader principle in non-invariant random-matrix problems: bulk spectral observables can be universal while annealed edge large deviations retain a localized branch that depends on the entry law.  For sub-Gaussian Wigner matrices, such a transition is known rigorously \cite{AGH,CDG}.  The algebraic interpretation of spherical critical points as real tensor eigenvectors, and the resulting generic complex count, goes back to Cartwright and Sturmfels \cite{CartwrightSturmfels}.  Our results identify the analogous obstruction directly at the level of spherical tensor landscapes.

\subsection{Contributions}
\enlargethispage{\baselineskip}

The paper contains two groups of unconditional results and one conditional
reduction.

\begin{enumerate}[label=\arabic*.,leftmargin=2.1em]
\item \emph{Regularized bulk universality.}  We define a frame-averaged
regularized Kac--Rice partition function \(\cZ_{N,L}^J\) on incoherent
frames; a variant requires only point incoherence, with arbitrary
tangent-frame completion.  See
\cref{def:regularized-Z,cor:point-incoherent}.  For subexponential
coordinates matching the first \(m\) Gaussian moments, we prove
\begin{equation}\label{eq:intro-ratio}
 \left|\log\frac{\E\cZ_{N,L}^J(V,\eps,f)}
 {\E\cZ_{N,L}^G(V,\eps,f)}\right|
 \le C L^{(p-1)(m-1)}
 N^{1-(m-1)(p-2)/2}
 (\log N)^{(p-1)(m-1)/2}.
\end{equation}
Consequently, the difference of the two pressures tends to zero for
every \(p\ge3,m\ge2\), and for \(L>L_0\) both converge to the Gaussian
variational limit.  The ratio tends to one whenever
\((m-1)(p-2)>2\); see \cref{thm:bulk-universality}.  No density, identical
distribution, or Morse assumption is used here.  A truncation argument
gives the same comparison under Weibull upper tails in the high-degree
window \(\alpha(p-2)>2\); see \cref{thm:weibull-bulk}.  The same proofs
cover the standard symmetric-coordinate tensor model; see
\cref{cor:symmetric-bulk}.

For the point-incoherent variant, we also allow a growing cutoff
\(\norm{x}_\infty\le r_N\).  With
\(q_N^{\rm gr}=\sqrt N\,r_N^{p-1}\), the same multiplicative estimate
holds with a constant independent of \(r_N\); hence
\(q_N^{\rm gr}\to0\) gives pressure universality.  This includes
Haar-full cutoffs much larger than
\(\sqrt{\log N/N}\); see
\cref{thm:growing-point-cutoff,cor:growing-cutoff-regimes}.  The
mesoscopic block lies strictly beyond this regime: any cutoff meeting
its cap has \(q_N^{\rm gr}\ge N^{1/(2p)+o(1)}\); see
\cref{cor:block-growing-cutoff-separation}.

We evaluate the Gaussian regularized pressure by an exact derivative-jet
computation and a GOE linear-statistic argument; see
\cref{thm:gaussian-variational}.  The iterated determinant/delta
de-regularization of the limiting formula, with \(N\to\infty\) taken
first, produces the pointwise Gaussian complexity
\begin{equation}\label{eq:intro-theta}
 \theta_p(u)=\frac12\log(p-1)-\frac{p-2}{4(p-1)}u^2-\cI_p(u),
\end{equation}
where \(\cI_p\) is the semicircle edge cost in \eqref{eq:Ip-def}.  A speed-\(N\) large-deviation principle for the expected GOE empirical measure then identifies the same formula with the exact Gaussian weighted critical-point complexity.

\item \emph{Localized obstructions for unregularized counts.}  We prove a one-spike lower
bound for unregularized local-maximum counts; see \cref{thm:one-spike}.  We then
construct a smooth sub-Gaussian law matching Gaussian moments through
order \(2p\) for which
\begin{equation}\label{eq:intro-counter}
 \liminf_{N\to\infty}\frac1N\log\E\Crt_N^J(B)
 >\lim_{N\to\infty}\frac1N\log\E\Crt_N^G(B)
\end{equation}
for a suitable compact high-energy interval \(B\), with no spatial
localization restriction on the count; see
\cref{thm:counterexample}.  Only a non-Gaussian lower rate is claimed.
The strict separation also holds for local maxima and, by symmetry,
local minima.  More generally, no prescribed finite matching order
forces the Gaussian energy-window rate; see
\cref{cor:no-finite-moment}.  In fact, a compactly supported smooth law
 with any prescribed finite matching order has a strict gap created by
 a coherent block involving on the order of \(N^{1/p}\) coordinates; see
\cref{thm:bounded-block-counterexample,cor:symmetric-bounded-block}.
For every fixed \(0\le\varepsilon<1/2\), the constructed law
automatically satisfies the \(N^{1/2-\varepsilon}\) entrywise cutoff for
all large \(N\).  For the
same bounded law, the regularized expectation ratio nevertheless
converges to one; see
  \cref{thm:bounded-coexistence}.  We also obtain one-spike
coexistence, Weibull-tail, and diagonal-free
variants, with the scope stated in
\cref{cor:coexistence,thm:subexp-counterexample,cor:diagonal-free-spike}.
Under the weaker profile-tail condition
\[
 \limsup_N\frac1N\log\sup_{\|x\|_2=1}
 \Pp\{|h_N^J(x)|\ge a\}\le-\frac{a^2}{2}
 \qquad(a>0),
\]
we prove the Gaussian quadratic upper bound \(a^2/2\) uniformly over all
profiles supported on \(o(N)\) coordinates and over profiles whose mass
outside such a support tends to zero.
The profile MGF condition \(({\rm PSG})_p\) is sufficient for the
displayed tail rate; in particular, the stronger joint sharp
sub-Gaussian bound is sufficient.
Combining the tail rate with the
exponential critical-point ceiling \({\rm(CC)}_p\) gives
\[
 \lim_{\rho\downarrow0}\limsup_N\frac1N\log
 \E\Crt_{N,k_N,\rho}^J([a,\infty))
 \le\log(p-1)-\frac{a^2}{2},
 \qquad k_N=o(N);
\]
see
\cref{prop:sharp-profile-energy,cor:sharp-approximately-sparse,cor:sharp-sparse-critical-upper}.
For comparison, under an arbitrary joint disorder law, conditioning on
any event of probability \(1-o(1)\) changes the law of
\(N^{-1}\log(1+\Crt_N^J(B))\) by \(o(1)\) in total variation.
Uniformly sub-Gaussian marginals, without independence, make
\(\max_I|J_I|\le N^{1/2-\varepsilon}\) such an event for every
\(0\le\varepsilon<1/2\).  The ceiling \({\rm(CC)}_p\), or merely a
suitable integrability bound, is needed only to transfer expectations;
see \cref{prop:quenched-cutoff}.  This does not prove a non-Gaussian
quenched limit and says nothing analogous for the logarithm of the
annealed expectation.

\item \emph{Quantitative transfer to exact universality.}  We do not
prove exact non-Gaussian universality for the whole disorder class.
Instead of imposing two named assumptions, we define the minimal
one-sided pressure defects and prove quantitative upper and lower
transfer bounds; see \cref{thm:conditional-transfer}.  The uncut
smoothed determinant \(\ell_\eta\) removes the spectral cutoff and all
remote-tail control.  A primitive criterion leaves only two directional
gradient-slice errors and, for the lower bound, an explicit near-zero
resolvent defect; see \cref{prop:primitive-dereg}.  Fixed Gaussian
divisibility proves the upper slice comparison, and all defects vanish
for Gaussian disorder; see
\cref{prop:gaussian-divisible-upper-slice,prop:Gaussian-dereg-benchmark}.
The localized complement satisfies an exact max formula, so its upper
branch bound is necessary and sufficient once the bulk slice has been
de-regularized.  Under the explicit compatibility condition
\(\Xi_p^{\rm sp}(V)\le\Phi_p^{\rm bulk}(V)\), the only further spatial
branch is the residual mixed-profile region; for a general \(V\), the
sparse weighted branch remains as well.  See
\cref{prop:bulk-sparse-reduction}.
\end{enumerate}

\subsection{Inputs and scope}

The proofs use three standard external inputs: the semicircle law and
Gaussian concentration for GOE matrices, the exact Gaussian spherical
\(p\)-spin Kac--Rice identity and its one-point asymptotics \cite{ABAC},
and ground-state tightness for spherical disorder with uniformly bounded
high moments \cite{SawhneySellke}.  The last input is used only for the
Weibull-tail counterexample, after the ordered tensor has been grouped
over permutation orbits.  The exact-complexity statement in
\cref{thm:conditional-transfer} is quantitative in two explicitly
defined defects; their vanishing is not claimed for general
non-Gaussian disorder.  Likewise,
\cref{prop:quenched-cutoff} is only a high-probability conditioning
equivalence for a sample log-count, not a quenched universality theorem.
The finite-support variational description and a quenched theorem in
\cref{sec:conjecture} remain open.  Smoothness or
absolute continuity alone does not imply an exponential weighted
gradient-slice comparison or the near-zero determinant estimate in
\cref{prop:primitive-dereg}.

\subsection{Organization}

\Cref{sec:model} defines the model and states the main results.  Spherical
derivative identities and the full derivative jet are developed in
\cref{sec:jet}.  Incoherent frames are treated in \cref{sec:frames}.
The multiplicative Lindeberg principle is proved in
\cref{sec:lindeberg}, and coordinate-influence estimates in
\cref{sec:influence}; together they prove bulk universality in
\cref{sec:bulk-proof}.  The Gaussian variational formula and its
de-regularization are established in \cref{sec:gaussian,sec:dereg}.  The
one-entry and mesoscopic-block obstructions, the cutoff equivalence, and
the moment-matching counterexamples are proved in \cref{sec:tail}.
\Cref{sec:transfer} contains the exact-complexity transfer principle,
and \cref{sec:conjecture} summarizes the remaining localized and quenched
questions.  Matrix derivative estimates,
the smooth moment-matching construction, finite-\(N\) Kac--Rice
de-regularization, and the frame-slice criteria for exponential
de-regularization are included in the appendices.

\section{Model, regularized complexity, and main results}\label{sec:model}

\subsection{The ordered spherical tensor model}

For \(x\in\mathbb S^{N-1}(1)\), put \(\sigma=\sqrt N x\) and define the energy density
\begin{equation}\label{eq:h-def}
 h_N^J(x):=\frac1N H_N^J(\sqrt N x)
 =N^{-1/2}\sum_{I\in[N]^p}J_I x_I,
 \qquad x_I:=\prod_{r=1}^p x_{i_r}.
\end{equation}
The ordered convention is used because it makes the independent coordinates explicit.  In the Gaussian case, the covariance is
\begin{equation}\label{eq:covariance-intro}
 \E H_N^G(\sigma)H_N^G(\tau)
 =N\left(\frac{\ip{\sigma}{\tau}}N\right)^p,
\end{equation}
so the field is isotropic and has the same law as the usual pure spherical \(p\)-spin field.

\begin{remark}[Tensor convention and algebraic ceiling]\label[remark]{rem:tensor-convention}
The ordered convention makes the independent disorder coordinates
explicit.  Let \(\mathfrak A_{N,p}\) be the set of permutation orbits in
\([N]^p\).  If \(\mathfrak a\in\mathfrak A_{N,p}\),
\(d_{\mathfrak a}=\abs{\mathfrak a}\), and
\(\sigma^{\mathfrak a}\) is the common monomial on that orbit, then
\[
 S_{\mathfrak a}:=\sum_{I\in\mathfrak a}J_I,
 \qquad
 H_N^J(\sigma)=N^{-(p-1)/2}\sum_{\mathfrak a}
 S_{\mathfrak a}\sigma^{\mathfrak a}.
\]
Equivalently, extending \(T_{\mathfrak a}=S_{\mathfrak a}/d_{\mathfrak a}\) over its orbit gives a symmetric tensor.  For Gaussian ordered entries, the variables \(S_{\mathfrak a}/\sqrt{d_{\mathfrak a}}\) are independent standard Gaussians and \(\operatorname{Var}(T_{\mathfrak a})=d_{\mathfrak a}^{-1}\), which is the standard symmetric pure \(p\)-spin normalization.  The replacement proof below also applies to independent non-identically distributed coordinates whose subexponential and moment-matching constants are uniform.

We will also use the standard symmetric-coordinate model
\begin{equation}\label{eq:symmetric-coordinate-model}
 H_{N,\mathrm{sym}}^\xi(\sigma)
 =N^{-(p-1)/2}\sum_{\mathfrak a\in\mathfrak A_{N,p}}
 \sqrt{d_{\mathfrak a}}\,\xi_{\mathfrak a}\sigma^{\mathfrak a},
\end{equation}
where the \(\xi_{\mathfrak a}\)'s are independent standardized
coordinates.  Equivalently, the raw symmetric entry on an ordered index
\(I\) is
\[
 T_I=\frac{\xi_{[I]}}{\sqrt{d_{[I]}}},
\]
where \([I]\) is its orbit; thus
\(\operatorname{Var}(T_I)=d_{[I]}^{-1}\).  Tail and moment assumptions
in this convention are always imposed on the standardized independent
coordinates \(\xi_{\mathfrak a}\), not on the dependent raw entries
\(T_I\).  For Gaussian \(\xi_{\mathfrak a}\), the covariance is exactly
\eqref{eq:covariance-intro}.  Thus this is not a different Gaussian
model; it is the usual choice of independent coordinates in the
symmetric coefficient space.

Its critical points are real tensor eigenvectors.  The generic complex
projective eigenvector count of Cartwright and Sturmfels
\cite{CartwrightSturmfels} gives the deterministic ceiling
\[
 \Crt_N^J(\R)\le
 2\frac{(p-1)^N-1}{p-2}
\]
away from the algebraic discriminant.  The regularized bulk theorem does
not require avoidance of that discriminant.
\end{remark}

\begin{definition}[Algebraic genericity]\label[definition]{def:algebraic-genericity}
Let \(\Delta_{N,p}\) denote the union of the proper real algebraic
discriminants in ordered coefficient space on which the spherical
critical set fails to be finite and nondegenerate, or on which the
generic tensor-eigenvector ceiling fails.  Let
\(J_N=(J_I)_{I\in[N]^p}\).  We say that the coefficient array is
algebraically generic, abbreviated \({\rm(AG)}\), if
\begin{equation}\label{eq:algebraic-genericity}
 \Pp\{J_N\in\Delta_{N,p}\}=0.
\end{equation}
Then, almost surely,
\[
 \Crt_N^J(\R)\le
 2\frac{(p-1)^N-1}{p-2}.
\]
Joint absolute continuity is one sufficient condition for
\({\rm(AG)}\), by \cref{lem:Morse} and the generic tensor-eigenvector
discriminant above, but it is not necessary.  In the symmetric-coordinate
convention, \(\Delta_{N,p}\) is taken in the symmetric coefficient space.
\end{definition}

\begin{definition}[Exponential critical-count ceiling]
\label[definition]{def:critical-count-ceiling}
We say that the array satisfies \({\rm(CC)}_p\) if there are deterministic
finite numbers \(D_{N,p}\) such that, almost surely,
\begin{equation}\label{eq:critical-count-ceiling}
 \Crt_N^J(\R)\le D_{N,p},
 \qquad
 \limsup_{N\to\infty}\frac1N\log D_{N,p}\le\log(p-1).
\end{equation}
This is the exact counting input used by the localized upper bounds
below; it requires neither nondegeneracy nor absolute continuity.
Algebraic genericity is a convenient sufficient condition, with
\(D_{N,p}=2\{(p-1)^N-1\}/(p-2)\).
\end{definition}

Let \(n=N-1\).  Write a frame as
\[
 O=(v_1,\ldots,v_n,x)\in\ON,
\]
where \(x\) is the last column and \(v_1,\ldots,v_n\) form an orthonormal basis of \(T_x\mathbb S^{N-1}(1)\).  At \(\sigma=\sqrt N x\), set
\begin{align}
 h_J(O)&=h_N^J(x),\label{eq:frame-h}\\
 g_{a,J}(O)&=D H_N^J(\sqrt N x)[v_a],\qquad 1\le a\le n,\label{eq:frame-g}\\
 A_{ab,J}(O)&=\Hess_S H_N^J(\sqrt N x)[v_a,v_b].\label{eq:frame-A}
\end{align}
Thus \(g_J(O)\in\R^n\) and \(A_J(O)\in\operatorname{Sym}_n\).

\subsection{Tail and moment assumptions}

\begin{definition}[Subexponential disorder]\label[definition]{def:subexp}
A real random variable \(\xi\) is \(K\)-subexponential if
\begin{equation}\label{eq:psi1}
 \E\exp(\abs{\xi}/K)\le2.
\end{equation}
It is \(K\)-sub-Gaussian if \(\E\exp(\xi^2/K^2)\le2\).  Sub-Gaussianity implies subexponentiality after changing \(K\).
\end{definition}

\begin{definition}[Weibull tail index]\label[definition]{def:stretched-tail}
For \(\alpha>0\), a real random variable \(\xi\) has Weibull tails of index \(\alpha\) if there is \(c_\xi\in(0,\infty)\) such that
\[
 \lim_{t\to\infty}t^{-\alpha}\log\Pp(\abs\xi>t)=-c_\xi.
\]
\end{definition}

\begin{definition}[Gaussian moment matching]\label[definition]{def:matching}
Let \(G\sim\cN(0,1)\).  We say that \(\xi\) matches the Gaussian moments through order \(m\) if
\begin{equation}\label{eq:matching}
 \E\xi^r=\E G^r,
 \qquad r=1,\ldots,m.
\end{equation}
In particular, for \(m\ge2\), \(\xi\) is centered and has variance one.
\end{definition}

\subsection{Incoherent frames and regularizers}

For \(L>0\), define
\begin{equation}\label{eq:good-frames}
 \cO_{N,L}
 :=\left\{O\in\ON:\max_{1\le i,j\le N}\abs{O_{ij}}
 \le L\sqrt{\frac{\log N}{N}}\right\}.
\end{equation}
Let \(\nu_N\) be normalized Haar measure on \(\ON\), and let
\[
 \omega_N=\Vol\bigl(\mathbb S^{N-1}(\sqrt N)\bigr).
\]

For \(\eps>0\), use the Cauchy approximate identity
\begin{equation}\label{eq:kappa}
 \kappa_\eps(t)=\frac1\pi\frac{\eps}{t^2+\eps^2}.
\end{equation}
The logarithmic derivatives satisfy
\begin{equation}\label{eq:kappa-deriv-bound}
 \norm{(\log\kappa_\eps)^{(r)}}_\infty\le C_{r,\eps},
 \qquad r\ge1.
\end{equation}

Fix an integer \(m\ge2\).  For \(\eta>0\), put
\begin{equation}\label{eq:ell-eta}
 \ell_\eta(t)=\frac12\log(t^2+\eta^2).
\end{equation}
We use the following function classes.

\begin{definition}[Admissible energy and spectral weights]\label[definition]{def:function-classes}
Let \(\cV_m\) be the class of \(V\in C^{m+1}(\R)\) such that \(\inf V> -\infty\) and
\[
 \max_{1\le r\le m+1}\norm{V^{(r)}}_\infty<\infty.
\]
Let \(\cF_m\) be the Fourier--Wiener class of real-valued functions
\[
 f(x)=c+\int_\R e^{itx}\widehat f_0(t)\dd t,
 \qquad c\in\R,
\]
where
\begin{equation}\label{eq:fourier-f}
 \int_\R(1+\abs t^{m+1})\abs{\widehat f_0(t)}\dd t<\infty.
\end{equation}
Here \(\widehat f_0(-t)=\overline{\widehat f_0(t)}\), so \(f\) is real.
Thus compact support and \(C^\infty\) regularity are not imposed; the
displayed integrability is exactly what the trace-derivative proof uses.
\end{definition}

\begin{definition}[Delocalized regularized Kac--Rice partition function]\label[definition]{def:regularized-Z}
For a Borel measurable function \(V\) bounded below,
\(f\in\cF_m\) or \(f=\ell_\eta\) with fixed \(\eta>0\), and
\(\eps>0\), define
\begin{align}
 \cZ_{N,L}^J(V,\eps,f)
 :=\omega_N\int_{\cO_{N,L}}
 &\exp\biggl\{-NV\bigl(h_J(O)\bigr)
 +\sum_{a=1}^n\log\kappa_\eps\bigl(g_{a,J}(O)\bigr)\notag\\
 &\hspace{36mm}+\Tr f\bigl(A_J(O)\bigr)\biggr\}\,\nu_N(\dd O).
 \label{eq:Z-def}
\end{align}
Its annealed pressure is
\begin{equation}\label{eq:pressure-def}
 \Phi_{N,L}^J(V,\eps,f)
 :=\frac1N\log\E\cZ_{N,L}^J(V,\eps,f).
\end{equation}
\end{definition}

For every fixed \(L>0\), the set \(\cO_{N,L}\) has positive Haar measure for all sufficiently large \(N\).  Indeed, a discrete cosine transform matrix has entries bounded by \(\sqrt{2/N}\), which is strictly below \(L\sqrt{\log N/N}\) for large \(N\); continuity then gives an open subset of admissible frames.  In particular, \(\cZ_{N,L}^J(V,\eps,f)>0\), as required by the multiplicative comparison below.

The factor \(\omega_N\) converts Haar averaging of the last column to spherical volume.  Formally, as \(\eps\downarrow0\) and \(f(t)\to\log\abs t\), the integrand becomes
\[
 e^{-NV(h)}\delta_0(g)\abs{\det A},
\]
which is the Kac--Rice density for a weighted critical-point count.  The frame restriction removes coordinate-localized points and simultaneously makes every disorder coordinate have small influence.

\subsection{Bulk universality}

Define the one-coordinate incoherent scale
\begin{equation}\label{eq:qN-main}
 q_N:=L^{p-1}N^{-(p-2)/2}(\log N)^{(p-1)/2}.
\end{equation}

\begin{theorem}[Multiplicative bulk universality]\label[theorem]{thm:bulk-universality}
Fix \(p\ge3\), \(m\ge2\), \(L>0\), \(\eps>0\),
\(V\in\cV_m\), and either \(f\in\cF_m\) or
\(f=\ell_\eta\) for fixed \(\eta>0\).  For each \(N\), let
\(J=(J_{I,N})_{I\in[N]^p}\) have independent, not necessarily
identically distributed coordinates satisfying
\begin{equation}\label{eq:array-assumptions}
 \sup_{N,I}\E e^{\abs{J_{I,N}}/K}\le2,
 \qquad
 \E J_{I,N}^r=\E G^r,\quad 1\le r\le m,
\end{equation}
where \(G\sim\cN(0,1)\).  Let \((G_I)\) be independent standard Gaussians.  Then there is a constant
\[
 C=C(p,m,L,K,\eps,V,f)<\infty
\]
such that, for all sufficiently large \(N\),
\begin{equation}\label{eq:bulk-univ-bound}
 \left|\log\frac{\E\cZ_{N,L}^J(V,\eps,f)}
 {\E\cZ_{N,L}^G(V,\eps,f)}\right|
 \le C Nq_N^{m-1}.
\end{equation}
Consequently, for every \(p\ge3\) and every \(m\ge2\),
\begin{equation}\label{eq:pressure-univ-weak}
 \Phi_{N,L}^J(V,\eps,f)-\Phi_{N,L}^G(V,\eps,f)\longrightarrow0.
\end{equation}
Assume additionally that
\begin{equation}\label{eq:moment-threshold}
 (m-1)(p-2)>2.
\end{equation}
Then
\begin{equation}\label{eq:ratio-one}
 \frac{\E\cZ_{N,L}^J(V,\eps,f)}
 {\E\cZ_{N,L}^G(V,\eps,f)}\longrightarrow1,
\end{equation}
and
\begin{equation}\label{eq:pressure-univ}
 \Phi_{N,L}^J(V,\eps,f)-\Phi_{N,L}^G(V,\eps,f)=o(N^{-1}).
\end{equation}
\end{theorem}

\begin{remark}[The actual tilted-moment hypothesis]\label[remark]{rem:TEM}
For bounded spectral weights \(f\in\cF_m\), the uniform
subexponential condition in \eqref{eq:array-assumptions} is a convenient
distribution-free sufficient condition, not the weakest input to the
aggregate proof.  With \(r=m+1\), it may be replaced by
\begin{equation}\label{eq:TEM}
 \sup_{N,I}\E\left[
 (1+\abs{J_{I,N}}^{2r})e^{bq_N\abs{J_{I,N}}}\right]<\infty
 \quad\text{for every fixed }b<\infty,
\end{equation}
together with matching through order \(m\).  This is precisely the
tilted-moment envelope used in
\eqref{eq:conditional-tilt-bounds}; the remaining aggregate argument is
unchanged.  For \(f=\ell_\eta\), we retain uniform subexponentiality:
the uncut determinant has polynomial degree growing with \(N\), so its
finite-\(N\) expectation uses moments of all orders.  For growing point
cutoffs, the same statement holds with \(q_N^{\rm gr}\) in place of
\(q_N\).
\end{remark}

\begin{corollary}[Moment thresholds]\label[corollary]{cor:moment-thresholds}
For every \(p\ge3\), matching only the mean and variance suffices for the regularized bulk-pressure conclusion \eqref{eq:pressure-univ-weak}.  The least matching order sufficient for the ratio conclusion is
\begin{equation}\label{eq:minimal-m}
 m_\star(p)=\left\lfloor\frac{2}{p-2}\right\rfloor+2
 =\begin{cases}
 4,&p=3,\\
 3,&p=4,\\
 2,&p\ge5.
 \end{cases}
\end{equation}
In particular, matching only mean and variance suffices for the stronger ratio conclusion when \(p\ge5\).
\end{corollary}

\begin{proof}
Since \(q_N\to0\), dividing \eqref{eq:bulk-univ-bound} by \(N\) proves pressure convergence for every \(m\ge2\).  The ratio bound tends to zero precisely under the strict power condition \(m-1>2/(p-2)\).  The least admissible integer is the displayed \(m_\star(p)\).
\end{proof}

\begin{corollary}[Sequential approximate moment matching]\label[corollary]{cor:approx-matching}
Retain the setup of \cref{thm:bulk-universality}, except that exact matching in \eqref{eq:array-assumptions} is not assumed, and set
\[
 \delta_{I,r,N}=\abs{\E J_{I,N}^r-\E G^r},
 \qquad 1\le r\le m.
\]
Here
\[
 q_N=L^{p-1}N^{-(p-2)/2}(\log N)^{(p-1)/2},
\]
as in \eqref{eq:qN-main}.  Then, for all sufficiently large \(N\),
\begin{equation}\label{eq:approx-moment-bound}
 \left|\log\frac{\E\cZ_{N,L}^J(V,\eps,f)}
 {\E\cZ_{N,L}^G(V,\eps,f)}\right|
 \le C\sum_{I\in[N]^p}
 \left(q_N^{m+1}+\sum_{r=1}^m\delta_{I,r,N}q_N^r\right).
\end{equation}
For an identically distributed triangular array, with errors \(\delta_{N,r}\), the right-hand side is
\[
 CN^p\left(q_N^{m+1}+\sum_{r=1}^m\delta_{N,r}q_N^r\right).
\]
Thus a right-hand side of \(o(1)\) implies the ratio conclusion, and \(o(N)\) implies pressure convergence.
\end{corollary}

\begin{remark}[Scope of the approximate-matching bound]
The estimate in \eqref{eq:approx-moment-bound} comes from sequential
coordinate replacement and is deliberately stated with the uniform
influence \(q_N\).  It does not retain the aggregate \(O(N)\)
squared-influence gain in \cref{thm:bulk-universality}: that gain uses
exact cancellation of the matched Taylor terms along the simultaneous
product-mixture path.  A coordinate-weighted refinement for nonzero
moment defects would require a separate aggregate comparison.
\end{remark}

\begin{corollary}[Point-incoherent frames]\label[corollary]{cor:point-incoherent}
Let \(\overline\sigma_{N-1}\) denote normalized surface measure on
\(\mathbb S^{N-1}(1)\), and define
\[
 D_{N,L}=\left\{x\in\mathbb S^{N-1}(1):
 \norm{x}_\infty\le L\sqrt{\frac{\log N}{N}}\right\},
 \qquad
 \widehat\cO_{N,L}
 =\{(v_1,\ldots,v_n,x)\in\ON:x\in D_{N,L}\}.
\]
Let \(\widehat\cZ_{N,L}^J\) be defined by replacing \(\cO_{N,L}\) with
\(\widehat\cO_{N,L}\) in \eqref{eq:Z-def}.  Every conclusion of
\cref{thm:bulk-universality} remains valid for
\(\widehat\cZ_{N,L}^J\), with the same quantitative bound.  For Gaussian
disorder, \eqref{eq:Gaussian-Z-exact} remains valid with
\(\nu_N(\cO_{N,L})\) replaced by
\(\overline\sigma_{N-1}(D_{N,L})\).  In particular, for \(L>L_0\) the
set \(D_{N,L}\) has spherical measure tending to one, while for every
fixed \(L>0\) the Gaussian variational pressure limit is unchanged.
\end{corollary}

\begin{proof}
The coefficient estimates use only \(\norm{x}_\infty\le r_N\).
Indeed, \((b_{1,I},\ldots,b_{n,I})\) is the coordinate vector of
\(P_{x^\perp}\nabla m_I(x)\), and \(C_I\) represents a compressed ambient
operator whose Hilbert--Schmidt and operator norms do not depend on the
chosen orthonormal basis of \(x^\perp\).  Thus the coordinatewise and
aggregate influence estimates hold uniformly on
\(\widehat\cO_{N,L}\), even though the finite-\(\eps\) product gradient
mollifier itself depends on that basis.  The replacement proof is
framewise and may therefore be integrated over the larger frame set.
Gaussian isotropy gives the stated finite-\(N\) formula.  The two
geometric conclusions follow from
\cref{lem:Haar-incoherent}: the full-measure statement uses \(L>L_0\),
whereas zero exponential cost holds for every fixed \(L>0\).
\end{proof}

\begin{theorem}[Uniform bulk comparison for growing point cutoffs]
\label[theorem]{thm:growing-point-cutoff}
Fix \(p\ge3\), \(m\ge2\), \(\eps>0\), \(V\in\cV_m\),
\(f\in\cF_m\) or \(f=\ell_\eta\) for fixed \(\eta>0\),
\(K<\infty\).
For a deterministic sequence \(r_N\) satisfying
\begin{equation}\label{eq:growing-cutoff-range}
 N^{-1/2}<r_N\le1,
\end{equation}
define
\begin{align}
 D_{N,r_N}
 &:=\left\{x\in\mathbb S^{N-1}(1):
             \norm{x}_\infty\le r_N\right\},\label{eq:growing-point-set}\\
 \widehat\cO_{N,r_N}
 &:=\left\{(v_1,\ldots,v_n,x)\in\ON:
             x\in D_{N,r_N}\right\},\label{eq:growing-frame-set}\\
 q_N^{\rm gr}&:=\sqrt N\,r_N^{p-1}.\label{eq:growing-qN}
\end{align}
Let
\begin{align}
 \widehat\cZ_{N,r_N}^J(V,\eps,f)
 :=\omega_N\int_{\widehat\cO_{N,r_N}}
 &\exp\biggl\{-NV\bigl(h_J(O)\bigr)
 +\sum_{a=1}^n\log\kappa_\eps\bigl(g_{a,J}(O)\bigr)\notag\\
 &\hspace{31mm}+\Tr f\bigl(A_J(O)\bigr)\biggr\}
 \,\nu_N(\dd O),\label{eq:growing-point-Z}
\end{align}
and put
\[
 \widehat\Phi_{N,r_N}^J(V,\eps,f)
 :=\frac1N\log\E\widehat\cZ_{N,r_N}^J(V,\eps,f).
\]
Suppose that the independent tensor coordinates \(J_{I,N}\) satisfy
\eqref{eq:array-assumptions}.  There are constants
\[
 q_\ast=q_\ast(p,m,K,\eps,V,f)>0,
 \qquad
 C=C(p,m,K,\eps,V,f)<\infty,
\]
independent of \(N\) and \(r_N\), such that, whenever
\(q_N^{\rm gr}\le q_\ast\),
\begin{equation}\label{eq:growing-point-comparison}
 \left|
 \log\frac{\E\widehat\cZ_{N,r_N}^J(V,\eps,f)}
              {\E\widehat\cZ_{N,r_N}^G(V,\eps,f)}
 \right|
 \le CN(q_N^{\rm gr})^{m-1}.
\end{equation}
Consequently,
\begin{align}
 q_N^{\rm gr}\longrightarrow0
 &\quad\Longrightarrow\quad
 \widehat\Phi_{N,r_N}^J(V,\eps,f)
 -\widehat\Phi_{N,r_N}^G(V,\eps,f)
 \longrightarrow0,\label{eq:growing-pressure-conclusion}\\
 N(q_N^{\rm gr})^{m-1}\longrightarrow0
 &\quad\Longrightarrow\quad
 \frac{\E\widehat\cZ_{N,r_N}^J(V,\eps,f)}
      {\E\widehat\cZ_{N,r_N}^G(V,\eps,f)}
 \longrightarrow1.\label{eq:growing-ratio-conclusion}
\end{align}
In the second case the pressure difference is \(o(N^{-1})\).

There are universal constants \(c,C_0>0\) such that
\begin{equation}\label{eq:growing-Haar-tail}
 1-\overline\sigma_{N-1}(D_{N,r_N})
 \le C_0N e^{-cNr_N^2}.
\end{equation}
In particular,
\begin{equation}\label{eq:growing-Haar-full-condition}
 cNr_N^2-\log N\longrightarrow+\infty
\end{equation}
implies
\[
 \overline\sigma_{N-1}(D_{N,r_N})\longrightarrow1.
\]
The weaker condition
\begin{equation}\label{eq:growing-Haar-zero-cost-condition}
 \sqrt N\,r_N\longrightarrow\infty
\end{equation}
already implies
\begin{equation}\label{eq:growing-Haar-zero-cost}
 \frac1N\log\overline\sigma_{N-1}(D_{N,r_N})
 \longrightarrow0.
\end{equation}
If both \eqref{eq:growing-Haar-zero-cost-condition} and
\(q_N^{\rm gr}\to0\) hold, then
\begin{align}
 \lim_{N\to\infty}\widehat\Phi_{N,r_N}^J(V,\eps,f)
 ={}&\frac12\log(2\pi e)+\log a_{\eps,p}\notag\\
 &+\sup_{u\in\R}
 \left\{-\frac{u^2}{2}-V(u)+\ell_{p,f}(u)\right\}.
 \label{eq:growing-Gaussian-variational}
\end{align}
\end{theorem}

\begin{proof}
The strict lower bound in \eqref{eq:growing-cutoff-range} ensures that
\(D_{N,r_N}\) has positive spherical measure: it contains a
neighborhood of \(N^{-1/2}(1,\ldots,1)\).  Thus both expectations in
\eqref{eq:growing-point-comparison} are finite and strictly positive.
Also,
\begin{equation}\label{eq:growing-sqrtn-r}
 \sqrt N\,r_N>1.
\end{equation}

We first track the cutoff uniformly through the influence proof.  Fix
\(O=(v_1,\ldots,v_n,x)\in\widehat\cO_{N,r_N}\).  The proof of
\cref{lem:coefficient-bounds} uses only
\(\norm{x}_\infty\le r_N\), not coordinate bounds on the tangent
vectors.  Hence
\begin{align}
 \abs{e_I(O)}
 &\le N^{-1/2}r_N^p,\label{eq:growing-e-bound}\\
 \sum_{a=1}^n\abs{b_{a,I}(O)}^2
 &\le C_pr_N^{2p-2},\label{eq:growing-b-bound}\\
 \norm{C_I(O)}_{\HS}
 &\le C_p\left(N^{-1/2}r_N^{p-2}+r_N^p\right),
 \label{eq:growing-C-HS}\\
 \norm{C_I(O)}_{\op}
 &\le C_pN^{-1/2}r_N^{p-2}.\label{eq:growing-C-op}
\end{align}
Since \(r_N\le1\) and \eqref{eq:growing-sqrtn-r} holds,
\begin{align}
 \sqrt N\,\norm{C_I}_{\HS}
 &\le C_pq_N^{\rm gr},\label{eq:growing-C-first}\\
 \norm{C_I}_{\HS}+\norm{C_I}_{\op}
 &\le C_pq_N^{\rm gr}.\label{eq:growing-C-higher}
\end{align}

Let \(\Psi_O\) be as in \eqref{eq:Psi-def}.  For every
\(1\le k\le m+1\), \eqref{eq:growing-e-bound} gives
\[
 N\abs{e_I}^k
 \le N^{1-k/2}r_N^{pk}
 =(q_N^{\rm gr})^kN^{1-k}r_N^k
 \le(q_N^{\rm gr})^k.
\]
Moreover,
\[
 \sum_a\abs{b_{a,I}}
 \le\sqrt N\left(\sum_a b_{a,I}^2\right)^{1/2}
 \le C_pq_N^{\rm gr},
\]
while, for \(k\ge2\),
\[
 \sum_a\abs{b_{a,I}}^k
 \le\left(\sum_a b_{a,I}^2\right)^{k/2}
 \le C_{p,m}r_N^{(p-1)k}
 \le C_{p,m}(q_N^{\rm gr})^k.
\]
The bounded derivatives of \(V\) and \(\log\kappa_\eps\), together
with the applicable one of
\cref{lem:trace-derivatives,lem:uncut-logdet-derivatives} and
\eqref{eq:growing-C-first}--\eqref{eq:growing-C-higher}, give
\begin{equation}\label{eq:growing-log-derivatives}
 \abs{\partial_I^k\Psi_O(J)}
 \le C_k(q_N^{\rm gr})^k,
 \qquad 1\le k\le m+1,
\end{equation}
with constants independent of \(N,r_N,O,I,J\).

Define, as before,
\[
 \mathfrak d_I(O)
 :=Ne_I(O)^2+\sum_a b_{a,I}(O)^2+\norm{C_I(O)}_{\HS}^2,
 \qquad
 \alpha_I(J):=\partial_I\Psi_O(J).
\]
The exact covariance identities
\eqref{eq:coefficient-orthogonality} and
\eqref{eq:C-covariance-identity} do not involve the cutoff and imply
\begin{equation}\label{eq:growing-aggregate-budget}
 \sum_I\mathfrak d_I\le C_pN.
\end{equation}
On the other hand,
\eqref{eq:growing-e-bound}--\eqref{eq:growing-C-HS} and
\eqref{eq:growing-sqrtn-r} imply
\begin{equation}\label{eq:growing-max-budget}
 \max_I\mathfrak d_I\le C_p(q_N^{\rm gr})^2.
\end{equation}
For example,
\[
 \frac{N^{-1}r_N^{2p-4}}{(q_N^{\rm gr})^2}
 =\frac1{N^2r_N^2}\le1,
 \qquad
 \frac{r_N^{2p}}{(q_N^{\rm gr})^2}
 =\frac{r_N^2}{N}\le1.
\]
The same covariance calculation as in
\cref{lem:aggregate-jet-budget} gives
\begin{equation}\label{eq:growing-alpha-budget}
 \max_I\abs{\alpha_I(J)}\le Cq_N^{\rm gr},
 \qquad
 \sum_I\abs{\alpha_I(J)}^2\le CN.
\end{equation}
For \(2\le k\le m+1\), the energy, gradient, and spectral terms
satisfy
\begin{equation}\label{eq:growing-higher-log-budget}
 \abs{\partial_I^k\Psi_O(J)}
 \le C_k\mathfrak d_I(q_N^{\rm gr})^{k-2}.
\end{equation}
For example, the energy term is bounded because
\[
 N\abs{e_I}^k\le(Ne_I^2)(q_N^{\rm gr})^{k-2}.
\]
The gradient and spectral terms use, respectively,
\(\max_a\abs{b_{a,I}}\le Cq_N^{\rm gr}\) and
\(\norm{C_I}_{\op}\le Cq_N^{\rm gr}\).

Put \(F_O=e^{\Psi_O}\).  Fa\`a di Bruno's formula and
\eqref{eq:growing-alpha-budget}--\eqref{eq:growing-higher-log-budget}
give, for \(2\le k\le m+1\),
\begin{align}
 \abs{\partial_I^kF_O(J)}
 &\le C_kF_O(J)
 \left(\abs{\alpha_I(J)}^k+
       \mathfrak d_I(q_N^{\rm gr})^{k-2}\right),
 \label{eq:growing-pointwise-F}\\
 \sum_I\abs{\partial_I^kF_O(J)}
 &\le C_kN(q_N^{\rm gr})^{k-2}F_O(J).
 \label{eq:growing-summed-F}
\end{align}

We finally check that the conditional replacement constants are
uniform.  In the uncut case,
\[
 \det(A_J^2+\eta^2I_n)^{1/2}
 \le\bigl(\eta+C_p\sqrt N\,\norm J_2\bigr)^n
\]
uniformly in the frame: this follows from
\eqref{eq:C-covariance-identity} and Cauchy--Schwarz.  The uniform
subexponential bound makes this polynomial envelope, and all
derivatives used below, integrable under every product-mixture law.
Let
\[
 \mu_{I,t}=(1-t)\mu_I+t\gamma,\qquad
 \lambda_t=\bigotimes_I\mu_{I,t},\qquad
 A_O(t)=\int F_O(z)\lambda_t(\dd z),
\]
and put \(s=m+1\).  The proof of
\cref{lem:conditional-aggregate-remainder} uses only the uniform
exponential moment, the cutoff-uniform estimates
\eqref{eq:growing-aggregate-budget}--\eqref{eq:growing-higher-log-budget},
and the small exponential tilt
\[
 e^{-Cq_N^{\rm gr}\abs u}F_O(z_{-I},0)
 \le F_O(z_{-I},u)
 \le e^{Cq_N^{\rm gr}\abs u}F_O(z_{-I},0).
\]
Choose \(K_\ast\), depending only on \(K\), so that
\[
 \sup_{I,t}\int e^{\abs u/K_\ast}\mu_{I,t}(\dd u)
 +\int e^{\abs u/K_\ast}\gamma(\dd u)<\infty.
\]
Then choose \(q_\ast\le1\), depending only on the displayed theorem
parameters, so that \(Cq_\ast\le(2K_\ast)^{-1}\).  The tilt is then
absorbed by the common exponential moment.  Repeating the Taylor
expansion through order
\(s-1=m\), all lower terms cancel by moment matching and
\eqref{eq:growing-pointwise-F}--\eqref{eq:growing-summed-F} give
\[
 \abs{A_O'(t)}
 \le CN(q_N^{\rm gr})^{s-2}A_O(t)
 =CN(q_N^{\rm gr})^{m-1}A_O(t).
\]
Every constant is independent of \(r_N\).  Integrating over
\(t\in[0,1]\), then over
\(\widehat\cO_{N,r_N}\), proves
\eqref{eq:growing-point-comparison}.  The pressure and ratio
conclusions follow.

For the geometric assertion, \eqref{eq:Haar-entry-tail} and a union
bound over the \(N\) coordinates of the last column give
\eqref{eq:growing-Haar-tail}.  Gaussian isotropy makes the expected
integrand independent of \(O\), so \eqref{eq:Gaussian-Z-exact}
remains valid with \(\nu_N(\cO_{N,L})\) replaced by
\(\overline\sigma_{N-1}(D_{N,r_N})\).

It remains to prove the weaker exponential-cost statement.  Put
\(c_N=\sqrt N\,r_N\), let \(Z_1,\ldots,Z_N\) be independent standard
Gaussians, and set
\[
 A_N=\left\{\max_i|Z_i|\le\frac{c_N}{2}\right\},
 \qquad B_N=\left\{\|Z\|\ge\frac{\sqrt N}{2}\right\}.
\]
On \(A_N\cap B_N\), the point \(Z/\|Z\|\) belongs to \(D_{N,r_N}\).
Writing \(q_N^\circ=\Pp\{|G|>c_N/2\}\), condition
\eqref{eq:growing-Haar-zero-cost-condition} gives \(q_N^\circ\to0\)
and
\[
 \Pp(A_N)=(1-q_N^\circ)^N
 \ge \exp\{-2Nq_N^\circ\}=\exp\{-o(N)\}.
\]
On the other hand, \(\Pp(B_N^c)\le e^{-c_1N}\).  Hence, for all large
\(N\), \(\Pp(A_N\cap B_N)\ge\frac12\Pp(A_N)\), which proves
\eqref{eq:growing-Haar-zero-cost}.  The stronger condition
\eqref{eq:growing-Haar-full-condition} gives convergence of the
spherical measure to one directly from \eqref{eq:growing-Haar-tail}.
The proof of \cref{thm:gaussian-variational} now applies verbatim to the
point-slice version of \eqref{eq:Gaussian-Z-exact}: use
\eqref{eq:growing-Haar-zero-cost} for its geometric factor, and retain
\cref{lem:GOE-linear-stat,lem:sphere-volume} and the same Laplace
argument for all other factors.  Combining the resulting Gaussian
limit with \eqref{eq:growing-pressure-conclusion} proves
\eqref{eq:growing-Gaussian-variational}.
\end{proof}

\begin{corollary}[Explicit growing-cutoff regimes]
\label[corollary]{cor:growing-cutoff-regimes}
In the setting of \cref{thm:growing-point-cutoff}, suppose
\[
 r_N=L_N\sqrt{\frac{\log N}{N}},
 \qquad L_N\longrightarrow\infty,
 \qquad r_N\le1.
\]
Then the cutoff is Haar-full.  Pressure universality holds if
\begin{equation}\label{eq:growing-L-pressure}
 L_N=o\left(
 \frac{N^{(p-2)/(2(p-1))}}{\sqrt{\log N}}
 \right),
\end{equation}
and the expectation-ratio conclusion holds if
\begin{equation}\label{eq:growing-L-ratio}
 L_N=o\left(
 \frac{
 N^{\frac{p-2}{2(p-1)}
 -\frac1{(m-1)(p-1)}}}
 {\sqrt{\log N}}
 \right).
\end{equation}
The latter permits a growing \(L_N\) precisely when
\((m-1)(p-2)>2\).

Equivalently, for \(r_N=N^{-a}\), the cutoff is Haar-full and pressure
universal whenever
\begin{equation}\label{eq:power-cutoff-pressure}
 \frac1{2(p-1)}<a<\frac12.
\end{equation}
It is Haar-full and satisfies the ratio conclusion whenever
\begin{equation}\label{eq:power-cutoff-ratio}
 \frac{m+1}{2(m-1)(p-1)}<a<\frac12.
\end{equation}
\end{corollary}

\begin{proof}
For the logarithmic parametrization,
\[
 q_N^{\rm gr}
 =L_N^{p-1}N^{-(p-2)/2}
  (\log N)^{(p-1)/2}.
\]
Thus \eqref{eq:growing-L-pressure} is equivalent to
\(q_N^{\rm gr}\to0\), whereas \eqref{eq:growing-L-ratio} is equivalent
to \(N(q_N^{\rm gr})^{m-1}\to0\).  Since \(L_N\to\infty\),
\(Nr_N^2/\log N=L_N^2\to\infty\), so the cutoff is Haar-full.
For \(r_N=N^{-a}\),
\[
 q_N^{\rm gr}=N^{1/2-a(p-1)}.
\]
Solving \(q_N^{\rm gr}\to0\),
\(N(q_N^{\rm gr})^{m-1}\to0\), and \(a<1/2\) gives the displayed
conditions.
\end{proof}

\begin{theorem}[Bulk universality below the \(\psi_1\) scale]
\label[theorem]{thm:weibull-bulk}
Fix \(0<\alpha<1\) with
\begin{equation}\label{eq:weibull-degree-window}
 \alpha(p-2)>2.
\end{equation}
Retain the parameters and independence assumptions of
\cref{thm:bulk-universality}, require \(f\in\cF_m\), and replace the
exponential-moment hypothesis
in \eqref{eq:array-assumptions} by
\begin{equation}\label{eq:uniform-weibull-upper}
 \sup_{N,I}\Pp(\abs{J_{I,N}}>t)
 \le C_0e^{-c_0t^\alpha},
 \qquad t\ge0,
\end{equation}
and retain exact Gaussian moment matching through order \(m\ge2\).  Choose
\[
 \frac1\alpha<\gamma<\frac{p-2}{2}.
\]
Then, for all sufficiently large \(N\),
\begin{equation}\label{eq:weibull-bulk-bound}
 \left|\log\frac{\E\cZ_{N,L}^J(V,\eps,f)}
 {\E\cZ_{N,L}^G(V,\eps,f)}\right|
 \le CNq_N^{m-1}+Ce^{-cN^{\alpha\gamma}}.
\end{equation}
The pressure and ratio conclusions of
\cref{thm:bulk-universality} therefore remain valid.  The same statement
holds for the point-incoherent functional of
\cref{cor:point-incoherent}.  In particular, since
\eqref{eq:weibull-degree-window} forces \(p\ge5\), mean and variance
matching already give the ratio conclusion.
No conclusion for the uncut weight \(\ell_\eta\) is asserted under
\eqref{eq:uniform-weibull-upper}.
\end{theorem}

\begin{corollary}[Standard symmetric-coordinate extension]
\label[corollary]{cor:symmetric-bulk}
Replace the ordered Hamiltonian by \(H_{N,\mathrm{sym}}^\xi\) in
\eqref{eq:symmetric-coordinate-model}.  If the independent coordinates
\(\xi_{\mathfrak a}\) satisfy the uniform tail and moment hypotheses of
\cref{thm:bulk-universality}, then all conclusions of
\cref{thm:bulk-universality,cor:moment-thresholds,cor:point-incoherent}
hold, with constants allowed to depend additionally on \(p!\).  Under
the hypotheses of \cref{thm:weibull-bulk}, its conclusions hold as well.
The Gaussian variational formulas below are unchanged.

The approximate-matching statement also holds in its natural
orbit-indexed form: in \eqref{eq:approx-moment-bound}, replace the sum
over \(I\in[N]^p\) by a sum over \(\mathfrak a\), and define
\[
 \delta_{\mathfrak a,r,N}
 =\abs{\E\xi_{\mathfrak a,N}^r-\E G^r}.
\]
\end{corollary}

\begin{proof}
Fix one representative \(I_{\mathfrak a}\) of each orbit.  A monomial
and every derivative of it are invariant under permuting the indices, so
the coefficient of \(\xi_{\mathfrak a}\) in any component of the
energy--gradient--Hessian jet is \(\sqrt{d_{\mathfrak a}}\) times the
coefficient of \(J_{I_{\mathfrak a}}\) in the ordered model.  Hence, for
any scalar jet coefficient \(c_I\) occurring in the proof,
\[
 \sum_{\mathfrak a}
 \bigl|\sqrt{d_{\mathfrak a}}\,c_{I_{\mathfrak a}}\bigr|^2
 =\sum_{I\in[N]^p}|c_I|^2,
 \qquad
 \max_{\mathfrak a}
 \bigl|\sqrt{d_{\mathfrak a}}\,c_{I_{\mathfrak a}}\bigr|
 \le\sqrt{p!}\max_I|c_I|.
\]
The same identities hold componentwise for the vector and matrix
coefficients of the jet.  Thus the exact aggregate \(O(N)\) influence
budget is unchanged, and each coordinatewise bound changes only by a
\(p\)-dependent constant.  The product-mixture and truncation arguments
therefore apply verbatim.  Finally,
\[
 \sum_{\mathfrak a}d_{\mathfrak a}
 \sigma^{\mathfrak a}\tau^{\mathfrak a}
 =\sum_{I\in[N]^p}\sigma_I\tau_I,
\]
so the Gaussian covariance, and hence every Gaussian formula, is the
same.
\end{proof}

\begin{remark}[What the theorem does and does not prove]\label[remark]{rem:scope-positive}
\Cref{thm:bulk-universality} is a theorem about a mollified Kac--Rice partition function on incoherent frames.  It does not by itself allow \(\eps\downarrow0\), replace \(f\) by \(\log\abs{\cdot}\), or restore coordinate-localized points uniformly in \(N\).  The counterexample in \cref{thm:counterexample} shows that the last step is false under moment matching alone.
\end{remark}

\begin{remark}[Hypothesis map for the positive results]
\label[remark]{rem:assumption-map}
The regularized comparison uses only independence, a uniform tail bound,
and moment matching for the tensor coordinates.  The coordinates need
not be identically distributed, and no density, anti-concentration, or
Morse hypothesis is imposed.  Point incoherence is sufficient by
\cref{cor:point-incoherent}; its geometric slice has zero exponential
cost for every fixed \(L>0\).  Exact-count upper bounds require only the
almost-sure exponential critical-count ceiling \({\rm(CC)}_p\);
algebraic genericity is one sufficient route, and joint absolute
continuity is in turn sufficient for algebraic genericity.  Kac--Rice
de-regularization separately requires a gradient density and a compatible
conditional slice.  The remaining exponential work
is measured by the one-sided defects in
\cref{def:dereg-defects,prop:primitive-dereg} and by the necessary and
sufficient localized branch in
\eqref{eq:exact-universality-iff-localized-defect}; finite matching does
not make those quantities vanish.
For the sparse-profile branch, the exact probabilistic input is only
the uniform tail-rate condition \(({\rm PT})_p\); the profile MGF
\(({\rm PSG})_p\), the full joint sharp sub-Gaussian bound, and
independent sharp sub-Gaussian coordinates are successively stronger
sufficient conditions.
\end{remark}

\subsection{Gaussian variational formula}

Write \(\rhoSC\) for the semicircle law on \([-\sqrt2,\sqrt2]\):
\begin{equation}\label{eq:semicircle}
 \rhoSC(\dd\lambda)=\frac1\pi\sqrt{2-\lambda^2}\,
 \1_{\{\abs\lambda\le\sqrt2\}}\dd\lambda.
\end{equation}
For \(f\in\cF_m\) or \(f=\ell_\eta\), define
\begin{equation}\label{eq:ell-pf}
 \ell_{p,f}(u)=\int
 f\bigl(\sqrt{2p(p-1)}\lambda-pu\bigr)\,\rhoSC(\dd\lambda),
\end{equation}
and
\begin{equation}\label{eq:a-eps-p}
 a_{\eps,p}=\E\kappa_\eps(\sqrt p\,G),
 \qquad G\sim\cN(0,1).
\end{equation}

\begin{theorem}[Gaussian regularized variational formula]\label[theorem]{thm:gaussian-variational}
There exists \(L_0<\infty\) such that, for every fixed \(L>L_0\),
\(\eps>0\), continuous \(V\) bounded below, and either
\(f\in\cF_m\) or \(f=\ell_\eta\) with fixed \(\eta>0\),
\begin{align}
 \lim_{N\to\infty}\Phi_{N,L}^G(V,\eps,f)
 ={}&\frac12\log(2\pi e)+\log a_{\eps,p}\notag\\
 &+\sup_{u\in\R}\left\{-\frac{u^2}{2}-V(u)+\ell_{p,f}(u)\right\}.
 \label{eq:Gaussian-variational}
\end{align}
If in addition \(V\in\cV_m\), the same limit holds with \(G\) replaced
by \(J\) under \cref{thm:bulk-universality}.  For
\(f\in\cF_m\), the same conclusion also follows under
\cref{thm:weibull-bulk}; no uncut \(\ell_\eta\) extension of that
Weibull truncation theorem is claimed.
\end{theorem}

\subsection{The determinant limit and the Gaussian pointwise complexity}

Set
\begin{equation}\label{eq:Einf}
 E_\infty=2\sqrt{\frac{p-1}{p}}.
\end{equation}
Define
\begin{equation}\label{eq:Ip-def}
 \cI_p(u)=
 \begin{cases}
 0,&\abs u\le E_\infty,\\[1mm]
 \displaystyle \frac{2}{E_\infty^2}
 \int_{E_\infty}^{\abs u}\sqrt{z^2-E_\infty^2}\,\dd z,
 &\abs u>E_\infty.
 \end{cases}
\end{equation}

\begin{theorem}[Universal bulk complexity potential]\label[theorem]{thm:theta}
Let
\begin{equation}\label{eq:theta-logpotential}
 \theta_p(u)=\frac12-\frac12\log p-\frac{u^2}{2}
 +\int\log\abs{\sqrt{2p(p-1)}\lambda-pu}\,\rhoSC(\dd\lambda).
\end{equation}
Then
\begin{equation}\label{eq:theta-explicit}
 \theta_p(u)=\frac12\log(p-1)
 -\frac{p-2}{4(p-1)}u^2-\cI_p(u).
\end{equation}
For every continuous \(V\) bounded below,
\begin{equation}\label{eq:bulk-pressure-formal}
 \Phi_p^{\rm bulk}(V)
 :=\sup_{u\in\R}\{\theta_p(u)-V(u)\}
\end{equation}
is the iterated determinant/delta de-regularization of the Gaussian pressure in \cref{thm:gaussian-variational}; see \cref{thm:dereg-variational}.  Moreover,
\begin{equation}\label{eq:dual-theta}
 \theta_p(u)=\inf_{V\in\cV_\infty}
 \{\Phi_p^{\rm bulk}(V)+V(u)\},
\end{equation}
where \(\cV_\infty=\bigcap_{m\ge2}\cV_m\).
\end{theorem}

\begin{lemma}[Expected GOE empirical-measure LDP]\label[lemma]{lem:GOE-mean-LDP}
Let \(\rho_N(x)\dd x\) be the normalized expected empirical measure of the \(N\)-dimensional GOE in the normalization used in \cite{ABAC}.  These probability measures satisfy a large-deviation principle at speed \(N\) with good rate function
\begin{equation}\label{eq:J-edge}
 J_{\rm edge}(x)=
 \begin{cases}
 0,&\abs x\le\sqrt2,\\[1mm]
 \displaystyle\int_{\sqrt2}^{\abs x}\sqrt{z^2-2}\dd z,
 &\abs x>\sqrt2.
 \end{cases}
\end{equation}
Consequently, if \(W\) is continuous and \(W(x)\le C-c_0x^2\) for some \(c_0>0\), then
\begin{equation}\label{eq:GOE-Varadhan}
 \lim_{N\to\infty}\frac1N\log
 \int_\R e^{NW(x)}\rho_N(x)\dd x
 =\sup_{x\in\R}\{W(x)-J_{\rm edge}(x)\}.
\end{equation}
For every compact interval \(K\) with nonempty interior, the same conclusion holds with both the integral and supremum restricted to \(K\).
\end{lemma}

\begin{proof}
Let \(\lambda_1\le\cdots\le\lambda_N\) be the GOE eigenvalues.  In the
normalization of \cite[(2.6)--(2.7)]{ABAC},
\(\operatorname{Var}(M_{ij})=(1+\delta_{ij})/(2N)\), so
\cite[Theorem~A.1]{ABAC} applies with \(\sigma=2^{-1/2}\).  Consequently
\(\lambda_N\) satisfies an LDP at speed \(N\) with rate
\[
 I_+(x)=
 \begin{cases}
 \displaystyle\int_{\sqrt2}^{x}\sqrt{z^2-2}\,\dd z,&x\ge\sqrt2,\\
 +\infty,&x<\sqrt2,
 \end{cases}
\]
and, by the symmetry \(M\overset d=-M\), \(-\lambda_1\) satisfies the
same LDP.  Moreover, semicircle convergence in probability, together
with boundedness of bounded-test-function integrals, gives
\(\rho_N=\E[N^{-1}\sum_i\delta_{\lambda_i}]\Rightarrow\rho_{\rm sc}\).

If an open set \(G\) meets \([-\sqrt2,\sqrt2]\), then the semicircle law
gives the LDP lower bound.  If \(x>\sqrt2\) belongs to \(G\), choose an
open interval \(D\) with \(x\in D\Subset G\); then
\[
 \rho_N(G)\ge \frac1N\Pp(\lambda_N\in D),
\]
and the extreme-eigenvalue lower bound applies.  The left tail is
identical.  For a closed set \(F\) meeting the semicircle support, the
LDP upper bound is trivial.  Otherwise write
\(F_-=F\cap(-\infty,-\sqrt2)\) and
\(F_+=F\cap(\sqrt2,\infty)\).  Omitting any empty terms,
\[
 \rho_N(F)\le
 \Pp(\lambda_1\le\sup F_-)+
 \Pp(\lambda_N\ge\inf F_+),
\]
which gives the required upper bound.  For \(R>\sqrt2\), these estimates
also show
\[
 \limsup_{N\to\infty}\frac1N\log\rho_N([-R,R]^c)
 \le-J_{\rm edge}(R)\longrightarrow-\infty,
\]
so the family is exponentially tight and \(J_{\rm edge}\) is good.

Since \(W\) is continuous, bounded above, and has a negative quadratic
upper envelope, the extended Varadhan lemma
\cite[Theorem~4.3.1 and Exercise~4.3.11]{DemboZeitouni} proves
\eqref{eq:GOE-Varadhan}.  On a nondegenerate compact interval \(K\), the
upper bound is applied to the closed set \(K\) and the lower bound to
\(\operatorname{int}K\).  Because \(\operatorname{int}K\) is dense in
\(K\) and \(W-J_{\rm edge}\) is continuous, the two restricted suprema
coincide, including when a maximizer is an endpoint.
\end{proof}

\begin{theorem}[Exact Gaussian weighted and local complexity]\label[theorem]{thm:Gaussian-local-complexity}
For every continuous \(V:\R\to\R\) bounded below, the Gaussian pure spherical \(p\)-spin field satisfies
\begin{equation}\label{eq:Gaussian-weighted-complexity}
 \lim_{N\to\infty}\frac1N\log
 \E\sum_{x:\nabla_SH_N^G(\sqrt N x)=0}
 e^{-NV(h_N^G(x))}
 =\sup_{u\in\R}\{\theta_p(u)-V(u)\}.
\end{equation}
In particular, for every compact interval \(B\subset\R\) with nonempty interior,
\begin{equation}\label{eq:Gaussian-local-complexity}
 \lim_{N\to\infty}\frac1N\log\E\Crt_N^G(B)
 =\sup_{u\in B}\theta_p(u).
\end{equation}
\end{theorem}

\begin{proof}
Put
\[
 s=\sqrt{\frac{p}{2(p-1)}},
 \qquad c=\frac{p-2}{2p}.
\]
The exact identity \cite[Theorem~2.2]{ABAC} is an equality of finite
measures on the energy axis.  Applying it first to simple weights and
then passing to general nonnegative weights by monotone approximation
gives
\begin{align}\label{eq:ABAC-weighted-identity}
 &\E\sum_{x:\nabla_SH_N^G(\sqrt N x)=0}e^{-NV(h_N^G(x))}\notag\\
 &\qquad=2N\sqrt{\frac2p}(p-1)^{N/2}
 \int_\R e^{-N\{cx^2+V(x/s)\}}\rho_N(x)\dd x.
\end{align}
Here \(\rho_N\) is the expected empirical density of the \(N\times N\)
GOE in \cite{ABAC}; the shift from the \((N-1)\times(N-1)\) conditional
Hessian in \cref{prop:Gaussian-jet} is part of the exact determinant
identity.
Apply \cref{lem:GOE-mean-LDP} with \(W(x)=-cx^2-V(x/s)\).  Since \(cs^2=(p-2)/(4(p-1))\) and \(J_{\rm edge}(su)=\cI_p(u)\), changing variables \(x=su\) proves \eqref{eq:Gaussian-weighted-complexity}.  The same exact identity restricted to \(sB\), followed by the compact-interval part of \cref{lem:GOE-mean-LDP}, proves \eqref{eq:Gaussian-local-complexity}.  We emphasize that the present \(\theta_p\) is the even pointwise potential; the function denoted \(\theta_p\) in \cite{ABAC} is its cumulative counterpart.
\end{proof}

\subsection{Unregularized non-universality under finite moment matching}

The next theorem gives the central quantitative obstruction.  Matching
through order \(2p\) is emphasized because it is the natural threshold for
thermodynamic universality of the spherical model \cite{SawhneySellke}.

\begin{theorem}[Smooth sub-Gaussian moment-matching counterexample]\label[theorem]{thm:counterexample}
Fix \(p\ge3\) and \(\sigma>1\).  There exists a symmetric probability law \(\mu=\mu_{p,\sigma}\) on \(\R\) with the following properties.
\begin{enumerate}[label=(\roman*),leftmargin=2em]
\item \(\mu\) has a strictly positive \(C^\infty\) density and is sub-Gaussian.
\item If \(\xi\sim\mu\) and \(G\sim\cN(0,1)\), then
\begin{equation}\label{eq:counter-match}
 \E\xi^r=\E G^r,
 \qquad r=1,\ldots,2p.
\end{equation}
\item There is a compact interval \(B\subset(0,\infty)\) such that
\begin{equation}\label{eq:counter-strict}
 \liminf_{N\to\infty}\frac1N\log\E_\mu\Crt_N^J(B)
 >\lim_{N\to\infty}\frac1N\log\E\Crt_N^G(B).
\end{equation}
Moreover, the same interval satisfies
\begin{align}
 \liminf_{N\to\infty}\frac1N\log\E_\mu\Crt_{N,\max}^J(B)
 &>\limsup_{N\to\infty}\frac1N\log\E\Crt_{N,\max}^G(B),
 \label{eq:counter-max}\\
 \liminf_{N\to\infty}\frac1N\log\E_\mu\Crt_{N,\min}^J(-B)
 &>\limsup_{N\to\infty}\frac1N\log\E\Crt_{N,\min}^G(-B).
 \label{eq:counter-min}
\end{align}
\end{enumerate}
In particular, sub-Gaussianity together with Gaussian moment matching
through order \(2p\) does not force the Gaussian rate of unregularized
critical-point counts in a compact energy window when no spatial
localization restriction is imposed.
\end{theorem}

\begin{theorem}[Bounded-disorder mesoscopic-block counterexample]
\label[theorem]{thm:bounded-block-counterexample}
Fix \(p\ge3\) and an integer \(r\ge2\).  There is a symmetric,
compactly supported probability law \(\mu\) with a \(C^\infty\) density
such that
\begin{enumerate}[label=(\roman*),leftmargin=2em]
\item \(\mu\) matches the standard Gaussian moments through order \(r\)
(hence is centered and has variance one);
\item for i.i.d.\ ordered-tensor entries with common law \(\mu\), there
is a compact interval \(B\subset(0,\infty)\) for which
\begin{align}
 \liminf_{N\to\infty}\frac1N\log
 \E_\mu\Crt_{N,\max}^J(B)
 &>\limsup_{N\to\infty}\frac1N\log
 \E\Crt_{N,\max}^G(B),\label{eq:bounded-block-max}\\
 \liminf_{N\to\infty}\frac1N\log
 \E_\mu\Crt_N^J(B)
 &>\lim_{N\to\infty}\frac1N\log
 \E\Crt_N^G(B).\label{eq:bounded-block-total}
\end{align}
The reflected local-minimum inequality holds on \(-B\).
\end{enumerate}
In particular, for every fixed \(\varepsilon\in[0,1/2)\),
\begin{equation}\label{eq:bounded-cutoff-automatic}
 \max_{I\in[N]^p}|J_I|\le N^{1/2-\varepsilon}
 \qquad\text{almost surely for all sufficiently large \(N\)}.
\end{equation}
Thus an \(N^{1/2-\varepsilon}\) entrywise cutoff, or even an
 \(N\)-independent entry bound, does not restore universality of the
 unregularized annealed count.  The obstruction in this theorem is a
 coherent block involving on the order of \(N^{1/p}\) coordinates,
 rather than a single quadratic-scale entry.
\end{theorem}

\begin{theorem}[Weibull-tail moment-matching counterexample]\label[theorem]{thm:subexp-counterexample}
Fix \(p\ge3\) and \(\alpha\in(0,2)\).  There exists a symmetric probability law \(\mu=\mu_{p,\alpha}\) with a strictly positive \(C^\infty\) density such that:
\begin{enumerate}[label=(\roman*),leftmargin=2em]
\item \(\mu\) has Weibull tails of index \(\alpha\) and all moments finite; when \(\alpha\ge1\), it is subexponential in the sense of \cref{def:subexp};
\item \(\mu\) matches the standard Gaussian moments through order \(2p\);
\item for some compact interval \(B\subset(0,\infty)\),
\begin{equation}\label{eq:subexp-counter-strict}
 \liminf_{N\to\infty}\frac1N\log\E_\mu\Crt_N^J(B)
 \ge0>
 \lim_{N\to\infty}\frac1N\log\E\Crt_N^G(B).
\end{equation}
The same strict inequalities hold for local maxima, with the Gaussian
term interpreted via a \(\limsup\); after replacing \(B\) with \(-B\),
they also hold for local minima.
\end{enumerate}
Thus \(2p\)-moment matching does not imply universality of these unregularized
energy-window counts for any subquadratic Weibull tail.  In particular,
the failure occurs within the subexponential class for every
\(\alpha\in[1,2)\).
\end{theorem}

\begin{corollary}[No finite-moment determination of annealed complexity]
\label[corollary]{cor:no-finite-moment}
For every fixed integer \(r\ge2\) and every \(p\ge3\), there exist two
symmetric, centered, unit-variance sub-Gaussian laws \(\mu_1,\mu_2\), both
with strictly positive \(C^\infty\) densities and with identical moments
through order \(r\), and a compact interval \(B\subset(0,\infty)\), such
that
\[
 \liminf_{N\to\infty}\frac1N\log\E_{\mu_1}\Crt_N^J(B)
 >
 \lim_{N\to\infty}\frac1N\log\E_{\mu_2}\Crt_N^J(B).
\]
Thus no prescribed finite collection of moments forces the Gaussian
energy-window rate for the spatially unrestricted, unregularized annealed
count: the non-Gaussian liminf may be strictly larger.  The statement
does not assert existence of the non-Gaussian limit.
\end{corollary}

\begin{corollary}[Symmetric-model counterexamples]
\label[corollary]{cor:symmetric-exact}
\leavevmode\par\noindent
The conclusions of
\cref{thm:counterexample,thm:subexp-counterexample} and
\cref{cor:no-finite-moment} remain valid for the symmetric-coordinate model
\eqref{eq:symmetric-coordinate-model}, when the standardized independent
orbit coordinates have the laws constructed in those results.  The
Gaussian comparison field is unchanged in law.
\end{corollary}

\begin{proof}
Let \(\mathfrak a_0\) be the diagonal orbit of
\((1,\ldots,1)\).  Since \(d_{\mathfrak a_0}=1\),
\[
 h_{N,\mathrm{sym}}^\xi(x)
 =\frac{\xi_{\mathfrak a_0}}{\sqrt N}x_1^p
  +R_{N,\mathrm{sym}}(x),
\]
and the distinguished coordinate is independent of the remainder.  For
centered sub-Gaussian standardized orbit coordinates, the remainder's
canonical metric satisfies
\begin{align*}
 &\norm{R_{N,\mathrm{sym}}(x)
       -R_{N,\mathrm{sym}}(y)}_{\psi_2}^2\\
 &\qquad\le\frac{CK^2}{N}\sum_{\mathfrak a}
 d_{\mathfrak a}
 \bigl(x^{\mathfrak a}-y^{\mathfrak a}\bigr)^2
 =\frac{CK^2}{N}
 \norm{x^{\tensor p}-y^{\tensor p}}_2^2.
\end{align*}
Thus the proof and conclusion of \cref{lem:remainder-tight} apply, and
then so does \cref{thm:one-spike}; the diagonal tail cost is exactly the
one in the ordered proof.  For the Weibull construction, the model is
already in the standard symmetric normalization and its coordinates
have uniformly bounded \((2p+\varepsilon_0)\)-moments.  The ground-state
input used in the ordered proof makes the full-field supremum tight, and
\(\abs{\xi_{\mathfrak a_0}}/\sqrt N\to0\) in probability; the triangle
inequality therefore gives the required remainder tightness.
Finally, absolute continuity in the symmetric coefficient space gives
the Morse property by \cref{lem:Morse}.
\end{proof}

\begin{corollary}[Bounded mesoscopic obstruction in the symmetric model]
\label[corollary]{cor:symmetric-bounded-block}
Fix \(p\ge3\) and an integer \(r\ge2\).  There is a symmetric,
compactly supported probability law \(\mu\) with a \(C^\infty\)
density, matching the standard Gaussian moments through order \(r\),
and a compact interval \(B\subset(0,\infty)\) such that the strict
total-count, local-maximum, and reflected local-minimum conclusions of
\cref{thm:bounded-block-counterexample} hold for the standard
symmetric-coordinate model \eqref{eq:symmetric-coordinate-model}.
In particular, for every fixed \(\varepsilon\in[0,1/2)\), its
standardized independent orbit coordinates satisfy the
\(N^{1/2-\varepsilon}\) cutoff almost surely for all sufficiently large
\(N\).
\end{corollary}

The proof of \cref{thm:counterexample} is quantitative.  The law \(\mu\) contains a small \(\cN(0,\sigma^2)\) mixture component while its low-order moments are corrected by a compactly supported smooth component.  A diagonal entry \(J_{1\cdots1}\approx a\sqrt N\) then creates a local maximum near the coordinate pole at exponential cost \(a^2/(2\sigma^2)\), which is smaller than the Gaussian landscape cost for sufficiently large \(a\).

\Cref{thm:bounded-block-counterexample} uses a different mechanism.
It assigns polynomially small but \(N\)-independent mass to a remote
bounded bump and corrects finitely many moments.  Requiring all
\(k_N^p\asymp N\) entries of a block, with
\(k_N\asymp N^{1/p}\), to fall in that bump creates an order-one
rank-one energy-density profile at speed-\(N\) cost.

\begin{corollary}[Coexistence and symmetric-model robustness]
\label[corollary]{cor:coexistence}
Fix \(p\ge3\), \(\sigma>1\), and let \(\mu_{p,\sigma}\) be the law from
\cref{thm:counterexample}.  For every fixed \(L,\eps>0\),
\(V\in\cV_{2p}\), and either \(f\in\cF_{2p}\) or
\(f=\ell_\eta\) for fixed \(\eta>0\),
\begin{equation}\label{eq:coexistence-ratio}
 \frac{\E_{\mu_{p,\sigma}}\cZ_{N,L}^J(V,\eps,f)}
 {\E\cZ_{N,L}^G(V,\eps,f)}
 \longrightarrow1.
\end{equation}
The same holds for the point-incoherent functional.  Nevertheless, for
the compact interval \(B\) supplied by \cref{thm:counterexample},
\begin{equation}\label{eq:coexistence-exact}
 \liminf_{N\to\infty}\frac1N\log
 \E_{\mu_{p,\sigma}}\Crt_N^J(B)
 >
 \lim_{N\to\infty}\frac1N\log\E\Crt_N^G(B).
\end{equation}
Both conclusions, as well as the local-maximum and local-minimum
versions of \eqref{eq:coexistence-exact}, remain valid for the standard
symmetric-coordinate model \eqref{eq:symmetric-coordinate-model}.
\end{corollary}

\begin{proof}
The law \(\mu_{p,\sigma}\) is sub-Gaussian and matches Gaussian moments
through order \(2p\).  Since
\[
 (2p-1)(p-2)>2\qquad (p\ge3),
\]
\eqref{eq:coexistence-ratio} follows from
\cref{thm:bulk-universality,cor:point-incoherent}; the exact separation
is \cref{thm:counterexample}.  For the symmetric-coordinate model, use
\cref{cor:symmetric-bulk,cor:symmetric-exact}.
\end{proof}

\begin{theorem}[Bounded-disorder coexistence]
\label[theorem]{thm:bounded-coexistence}
For every \(p\ge3\), there is a law \(\mu\) as in
\cref{thm:bounded-block-counterexample}, chosen to match Gaussian
moments through order \(2p\).  For this compactly supported law and
every fixed \(L,\eps>0\), \(V\in\cV_{2p}\), and either
\(f\in\cF_{2p}\) or \(f=\ell_\eta\) for fixed \(\eta>0\),
\begin{equation}\label{eq:bounded-coexistence-ratio}
 \frac{\E_\mu\cZ_{N,L}^J(V,\eps,f)}
 {\E\cZ_{N,L}^G(V,\eps,f)}
 \longrightarrow1,
\end{equation}
and likewise for the point-incoherent functional.  Nevertheless, on
the compact interval \(B\) supplied by that theorem,
\[
 \liminf_{N\to\infty}\frac1N\log\E_\mu\Crt_N^J(B)
 >
 \lim_{N\to\infty}\frac1N\log\E\Crt_N^G(B).
\]
The analogous coexistence holds in the standard symmetric-coordinate
model for a compactly supported law supplied by
\cref{cor:symmetric-bounded-block} with matching order \(2p\).
\end{theorem}

\begin{proof}
Apply \cref{thm:bounded-block-counterexample} with \(r=2p\).  The
resulting law is bounded, hence subexponential, and
\((2p-1)(p-2)>2\).  The ratio assertion follows from
\cref{thm:bulk-universality,cor:point-incoherent}; the strict
unregularized separation is
\eqref{eq:bounded-block-total}.  In the symmetric model, use
\cref{cor:symmetric-bulk,cor:symmetric-bounded-block}.
\end{proof}

\subsection{A conditional reduction to exact universality}

For an orthogonal completion of \(x\in\mathbb S^{N-1}(1)\), define
\begin{equation}\label{eq:wNL}
 w_{N,L}(x)=\nu_{N,x}\left\{(v_1,\ldots,v_n):
 (v_1,\ldots,v_n,x)\in\cO_{N,L}\right\},
\end{equation}
where \(\nu_{N,x}\) is Haar probability on orthonormal bases of \(x^\perp\).  Define the exact frame-weighted count
\begin{equation}\label{eq:exact-weighted}
 \mathscr C_{N,L}^J(V)
 =\E\sum_{x:\nabla_SH_N^J(\sqrt N x)=0}
 w_{N,L}(x)e^{-NV(h_N^J(x))}.
\end{equation}
Also write \(\mathscr C_N^J(V)\) for the same count with \(w_{N,L}\equiv1\).

For Gaussian disorder, the expected energy-weighted critical-point
measure is a constant multiple of normalized surface measure.  Since
\(\int w_{N,L}(x)\,\overline\sigma_{N-1}(\dd x)
=\nu_N(\cO_{N,L})\) by disintegration of Haar measure, integrating
\(w_{N,L}\) against that expected measure gives the exact identity
\begin{equation}\label{eq:Gaussian-frame-weight}
 \mathscr C_{N,L}^G(V)=\nu_N(\cO_{N,L})\mathscr C_N^G(V).
\end{equation}
Hence \cref{lem:Haar-incoherent,thm:Gaussian-local-complexity} show that both sides have limiting exponent \(\Phi_p^{\rm bulk}(V)\) whenever \(L>L_0\).

\begin{remark}[Exact point-incoherent weight]\label[remark]{rem:exact-point-weight}
For the larger frame set \(\widehat\cO_{N,L}\) from
\cref{cor:point-incoherent}, the conditional finite-\(N\)
de-regularization statement in \cref{prop:finite-N-dereg}, under its
hard-slice hypotheses, gives
\[
 \widehat{\mathscr C}_{N,L}^J(V)
 :=\E\sum_{\substack{x:\nabla_SH_N^J(\sqrt N x)=0\\x\in D_{N,L}}}
 e^{-NV(h_N^J(x))}.
\]
Indeed, at finite \(\eps\) the gradient mollifier depends on the tangent
basis, but after \(\eps\downarrow0\) the Kac--Rice factor
\(\delta_0(g)\abs{\det A}\) is basis invariant and every tangent
completion of \(x\in D_{N,L}\) is allowed.  Thus the exact frame weight
is the indicator of point incoherence.  We henceforth use this exact
point-incoherent form in the transfer theorem; no completion estimate
is then needed.
\end{remark}

\begin{definition}[One-sided de-regularization defects]
\label[definition]{def:dereg-defects}
The point-incoherent exact count is
\begin{equation}\label{eq:exact-point-count}
 \widehat{\mathscr C}_{N,L}^J(V)
 :=\E\sum_{\substack{x:\nabla_SH_N^J(\sqrt N x)=0\\x\in D_{N,L}}}
 e^{-NV(h_N^J(x))}.
\end{equation}
Whenever this quantity is finite and positive, put
\begin{align}
 \widehat Q_{N,L}
 &:=
 \frac1N\log\widehat{\mathscr C}_{N,L}^J(V),\label{eq:Qhat}\\
 \widehat P_{N,L}(\eps,\eta)
 &:=
 \frac1N\log
 \E\widehat\cZ_{N,L}^J(V,\eps,\ell_\eta).\label{eq:Phat-uncut}
\end{align}
For a nonnegative function \(a(\eps,\eta)\), write
\[
 \liminf_{(\eps,\eta)\to(0,0)}a(\eps,\eta)
 :=\lim_{\delta\downarrow0}
 \inf_{0<\eps,\eta<\delta}a(\eps,\eta).
\]
Define the upper and lower transfer defects
\begin{align}
 \delta_{L,\mathrm{dr}}^{J,\mathrm{up}}(V)
 &:=
 \liminf_{(\eps,\eta)\to(0,0)}
 \limsup_{N\to\infty}
 [\widehat Q_{N,L}-\widehat P_{N,L}(\eps,\eta)]_+,
 \label{eq:direct-dereg-up}\\
 \delta_{L,\mathrm{dr}}^{J,\mathrm{low}}(V)
 &:=
 \liminf_{(\eps,\eta)\to(0,0)}
 \limsup_{N\to\infty}
 [\widehat P_{N,L}(\eps,\eta)-\widehat Q_{N,L}]_+.
 \label{eq:direct-dereg-low}
\end{align}
Thus only one favorable vanishing-regularizer test sequence is required
in each direction, and the two sequences may differ.  They require
neither a common test sequence nor an iterated absolute comparison of
the uncut pressures.  The frame-slice quantities in
\cref{prop:primitive-dereg} give checkable upper bounds for them.
\end{definition}

\begin{definition}[Localized-complement exponents]
\label[definition]{def:localization-defect}
Set
\begin{align}
 \mathscr L_{N,L}^J(V)
 &:=
 \E\sum_{\substack{x:\nabla_SH_N^J(\sqrt N x)=0\\x\notin D_{N,L}}}
 e^{-NV(h_N^J(x))},\label{eq:localized-defect-count}\\
 \overline\Lambda_L^J(V)
 &:=
 \limsup_{N\to\infty}\frac1N\log\mathscr L_{N,L}^J(V),\qquad
 \underline\Lambda_L^J(V)
 :=
 \liminf_{N\to\infty}\frac1N\log\mathscr L_{N,L}^J(V).
 \label{eq:localized-defect-exponents}
\end{align}
We use \(\log0=-\infty\) and allow extended-real exponents.
\end{definition}

\begin{theorem}[Quantitative exact-complexity transfer]
\label[theorem]{thm:conditional-transfer}
For all sufficiently large \(N\), suppose
\(\widehat Q_{N,L}\) and
\(\widehat P_{N,L}(\eps,\eta)\) are finite for every
\(\eps,\eta>0\).  Suppose also that there exist \(\Phi\in\R\) and a
function \(F\) on \((0,\infty)^2\) such that
\begin{equation}\label{eq:abstract-pressure-bridge}
 \widehat P_{N,L}(\eps,\eta)\longrightarrow F(\eps,\eta)
 \quad\text{for every fixed }\eps,\eta>0,
 \qquad
 F(\eps,\eta)\longrightarrow\Phi
 \quad\text{as }(\eps,\eta)\to(0,0).
\end{equation}
Then
\begin{align}
 \limsup_{N\to\infty}\widehat Q_{N,L}
 &\le\Phi+\delta_{L,\mathrm{dr}}^{J,\mathrm{up}}(V),
 \label{eq:weighted-upper-defect}\\
 \liminf_{N\to\infty}\widehat Q_{N,L}
 &\ge\Phi-\delta_{L,\mathrm{dr}}^{J,\mathrm{low}}(V).
 \label{eq:weighted-lower-defect}
\end{align}
Consequently,
\begin{align}
 \limsup_{N\to\infty}\frac1N\log\mathscr C_N^J(V)
 &\le
 \max\{\Phi+\delta_{L,\mathrm{dr}}^{J,\mathrm{up}}(V),
            \overline\Lambda_L^J(V)\},\label{eq:full-upper-defect}\\
 \liminf_{N\to\infty}\frac1N\log\mathscr C_N^J(V)
 &\ge
 \max\{\Phi-\delta_{L,\mathrm{dr}}^{J,\mathrm{low}}(V),
            \underline\Lambda_L^J(V)\}.
 \label{eq:full-lower-defect}
\end{align}
If both de-regularization defects vanish, then
\begin{equation}\label{eq:weighted-exact-bulk}
 \lim_{N\to\infty}\frac1N
 \log\widehat{\mathscr C}_{N,L}^J(V)=\Phi
\end{equation}
and the exact identities
\begin{align}
 \limsup_{N\to\infty}\frac1N\log\mathscr C_N^J(V)
 &=\max\{\Phi,\overline\Lambda_L^J(V)\},
 \label{eq:exact-upper-defect-formula}\\
 \liminf_{N\to\infty}\frac1N\log\mathscr C_N^J(V)
 &=\max\{\Phi,\underline\Lambda_L^J(V)\}
 \label{eq:exact-lower-defect-formula}
\end{align}
hold.  In particular,
\begin{equation}\label{eq:exact-universality-iff-localized-defect}
 \lim_{N\to\infty}\frac1N\log\mathscr C_N^J(V)=\Phi
\quad\Longleftrightarrow\quad
 \overline\Lambda_L^J(V)\le\Phi.
\end{equation}

A concrete sufficient condition for \eqref{eq:abstract-pressure-bridge}
is that, for any fixed \(L>0\), the disorder array and \(V\) satisfy the
hypotheses of \cref{thm:bulk-universality} with \(m=2\).
Indeed, \eqref{eq:point-Haar-zero-cost} is automatic.  In this concrete
case,
\[
 \Phi=\Phi_p^{\rm bulk}(V)
 =\sup_{u\in\R}\{\theta_p(u)-V(u)\}.
\]
When the two de-regularization defects vanish,
\eqref{eq:exact-universality-iff-localized-defect} therefore gives a
criterion for the localized branch that is both necessary and sufficient.
\end{theorem}

\begin{proposition}[Bounded-regularizer transfer below the
\texorpdfstring{\(\psi_1\)}{psi1} scale]
\label[proposition]{prop:weibull-exact-transfer}
Fix any \(L>0\).  Assume the hypotheses on the disorder array and \(V\)
from \cref{thm:weibull-bulk}, and assume
that \(\widehat{\mathscr C}_{N,L}^J(V)\) is finite and positive for all
sufficiently large \(N\).  Put
\(\Phi=\Phi_p^{\rm bulk}(V)\) and
\[
 \widehat P^{\,c}_{N,L}(\eps,\eta,R)
 =
 \frac1N\log
 \E\widehat\cZ_{N,L}^J(V,\eps,f_{\eta,R})
\]
and define
\begin{align*}
 \delta_{L,\mathrm{dr},c}^{J,\mathrm{up}}(V)
 &:=
 \liminf_{R\to\infty}\liminf_{\eta\downarrow0}
 \liminf_{\eps\downarrow0}\limsup_N
 [\widehat Q_{N,L}-\widehat P^{\,c}_{N,L}(\eps,\eta,R)]_+,\\
 \delta_{L,\mathrm{dr},c}^{J,\mathrm{low}}(V)
 &:=
 \liminf_{R\to\infty}\liminf_{\eta\downarrow0}
 \liminf_{\eps\downarrow0}\limsup_N
 [\widehat P^{\,c}_{N,L}(\eps,\eta,R)-\widehat Q_{N,L}]_+.
\end{align*}
Then every conclusion of \cref{thm:conditional-transfer} holds with
the two uncut defects replaced by these cutoff defects.  In particular,
if both vanish, the exact max identities
\eqref{eq:exact-upper-defect-formula}--\eqref{eq:exact-lower-defect-formula}
and the necessary-and-sufficient criterion
\eqref{eq:exact-universality-iff-localized-defect} remain valid.
This proposition preserves the Weibull transfer route but makes no
uncut determinant or remote-tail claim.
\end{proposition}

\begin{proof}
For fixed \((\eps,\eta,R)\),
\cref{thm:weibull-bulk,cor:point-incoherent,thm:gaussian-variational}
give convergence of
\(\widehat P^{\,c}_{N,L}\) to the cutoff Gaussian variational
expression.  Its iterated limit is \(\Phi\) by
\eqref{eq:dereg-limit}.  To spell out the diagonal selection, consider
either cutoff defect.  If it is finite, fix \(\gamma\) strictly larger
than that defect.  The three nested
liminfs allow us to choose successively
\[
 R_j\longrightarrow\infty,\qquad
 \eta_j\downarrow0,\qquad
 \eps_j\downarrow0
\]
so that the corresponding one-sided discrepancy is at most \(\gamma\)
after \(\limsup_N\).  Because the Gaussian cutoff expression has a full
iterated limit, the choices may simultaneously be made so that it tends
to \(\Phi\).  The two one-sided arguments in the proof of
\cref{thm:conditional-transfer} then give the claimed bounds.  If a
defect is infinite, its corresponding bound is immediate.
\end{proof}

\begin{corollary}[Nonnegative-bulk criterion]
\label[corollary]{cor:nonnegative-bulk}
\leavevmode\par\noindent
In the setting of \cref{thm:conditional-transfer}, suppose that both
de-regularization defects vanish and \(\Phi\ge0\).  If
\begin{equation}\label{eq:zero-rate-localized-complement}
 \overline\Lambda_L^J(V)\le0,
\end{equation}
then the full exact-complexity limit exists and equals \(\Phi\).
\end{corollary}

\begin{proof}
This is immediate from
\eqref{eq:exact-universality-iff-localized-defect}.
\end{proof}

\begin{proposition}[The localization branch is genuinely law dependent]
\label[proposition]{prop:localization-defect-necessary}
Fix \(p\ge3\) and a finite matching order \(r\ge2\).  There exist an
absolutely continuous, compactly supported, centered unit-variance law
\(\mu\), matching the Gaussian moments through order \(r\), and a
potential \(V\in C_b^\infty(\R)\), \(V\ge0\), such that for every fixed
\(L>0\),
\begin{equation}\label{eq:block-positive-localization-defect}
 \underline\Lambda_L^\mu(V)
 >\Phi_p^{\rm bulk}(V).
\end{equation}
For every fixed \(\varepsilon\in[0,1/2)\), the same law satisfies
\[
 \max_I|J_I|\le N^{1/2-\varepsilon}
\]
eventually almost surely.
\end{proposition}

\begin{proof}
Use the notation from the proof of
\cref{thm:bounded-block-counterexample}: let
\(k_N=\lfloor cN^{1/p}\rfloor\), let \(\pi>0\) be the probability of
the prescribed coefficient interval, and let \(\mathcal K_{N,1/2}\)
be the corresponding block cap.  On the block-alignment event and the
independent bounded-remainder event, which together have probability
at least \(\frac12\pi^{k_N^p}\), the cap argument produces a local
maximum in \(\mathcal K_{N,1/2}\) with energy in a compact interval
\(B\).  Hence
\[
 \liminf_N\frac1N\log
 \E_\mu\#\{x\in\mathcal K_{N,1/2}:
 x\text{ is a local maximum},\ h_N^J(x)\in B\}
 \ge c^p\log\pi=:\Gamma.
\]
The parameters in that proof were chosen so that
\(\Gamma>\sup_{u\in B}\theta_p(u)\).
By \cref{cor:block-growing-cutoff-separation}, there is
\(\kappa=\kappa(c)>0\) such that every point in the cap satisfies
\[
 \norm{x}_\infty\ge\kappa N^{-1/(2p)}
\]
for all sufficiently large \(N\).  For every fixed \(L\), this
eventually exceeds
\(L\sqrt{\log N/N}\); hence the cap is disjoint from \(D_{N,L}\).
By continuity of \(\theta_p\), choose an open neighborhood
\(U\supset B\) such that
\(\sup_{u\in U}\theta_p(u)<\Gamma\).  Since
\(\sup_\R\theta_p<\infty\), choose
\(C>\sup_\R\theta_p-\Gamma\), and take
\(V\in C_b^\infty(\R)\) with
\[
 V=0\ \text{on }B,\qquad 0\le V\le C,\qquad
 V=C\ \text{on }U^c.
\]
Then
\[
 \Phi_p^{\rm bulk}(V)
 =\sup_u\{\theta_p(u)-V(u)\}<\Gamma.
\]
The localized-complement count dominates the displayed cap count, which
proves \eqref{eq:block-positive-localization-defect}.  Compact support
gives the cutoff assertion.
\end{proof}

\begin{remark}
Finite matching, ordinary sub-Gaussianity, bounded disorder, and a fixed
\(N^{1/2-\varepsilon}\) entry cutoff with
\(0\le\varepsilon<1/2\) therefore do not imply the localized-branch
condition \(\overline\Lambda_L^J(V)\le\Phi\), and hence cannot by
themselves imply exact universality.  Nor does finite-\(N\)
smoothness by itself establish that the two de-regularization defects
vanish: a
high-dimensional gradient local limit and near-zero Hessian control are
genuine exponential-scale questions.  What has been removed is the
artificial remote spectral cutoff and the stronger requirement of a
full iterated absolute comparison.
\end{remark}

\section{Spherical geometry and the derivative jet}\label{sec:jet}

\subsection{Spherical gradient and Hessian}

Let \(F:\R^N\to\R\) be \(C^2\).  At \(\sigma\in\SN\), let
\[
 P_\sigma=I-\frac{\sigma\sigma^{\mathsf T}}N
\]
be the orthogonal projection onto \(T_\sigma\SN\).

\begin{lemma}[Spherical derivatives]\label[lemma]{lem:spherical-derivatives}
For \(v,w\in T_\sigma\SN\),
\begin{align}
 \nabla_SF(\sigma)&=P_\sigma\nabla F(\sigma),\label{eq:sph-grad}\\
 \Hess_SF(\sigma)[v,w]
 &=D^2F(\sigma)[v,w]
 -\frac{\ip{\nabla F(\sigma)}{\sigma}}N\ip vw.\label{eq:sph-hess}
\end{align}
If \(F\) is homogeneous of degree \(p\), then
\begin{equation}\label{eq:sph-hess-hom}
 \Hess_SF(\sigma)[v,w]
 =D^2F(\sigma)[v,w]-\frac{pF(\sigma)}N\ip vw.
\end{equation}
\end{lemma}

\begin{proof}
The gradient identity follows by restricting the differential to the tangent space.  If \(\gamma\) is the spherical geodesic with \(\gamma(0)=\sigma\), \(\dot\gamma(0)=v\), then \(\ddot\gamma(0)=-\norm v^2\sigma/N\).  Hence
\[
 \frac{\dd^2}{\dd t^2}F(\gamma(t))\bigg|_{t=0}
 =D^2F(\sigma)[v,v]-\frac{\ip{\nabla F(\sigma)}\sigma}N\norm v^2.
\]
Polarization gives \eqref{eq:sph-hess}.  Euler's identity
\(\ip{\nabla F(\sigma)}\sigma=pF(\sigma)\) gives \eqref{eq:sph-hess-hom}.
\end{proof}

\subsection{Explicit jet formulas}

For \(I=(i_1,\ldots,i_p)\), let \(x_I=\prod_r x_{i_r}\).  For a frame \(O=(v_1,\ldots,v_n,x)\), define
\begin{align}
 b_{a,I}(O)&=\sum_{s=1}^p v_{a,i_s}\prod_{r\ne s}x_{i_r},\label{eq:b-def}\\
 C_{ab,I}(O)&=N^{-1/2}\sum_{\substack{1\le s,t\le p\\s\ne t}}
 v_{a,i_s}v_{b,i_t}\prod_{r\ne s,t}x_{i_r}
 -pN^{-1/2}x_I\delta_{ab}.
 \label{eq:C-def}
\end{align}

\begin{proposition}[Linear representation of the frame jet]\label[proposition]{prop:jet-linear}
For every frame \(O\),
\begin{align}
 h_J(O)&=\sum_I e_I(O)J_I,
 &e_I(O)&=N^{-1/2}x_I,\label{eq:jet-h-linear}\\
 g_{a,J}(O)&=\sum_I b_{a,I}(O)J_I,\label{eq:jet-g-linear}\\
 A_{ab,J}(O)&=\sum_I C_{ab,I}(O)J_I.\label{eq:jet-A-linear}
\end{align}
\end{proposition}

\begin{proof}
Equation \eqref{eq:jet-h-linear} is \eqref{eq:h-def}.  Differentiating \eqref{eq:model-intro} in direction \(v_a\) and substituting \(\sigma=\sqrt N x\), the factor \(N^{-(p-1)/2}\) cancels \(N^{(p-1)/2}\) from the remaining \(p-1\) coordinates, giving \eqref{eq:jet-g-linear}.  A second derivative leaves \(p-2\) factors of \(\sqrt N\), producing the prefactor \(N^{-1/2}\) in the first term of \eqref{eq:C-def}.  The second term follows from \eqref{eq:sph-hess-hom} and \(H_N^J/N=h_J\).
\end{proof}

\subsection{Gaussian derivative jet}

We use the GOE normalization in which \(M_n\) is symmetric, its off-diagonal entries have variance \(1/(2n)\), its diagonal entries have variance \(1/n\), and its limiting spectrum is supported on \([-\sqrt2,\sqrt2]\).

\begin{proposition}[Gaussian jet law]\label[proposition]{prop:Gaussian-jet}
Let the disorder be standard Gaussian.  For every deterministic frame \(O\),
\begin{equation}\label{eq:Gaussian-jet-law}
 (h_G(O),g_G(O),A_G(O))\stackrel d=
 \left(
 \frac{Z_0}{\sqrt N},
 \sqrt p\,(Z_1,\ldots,Z_n),
 \gamma_{N,p}M_n-\frac{pZ_0}{\sqrt N}I_n
 \right),
 \qquad
 \gamma_{N,p}=\sqrt{\frac{2p(p-1)n}{N}}.
\end{equation}
Here \(Z_0,Z_1,\ldots,Z_n\) are independent standard Gaussians and \(M_n\) is an independent GOE matrix.
\end{proposition}

\begin{proof}
For every \(O\in\ON\), the map \(J\mapsto O^{\tensor p}J\) is orthogonal on \(\R^{N^p}\) and therefore preserves the standard Gaussian ordered tensor.  It is thus enough to take \(x=e_N\) and \(v_a=e_a\), \(1\le a\le n\).  Then
\[
 h_G=N^{-1/2}G_{N\cdots N}.
\]
For each \(a\), the gradient is the sum of the \(p\) independent entries obtained by placing \(a\) in one of the \(p\) positions and \(N\) elsewhere, so it is \(\cN(0,p)\).  These sets of tensor coordinates are pairwise disjoint and disjoint from \(G_{N\cdots N}\).

For \(a\ne b\), the ambient Hessian entry is \(N^{-1/2}\) times the sum of the \(p(p-1)\) independent coordinates obtained by placing \(a\) and \(b\) in two distinct ordered positions; its variance is \(p(p-1)/N\).  For \(a=b\), each unordered pair of positions is counted twice, so the variance is \(4\binom p2/N=2p(p-1)/N\).  The covariance pattern is exactly that of \(\gamma_{N,p}M_n\).  These tensor coordinates are disjoint from those in the energy and gradient.  Finally, the spherical correction is \(-ph_GI_n\).
\end{proof}

\section{Incoherent frames}\label{sec:frames}

The next lemma is standard, but we include a proof because the logarithmic incoherence scale drives the replacement exponent.

\begin{lemma}[Haar frames are incoherent]\label[lemma]{lem:Haar-incoherent}
There are universal constants \(c,C>0\) such that, for Haar \(O\in\ON\),
\begin{equation}\label{eq:Haar-entry-tail}
 \Pp(\abs{O_{ij}}>t)\le C e^{-cNt^2},
 \qquad 0<t<1.
\end{equation}
Consequently, if \(L>L_0:=\sqrt{3/c}\), then
\begin{equation}\label{eq:good-prob}
 \nu_N(\cO_{N,L})\longrightarrow1,
 \qquad
 \frac1N\log\nu_N(\cO_{N,L})\longrightarrow0.
\end{equation}
For the point cutoff, however, every fixed \(L>0\) has zero
exponential cost:
\begin{equation}\label{eq:point-Haar-zero-cost}
 \frac1N\log\overline\sigma_{N-1}(D_{N,L})
 \longrightarrow0.
\end{equation}
\end{lemma}

\begin{proof}
Each entry of a Haar orthogonal matrix has the law of the first coordinate of a uniform point on \(\mathbb S^{N-1}(1)\).  Writing this point as \(Z/\norm Z\) with \(Z\sim\cN(0,I_N)\), a Gaussian tail bound together with \(\Pp(\norm Z^2<N/2)\le e^{-cN}\) yields \eqref{eq:Haar-entry-tail}.  A union bound over the \(N^2\) entries gives, for \(r_N=L\sqrt{\log N/N}\),
\[
 \Pp(O\notin\cO_{N,L})\le CN^2e^{-cL^2\log N}
 =CN^{2-cL^2},
\]
which tends to zero when \(cL^2>2\).  The logarithmic statement follows because a probability converging to one has logarithm \(o(1)\).

For the last assertion, set \(c_N=L\sqrt{\log N}\) and use the same
Gaussian representation.  With
\[
 A_N=\left\{\max_i|Z_i|\le\frac{c_N}{2}\right\},
 \qquad B_N=\left\{\|Z\|\ge\frac{\sqrt N}{2}\right\},
\]
we have \(A_N\cap B_N\subset\{Z/\|Z\|\in D_{N,L}\}\).  If
\(q_N^\circ=\Pp\{|G|>c_N/2\}\), then \(q_N^\circ\to0\) and, for all
large \(N\),
\[
 \Pp(A_N)=(1-q_N^\circ)^N
 \ge e^{-2Nq_N^\circ}=e^{-o(N)}.
\]
Since \(\Pp(B_N^c)\le e^{-c_1N}\), it follows that
\(\Pp(A_N\cap B_N)\ge\frac12\Pp(A_N)=e^{-o(N)}\).
The reverse exponential bound is automatic because the normalized
spherical measure is at most one, proving \eqref{eq:point-Haar-zero-cost}.
\end{proof}

Set throughout
\begin{equation}\label{eq:rN-qN}
 r_N=L\sqrt{\frac{\log N}{N}},
 \qquad q_N=\sqrt N\,r_N^{p-1}
 =L^{p-1}N^{-(p-2)/2}(\log N)^{(p-1)/2}.
\end{equation}
For \(p\ge3\), \(q_N\to0\).

\begin{lemma}[Coefficient bounds on incoherent frames]\label[lemma]{lem:coefficient-bounds}
For \(O\in\cO_{N,L}\) and every \(I\in[N]^p\),
\begin{align}
 \abs{e_I(O)}&\le N^{-1/2}r_N^p,\label{eq:e-bound}\\
 \sum_{a=1}^n\abs{b_{a,I}(O)}^2
 &\le C_pr_N^{2p-2},\label{eq:b-bound}\\
 \norm{C_I(O)}_{\HS}
 &\le C_p\bigl(N^{-1/2}r_N^{p-2}+r_N^p\bigr),
 &\norm{C_I(O)}_{\op}
 &\le C_pN^{-1/2}r_N^{p-2}.
 \label{eq:C-norm-bound}
\end{align}
\end{lemma}

\begin{proof}
The energy bound is immediate.  Let \(m_I(x)=x_{i_1}\cdots x_{i_p}\).  Since \(b_{a,I}=\ip{v_a}{\nabla m_I(x)}\), orthogonal projection and the fact that \(\nabla m_I\) is supported on at most \(p\) coordinates give
\[
 \sum_a\abs{b_{a,I}}^2
 \le\norm{\nabla m_I(x)}_2^2
 \le C_pr_N^{2p-2}.
\]
The ambient part of \(C_I\) is \(N^{-1/2}\) times the compression of \(D^2m_I(x)\) to \(x^\perp\).  This ambient Hessian is supported on at most \(p\) rows and columns, and its entries are bounded by \(C_pr_N^{p-2}\).  Hence both its Hilbert--Schmidt and operator norms are at most \(C_pN^{-1/2}r_N^{p-2}\).  The spherical correction in \eqref{eq:C-def} is \(-pN^{-1/2}m_I(x)I_n\); its Hilbert--Schmidt norm is at most \(C_pr_N^p\), while its operator norm is at most \(C_pN^{-1/2}r_N^p\).  Since \(r_N\le1\) for large \(N\), the claimed bounds follow.
\end{proof}

\section{A multiplicative Lindeberg principle}\label{sec:lindeberg}

An additive invariance principle is insufficient for annealed complexity because the partition function itself is exponential in \(N\).  We need a comparison that is relative to the positive observable.

\begin{lemma}[One-coordinate multiplicative replacement]\label[lemma]{lem:one-coordinate-Lindeberg}
Fix \(m\ge2\).  Let \(X,Y\) be real random variables satisfying
\[
 \E X^r=\E Y^r,
 \qquad r=1,\ldots,m,
\]
and
\begin{equation}\label{eq:exp-moment-XY}
 \E e^{\abs X/K}+\E e^{\abs Y/K}\le4.
\end{equation}
Let \(F:\R\to(0,\infty)\) be \(C^{m+1}\), fix constants \(C_1,\ldots,C_{m+1}<\infty\), and set
\begin{equation}\label{eq:a0-def}
 a_0=a_0(K,m,C_1):=\min\{1,(2KC_1)^{-1}\},
\end{equation}
with the convention \(a_0=1\) if \(C_1=0\).  Suppose that, for some \(0<a\le a_0\),
\begin{equation}\label{eq:multiplicative-derivative}
 \abs{F^{(r)}(t)}\le C_r a^rF(t),
 \qquad t\in\R,\quad r=1,\ldots,m+1.
\end{equation}
Then
\begin{equation}\label{eq:one-coordinate-log}
 \abs{\log\E F(X)-\log\E F(Y)}\le C a^{m+1},
\end{equation}
where \(C\) depends only on \(m,K,C_1,\ldots,C_{m+1}\).
\end{lemma}

\begin{proof}
The first derivative estimate implies
\begin{equation}\label{eq:log-lipschitz-F}
 e^{-C_1a\abs t}F(0)\le F(t)\le e^{C_1a\abs t}F(0).
\end{equation}
Taylor's theorem at zero gives
\[
 F(t)=\sum_{r=0}^m\frac{F^{(r)}(0)}{r!}t^r+R_{m+1}(t),
\]
where, by \eqref{eq:multiplicative-derivative} and \eqref{eq:log-lipschitz-F},
\begin{equation}\label{eq:Taylor-rem-bound}
 \abs{R_{m+1}(t)}
 \le C a^{m+1}\abs t^{m+1}e^{C_1a\abs t}F(0).
\end{equation}
By \eqref{eq:a0-def}, \(C_1a_0\le(2K)^{-1}\).  The exponential-moment assumption then implies
\[
 \E\abs X^{m+1}e^{C_1a\abs X}
 +\E\abs Y^{m+1}e^{C_1a\abs Y}\le C_{m,K}.
\]
Moment matching cancels the Taylor polynomials, hence
\begin{equation}\label{eq:EF-diff}
 \abs{\E F(X)-\E F(Y)}\le Ca^{m+1}F(0).
\end{equation}
On the other hand, \eqref{eq:log-lipschitz-F} and Jensen's inequality give
\[
 \E F(X)\ge F(0)\E e^{-C_1a\abs X}
 \ge F(0)e^{-C_1a\E\abs X}\ge cF(0),
\]
and similarly for \(Y\).  Dividing \eqref{eq:EF-diff} by the smaller expectation and using \(\abs{\log u-\log v}\le\abs{u-v}/\min(u,v)\) proves the claim.
\end{proof}

\begin{theorem}[Multiplicative Lindeberg replacement]\label[theorem]{thm:multiplicative-Lindeberg}
Let \(X_1,\ldots,X_M\) and \(Y_1,\ldots,Y_M\) be two independent families.  Assume that, for each \(i\), \(X_i,Y_i\) satisfy the hypotheses of \cref{lem:one-coordinate-Lindeberg} with the same constants and match through order \(m\).  Let \(F:\R^M\to(0,\infty)\) be \(C^{m+1}\).  Suppose that for each coordinate \(i\), every \(z\in\R^M\), and \(r=1,\ldots,m+1\),
\begin{equation}\label{eq:multi-derivative}
 \abs{\partial_i^rF(z)}\le C_r a_i^rF(z),
 \qquad 0<a_i\le a_0(K,m,C_1).
\end{equation}
Then
\begin{equation}\label{eq:multi-log-bound}
 \abs{\log\E F(X)-\log\E F(Y)}
 \le C\sum_{i=1}^M a_i^{m+1}.
\end{equation}
\end{theorem}

\begin{proof}
Replace the coordinates one at a time.  At the \(i\)-th step, condition on all coordinates except the \(i\)-th and apply \cref{lem:one-coordinate-Lindeberg} pointwise.  Integrating the upper and lower relative comparisons over the remaining variables gives
\[
 e^{-Ca_i^{m+1}}\E F(Z^{(i-1)})
 \le\E F(Z^{(i)})
 \le e^{Ca_i^{m+1}}\E F(Z^{(i-1)}),
\]
where \(Z^{(i)}\) has its first \(i\) coordinates from the \(Y\)-family and the rest from the \(X\)-family.  Taking logarithms and summing proves \eqref{eq:multi-log-bound}.
\end{proof}

The preceding sequential principle pays the sum of coordinatewise worst-case
influences.  For the pressure, a simultaneous mixture path preserves the much
smaller aggregate influence budget.

\begin{lemma}[Product-mixture Lindeberg path]\label[lemma]{lem:product-mixture-path}
Let \(\mu_i\) and \(\gamma_i\), \(1\le i\le M\), be probability laws whose
moments agree through order \(m\), and put
\[
 \mu_{i,t}=(1-t)\mu_i+t\gamma_i,
 \qquad \lambda_t=\bigotimes_{i=1}^M\mu_{i,t},
 \qquad A(t)=\int_{\R^M}F(z)\,\lambda_t(\dd z),
\]
where \(F:\R^M\to(0,\infty)\) is \(C^{m+1}\).  Let
\(Z\sim\lambda_t\), \(X_i\sim\mu_i\), and \(Y_i\sim\gamma_i\) be
independent.  Assume \(0<A(t)<\infty\) on \([0,1]\), that the
expressions below are absolutely integrable, and that differentiation
under the finite product measure is justified.  Set
\(\Delta_i^X=X_i-Z_i\) and \(\Delta_i^Y=Y_i-Z_i\).  If, uniformly in
\(t\in[0,1]\),
\begin{align}
 &\sum_{i=1}^M\E\left[
  \abs{\Delta_i^X}^{m+1}
  \int_0^1(1-s)^m
  \abs{\partial_i^{m+1}F(Z+s\Delta_i^Xe_i)}\dd s\right]
 \le R A(t),\label{eq:mixture-remainder-X}\\
 &\sum_{i=1}^M\E\left[
  \abs{\Delta_i^Y}^{m+1}
  \int_0^1(1-s)^m
  \abs{\partial_i^{m+1}F(Z+s\Delta_i^Ye_i)}\dd s\right]
 \le R A(t),\label{eq:mixture-remainder-Y}
\end{align}
then
\begin{equation}\label{eq:product-mixture-bound}
 \abs{\log A(1)-\log A(0)}\le \frac{2R}{m!}.
\end{equation}
\end{lemma}

\begin{proof}
Since the product is finite, differentiation of the multilinear product
measure gives
\[
 A'(t)=\sum_{i=1}^M
 \E\bigl[F(Z_{-i},Y_i)-F(Z_{-i},X_i)\bigr].
\]
Adjoin the independent coordinate \(Z_i\).  Taylor expansion at the common
base point \(Z\) gives, for \(U_i=X_i\) or \(Y_i\),
\[
 F(Z+(U_i-Z_i)e_i)
 =\sum_{r=0}^m\frac{\partial_i^rF(Z)}{r!}(U_i-Z_i)^r
 +\frac{(U_i-Z_i)^{m+1}}{m!}
 \int_0^1(1-s)^m
 \partial_i^{m+1}F(Z+s(U_i-Z_i)e_i)\dd s.
\]
Conditionally on \(Z\), the moments of \(Y_i-Z_i\) and \(X_i-Z_i\)
agree through order \(m\), because they are polynomials in \(Z_i\) whose
coefficients are the corresponding moments of \(Y_i\) and \(X_i\).
Thus all Taylor-polynomial terms cancel in \(A'(t)\).  The two integral
remainders and \eqref{eq:mixture-remainder-X}--\eqref{eq:mixture-remainder-Y}
give \(\abs{A'(t)}\le2RA(t)/m!\).  Divide by \(A(t)>0\) and integrate over
\(t\in[0,1]\).
\end{proof}

\section{Coordinate influences of the regularized Kac--Rice functional}\label{sec:influence}

For a fixed frame, write the logarithm of the integrand in \eqref{eq:Z-def} as
\begin{equation}\label{eq:Psi-def}
 \Psi_O(J)=-NV(h_J(O))
 +\sum_{a=1}^n\log\kappa_\eps(g_{a,J}(O))
 +\Tr f(A_J(O)).
\end{equation}
We estimate derivatives with respect to a single tensor coordinate \(J_I\).

\begin{lemma}[Energy and gradient influence]\label[lemma]{lem:energy-gradient-influence}
For \(O\in\cO_{N,L}\), \(I\in[N]^p\), and \(1\le r\le m+1\),
\begin{align}
 \abs{\partial_I^r[-NV(h_J(O))]}&\le C q_N^r,\label{eq:energy-influence}\\
 \abs{\partial_I^r\sum_{a=1}^n\log\kappa_\eps(g_{a,J}(O))}
 &\le C q_N^r.\label{eq:gradient-influence}
\end{align}
The constants are uniform in \(J,N,O,I\).
\end{lemma}

\begin{proof}
Because \(h_J\) is linear in \(J_I\),
\[
 \partial_I^r[-NV(h_J)]=-NV^{(r)}(h_J)e_I^r.
\]
By \eqref{eq:e-bound},
\[
 N\abs{e_I}^r\le N^{1-r/2}r_N^{pr}
 =q_N^rN^{1-r}r_N^r\le q_N^r,
\]
which proves \eqref{eq:energy-influence}.

Similarly, since each \(g_{a,J}\) is linear in \(J_I\),
\[
 \partial_I^r\sum_a\log\kappa_\eps(g_{a,J})
 =\sum_a(\log\kappa_\eps)^{(r)}(g_{a,J})b_{a,I}^r.
\]
For \(r=1\), Cauchy--Schwarz and \eqref{eq:b-bound} give
\[
 \sum_a\abs{b_{a,I}}\le\sqrt N
 \left(\sum_a\abs{b_{a,I}}^2\right)^{1/2}
 \le C_pq_N.
\]
For \(r\ge2\), monotonicity of \(\ell^r\)-norms gives
\[
 \sum_a\abs{b_{a,I}}^r
 \le\left(\sum_a\abs{b_{a,I}}^2\right)^{r/2}
 \le C_{p,m}r_N^{(p-1)r}\le C_{p,m}q_N^r.
\]
Together with \eqref{eq:kappa-deriv-bound}, these estimates prove \eqref{eq:gradient-influence}.
\end{proof}

The spectral term uses the matrix derivative estimate proved in \cref{app:matrix-derivatives}.

\begin{lemma}[Spectral influence]\label[lemma]{lem:spectral-influence}
For \(O\in\cO_{N,L}\), \(I\in[N]^p\), and \(1\le r\le m+1\),
\begin{equation}\label{eq:spectral-influence}
 \abs{\partial_I^r\Tr f(A_J(O))}\le C q_N^r.
\end{equation}
\end{lemma}

\begin{proof}
By linearity, \(\partial_IA_J=C_I\) and higher coordinate derivatives
of \(A_J\) vanish.  Use \cref{lem:trace-derivatives} when
\(f\in\cF_m\), and \cref{lem:uncut-logdet-derivatives} when
\(f=\ell_\eta\).  For \(r=1\), the applicable estimate and
\eqref{eq:C-norm-bound} give
\[
 \abs{D\Tr f(A_J)[C_I]}
 \le C\sqrt N\norm{C_I}_{\HS}
 \le C\bigl(r_N^{p-2}+\sqrt N r_N^p\bigr)
 \le Cq_N,
\]
where the last inequality uses \(\sqrt N r_N\ge1\) and \(r_N\le1\) for large \(N\).
For \(r\ge2\), the same applicable lemma gives
\[
 \abs{D^r\Tr f(A_J)[C_I,\ldots,C_I]}
 \le C_r\norm{C_I}_{\HS}^2\norm{C_I}_{\op}^{r-2}.
\]
Each norm in \eqref{eq:C-norm-bound} is at most \(C_pq_N\) for large
\(N\).  The right-hand side is therefore at most \(C_rq_N^r\), with
the constant allowed to depend on \(\eta\) in the uncut case.
\end{proof}

\begin{proposition}[Influence of the positive integrand]\label[proposition]{prop:integrand-influence}
Let
\[
 F_O(J)=e^{\Psi_O(J)}.
\]
For each \(1\le r\le m+1\),
\begin{equation}\label{eq:FO-derivative}
 \abs{\partial_I^rF_O(J)}\le C_rq_N^rF_O(J),
\end{equation}
uniformly over \(O\in\cO_{N,L}\), \(I\in[N]^p\), and \(J\in\R^{N^p}\).  The same estimate holds for
\begin{equation}\label{eq:F-total}
 F(J):=\cZ_{N,L}^J(V,\eps,f).
\end{equation}
\end{proposition}

\begin{proof}
By \cref{lem:energy-gradient-influence,lem:spectral-influence},
\[
 \abs{\partial_I^r\Psi_O(J)}\le C_rq_N^r.
\]
Fa\`a di Bruno's formula expresses \(\partial_I^r e^{\Psi_O}\) as
\(e^{\Psi_O}\) times a finite sum of products
\(\prod_j\partial_I^{r_j}\Psi_O\) with \(\sum_jr_j=r\).  Each product
is \(O(q_N^r)\), proving \eqref{eq:FO-derivative}.  For
\(f\in\cF_m\), differentiation under the Haar integral follows from
the deterministic envelope
\[
 F_O(J)\le
 \exp\{-N\inf V+n\log(1/(\pi\eps))+n\norm f_\infty\},
\]
and the preceding estimates give the same envelope up to a constant
factor for every derivative used here.  For \(f=\ell_\eta\),
\[
 e^{\Tr\ell_\eta(A_J)}
 =\det(A_J^2+\eta^2I_n)^{1/2}
 \le(\eta+\norm{A_J}_{\HS})^n.
\]
The linear jet representation and
\eqref{eq:C-covariance-identity} give, uniformly in \(O\),
\[
 \norm{A_J(O)}_{\HS}
 \le C_p\sqrt N\,\norm J_2.
\]
This is a polynomial envelope for each fixed \(N\).  It is integrable
under every product-mixture law by uniform subexponentiality, and the
derivative bounds give the same envelope for all derivatives used
below.  Thus dominated differentiation and positivity give
\[
 \abs{\partial_I^rF(J)}
 \le\omega_N\int_{\cO_{N,L}}\abs{\partial_I^rF_O(J)}\nu_N(\dd O)
 \le C_rq_N^rF(J).
\]
\end{proof}

The coordinatewise estimate in \cref{prop:integrand-influence} is useful for
ordinary sequential replacement, but it discards an orthogonality gain.  The
following estimate retains that gain.  Define the deterministic influence
weight
\begin{equation}\label{eq:dI-def}
 \mathfrak d_I(O):=N e_I(O)^2+
 \sum_{a=1}^n b_{a,I}(O)^2+\norm{C_I(O)}_{\HS}^2.
\end{equation}

\begin{lemma}[Aggregate jet budget]\label[lemma]{lem:aggregate-jet-budget}
Uniformly in \(O\in\cO_{N,L}\),
\begin{equation}\label{eq:dI-budget}
 \sum_{I\in[N]^p}\mathfrak d_I\le C N,
 \qquad
 \max_{I\in[N]^p}\mathfrak d_I\le Cq_N^2.
\end{equation}
Writing \(\alpha_I(J)=\partial_I\Psi_O(J)\), one also has
\begin{equation}\label{eq:first-aggregate-budget}
 \max_I\abs{\alpha_I(J)}\le Cq_N,
 \qquad
 \sum_I\abs{\alpha_I(J)}^2\le CN,
\end{equation}
and, for \(2\le k\le m+1\),
\begin{equation}\label{eq:higher-log-jet-budget}
 \abs{\partial_I^k\Psi_O(J)}
 \le C_k\mathfrak d_Iq_N^{k-2}.
\end{equation}
Consequently, for \(2\le k\le m+1\),
\begin{align}
 \abs{\partial_I^kF_O(J)}
 &\le C_kF_O(J)
 \left(\abs{\alpha_I(J)}^k+
       \mathfrak d_Iq_N^{k-2}\right),
 \label{eq:pointwise-aggregate-F}\\
 \sum_I\abs{\partial_I^kF_O(J)}
 &\le C_kNq_N^{k-2}F_O(J).
 \label{eq:aggregate-F-budget}
\end{align}
All constants are uniform in \(J,N,O\).
\end{lemma}

\begin{proof}
The linear representation in \cref{prop:jet-linear} and the Gaussian jet
law in \cref{prop:Gaussian-jet} give the exact identities
\begin{align}
 \sum_I e_I^2&=\frac1N,
 &\sum_I b_{a,I}b_{c,I}&=p\delta_{ac},
 \label{eq:coefficient-orthogonality}\\
 \sum_I C_{ab,I}C_{cd,I}
 &=\frac{p(p-1)}N
 (\delta_{ac}\delta_{bd}+\delta_{ad}\delta_{bc})
 +\frac{p^2}N\delta_{ab}\delta_{cd}.
 \label{eq:C-covariance-identity}
\end{align}
Indeed, the left-hand sides are precisely the corresponding covariances when
the tensor coordinates are independent standard Gaussians.  Summing the
diagonal cases in these identities proves the first bound in
\eqref{eq:dI-budget}.  The second follows from
\cref{eq:e-bound,eq:b-bound,eq:C-norm-bound}; in particular,
\(N^{-1}r_N^{2p-4}+r_N^{2p}\le Cq_N^2\) for all sufficiently large
\(N\).

Put \(\ell_\eps=\log\kappa_\eps\).  The three parts of the first log
derivative are
\[
 -NV'(h_J)e_I,
 \qquad \sum_a\ell_\eps'(g_{a,J})b_{a,I},
 \qquad \Tr\bigl(f'(A_J)C_I\bigr).
\]
The first two identities in \eqref{eq:coefficient-orthogonality}, followed
by \eqref{eq:C-covariance-identity} with \(B=f'(A_J)\), yield
\begin{align*}
 \sum_I\abs{NV'(h_J)e_I}^2&\le CN,\\
 \sum_I\left|\sum_a\ell_\eps'(g_{a,J})b_{a,I}\right|^2&\le CN,\\
 \sum_I\abs{\Tr(B C_I)}^2
 &=\frac{2p(p-1)}N\norm{B}_{\HS}^2
   +\frac{p^2}N(\Tr B)^2\le CN.
\end{align*}
Here all scalar derivatives are bounded and
\(\norm{f'(A_J)}_{\op}\le\norm{f'}_\infty\).  The bound on
\(\sum_I\abs{\alpha_I}^2\) follows by applying
\(\abs{x+y+z}^2\le3(x^2+y^2+z^2)\).  The maximum bound on
\(\abs{\alpha_I}\) is the \(k=1\) case of
\cref{lem:energy-gradient-influence,lem:spectral-influence}.

For \(k\ge2\), the energy contribution is bounded by
\[
 C N\abs{e_I}^k
 \le C(Ne_I^2)q_N^{k-2}.
\]
For the gradient contribution, use
\(\max_a\abs{b_{a,I}}\le(\sum_a b_{a,I}^2)^{1/2}\le Cq_N\) to obtain
\[
 C_k\sum_a\abs{b_{a,I}}^k
 \le C_k\left(\sum_a b_{a,I}^2\right)q_N^{k-2}.
\]
For the spectral contribution, the applicable one of
\cref{lem:trace-derivatives,lem:uncut-logdet-derivatives} gives
\[
 C_k\norm{C_I}_{\HS}^2\norm{C_I}_{\op}^{k-2}
 \le C_k\norm{C_I}_{\HS}^2q_N^{k-2}.
\]
This proves \eqref{eq:higher-log-jet-budget}.

Fa\`a di Bruno's formula, \eqref{eq:first-aggregate-budget},
\eqref{eq:dI-budget}, and \eqref{eq:higher-log-jet-budget} imply
\eqref{eq:pointwise-aggregate-F}.  Indeed, suppose a Bell-polynomial
monomial contains \(s\) singleton blocks and \(t\ge1\) higher-order
blocks.  It is bounded by
\[
 C\abs{\alpha_I}^s\mathfrak d_I^tq_N^{k-s-2t}.
\]
With \(x=\abs{\alpha_I}/q_N\) and
\(y=\mathfrak d_I/q_N^2\), both \(x\) and \(y\) are uniformly bounded,
and
\[
 x^sy^t\le C(x^k+y).
\]
This follows by using \(y^t\le Cy\) and splitting into \(x\le1\) and
\(x>1\).  Hence the displayed monomial is at most a constant times
\(\abs{\alpha_I}^k+\mathfrak d_Iq_N^{k-2}\).  The monomial containing
only first derivatives is \(\abs{\alpha_I}^k\).  Finally,
\[
 \sum_I\abs{\alpha_I}^k
 \le(\max_I\abs{\alpha_I})^{k-2}
     \sum_I\abs{\alpha_I}^2
 \le CNq_N^{k-2},
\]
and \eqref{eq:dI-budget} proves \eqref{eq:aggregate-F-budget}.
\end{proof}

For \(I\in[N]^p\), let \(\mu_I\) denote the law of \(J_{I,N}\), let
\(\gamma\) be standard Gaussian measure, and set
\[
 \mu_{I,t}=(1-t)\mu_I+t\gamma,\qquad
 \lambda_t=\bigotimes_{I\in[N]^p}\mu_{I,t},\qquad
 \lambda_{t,-I}=\bigotimes_{I'\ne I}\mu_{I',t}.
\]
For a fixed frame \(O\), put
\[
 A_O(t)=\int F_O(z)\lambda_t(\dd z).
\]

\begin{lemma}[Conditional Taylor remainder and aggregate summation]
\label[lemma]{lem:conditional-aggregate-remainder}
Assume \eqref{eq:array-assumptions}, and write \(r=m+1\).  Fix
\(O\in\cO_{N,L}\), \(t\in[0,1]\), \(I\in[N]^p\), and
\(z_{-I}\in\R^{[N]^p\setminus\{I\}}\).  Define
\[
 \phi_I(s)=F_O(z_{-I},s),\qquad
 a_I^0=\abs{\partial_I\Psi_O(z_{-I},0)},\qquad
 B_I=\int\phi_I(s)\mu_{I,t}(\dd s),
\]
and
\[
 \Delta_I(z_{-I})
 =\int\phi_I(s)\gamma(\dd s)
  -\int\phi_I(s)\mu_I(\dd s).
\]
Then, for all sufficiently large \(N\), the following estimates hold
uniformly in \(O,t,I,z_{-I}\):
\begin{align}
 \abs{\phi_I^{(r)}(s)}
 &\le C\phi_I(s)
 \left\{(a_I^0+\mathfrak d_I\abs s)^r
       +\mathfrak d_Iq_N^{r-2}\right\},
 \label{eq:shifted-r-derivative}\\
 c\phi_I(0)\le B_I&\le C\phi_I(0),\qquad
 \int\abs s^k\phi_I(s)\mu_{I,t}(\dd s)
 \le C_kB_I,\quad 1\le k\le2r,
 \label{eq:conditional-tilt-bounds}\\
 \abs{\Delta_I(z_{-I})}
 &\le CB_I
 \left\{(a_I^0)^r+\mathfrak d_Iq_N^{r-2}\right\}.
 \label{eq:conditional-coordinate-remainder}
\end{align}
Moreover,
\begin{equation}\label{eq:conditional-aggregate-sum}
 \sum_{I\in[N]^p}
 \int\abs{\Delta_I(z_{-I})}\lambda_{t,-I}(\dd z_{-I})
 \le CNq_N^{r-2}A_O(t).
\end{equation}
\end{lemma}

\begin{proof}
For brevity, write \(q=q_N\) and \(d_I=\mathfrak d_I(O)\).  By
\eqref{eq:higher-log-jet-budget} with \(k=2\),
\[
 \abs{\partial_I\Psi_O(z_{-I},s)
       -\partial_I\Psi_O(z_{-I},0)}
 \le Cd_I\abs s.
 \tag{\(\ast\)}
\]
The first bound in \eqref{eq:first-aggregate-budget}, integrated along
the same coordinate line, also gives
\begin{equation}\label{eq:phi-log-lipschitz}
 e^{-Cq\abs s}\phi_I(0)
 \le\phi_I(s)
 \le e^{Cq\abs s}\phi_I(0).
\end{equation}
Combining \((\ast)\) with
\eqref{eq:pointwise-aggregate-F} at derivative order \(r\) proves
\eqref{eq:shifted-r-derivative}.

Choose \(K_\ast\ge K\), depending only on \(K\), so large that
\[
 \int e^{\abs s/K_\ast}\rho(\dd s)\le2
\]
for every
\(\rho\in\{\mu_I,\gamma,\mu_{I,t}:I\in[N]^p,\ t\in[0,1]\}\).
This is possible because the displayed integral is affine under
mixtures.  If \(S\sim\mu_{I,t}\), then
\(\E\abs S\le2K_\ast\).  Consequently, Jensen's inequality and
\eqref{eq:phi-log-lipschitz} yield
\[
 \frac{B_I}{\phi_I(0)}
 \ge\E e^{-Cq\abs S}
 \ge e^{-Cq\E\abs S}
 \ge e^{-2CK_\ast q}\ge c.
\]
For sufficiently large \(N\), \(Cq\le(2K_\ast)^{-1}\).  Hence, for
each fixed \(k\le2r\),
\[
 x^ke^{Cqx}\le C_{k,K_\ast}e^{x/K_\ast},
 \qquad x\ge0.
\]
Using the upper bound in \eqref{eq:phi-log-lipschitz}, first with
\(k=0\) and then with \(1\le k\le2r\), gives
\[
 B_I\le C\phi_I(0),\qquad
 \int\abs s^k\phi_I(s)\mu_{I,t}(\dd s)
 \le C_k\phi_I(0)\le C_k'B_I.
\]
This proves \eqref{eq:conditional-tilt-bounds}.

Let
\[
 P_{I,r-1}(s)
 =\sum_{k=0}^{r-1}\frac{\phi_I^{(k)}(0)}{k!}s^k.
\]
Taylor's formula with integral remainder gives
\[
 \phi_I(s)-P_{I,r-1}(s)
 =\frac{s^r}{(r-1)!}
   \int_0^1(1-u)^{r-1}\phi_I^{(r)}(us)\dd u.
\]
By \eqref{eq:shifted-r-derivative},
\eqref{eq:phi-log-lipschitz}, and
\((x+y)^r\le2^{r-1}(x^r+y^r)\),
\[
 \abs{\phi_I(s)-P_{I,r-1}(s)}
 \le C\phi_I(0)\abs s^re^{Cq\abs s}
 \left\{(a_I^0)^r+d_I^r\abs s^r
                 +d_Iq^{r-2}\right\}.
\]
The exponential-moment estimate above, applied separately to
\(\rho=\mu_I\) and \(\rho=\gamma\), therefore implies
\[
 \int\abs{\phi_I-P_{I,r-1}}\dd\rho
 \le C\phi_I(0)
 \left\{(a_I^0)^r+d_I^r+d_Iq^{r-2}\right\}.
\]
Since \(d_I\le Cq^2\), \(q\le1\), and \(r\ge3\),
\[
 d_I^r
 \le C d_Iq^{2r-2}
 \le C d_Iq^{r-2}.
\]
Moreover, \(r-1=m\), so \(\mu_I\) and \(\gamma\) integrate
\(P_{I,r-1}\) identically by moment matching.  The last two displays
and the lower bound \(B_I\ge c\phi_I(0)\) prove
\eqref{eq:conditional-coordinate-remainder}.

It remains to sum the conditional bounds.  Set
\[
 \mathcal S
 =\sum_I\int B_I(z_{-I})
       (a_I^0(z_{-I}))^r\lambda_{t,-I}(\dd z_{-I}).
\]
Because \(B_I\) is the conditional integral of \(F_O\), Tonelli's
theorem gives the exact reinsertion identity
\[
 \mathcal S
 =\int F_O(z)\sum_I
   (a_I^0(z_{-I}))^r\lambda_t(\dd z).
\]
Applying \((\ast)\) with \(s=z_I\),
\[
 a_I^0(z_{-I})
 \le\abs{\alpha_I(z)}+Cd_I\abs{z_I}.
\]
It follows that
\begin{align*}
 \mathcal S
 &\le C\int F_O(z)\sum_I\abs{\alpha_I(z)}^r
          \lambda_t(\dd z)\\
 &\quad+C\sum_Id_I^r
       \int\abs{z_I}^rF_O(z)\lambda_t(\dd z).
\end{align*}
Pointwise in \(z\), \eqref{eq:first-aggregate-budget} gives
\[
 \sum_I\abs{\alpha_I(z)}^r
 \le\left(\max_I\abs{\alpha_I(z)}\right)^{r-2}
      \sum_I\abs{\alpha_I(z)}^2
 \le CNq^{r-2}.
\]
On the other hand, \eqref{eq:conditional-tilt-bounds} and another
conditional integration imply
\[
 \int\abs{z_I}^rF_O(z)\lambda_t(\dd z)
 \le C_rA_O(t).
\]
Finally, \eqref{eq:dI-budget} yields
\[
 \sum_Id_I^r
 \le(\max_Id_I)^{r-1}\sum_Id_I
 \le CNq^{2r-2}
 \le CNq^{r-2}.
\]
Thus
\begin{equation}\label{eq:a0-aggregate-sum}
 \mathcal S\le CNq^{r-2}A_O(t).
\end{equation}
The remaining term in
\eqref{eq:conditional-coordinate-remainder} satisfies
\[
 \sum_I d_Iq^{r-2}
       \int B_I(z_{-I})\lambda_{t,-I}(\dd z_{-I})
 =q^{r-2}A_O(t)\sum_Id_I
 \le CNq^{r-2}A_O(t).
\]
Combining this estimate with \eqref{eq:a0-aggregate-sum} proves
\eqref{eq:conditional-aggregate-sum}.
\end{proof}

\begin{proposition}[Aggregate product-mixture comparison]
\label[proposition]{prop:aggregate-product-mixture}
Under \eqref{eq:array-assumptions}, uniformly in \(t\in[0,1]\) and
\(O\in\cO_{N,L}\),
\begin{equation}\label{eq:AO-log-derivative}
 \abs{A_O'(t)}\le CNq_N^{m-1}A_O(t).
\end{equation}
\end{proposition}

\begin{proof}
Put \(r=m+1\).  Because the product has finitely many coordinates and every factor
\(\mu_{I,t}\) is affine in \(t\), direct differentiation gives
\[
 A_O'(t)
 =\sum_{I\in[N]^p}
   \int\Delta_I(z_{-I})\lambda_{t,-I}(\dd z_{-I}).
\]
The deterministic envelope for \(f\in\cF_m\), or the polynomial
envelope for \(f=\ell_\eta\), used in the proof of
\cref{prop:integrand-influence} makes all signed-measure integrals
absolutely integrable and justifies both differentiation and the
applications of Fubini's theorem in this identity.
Therefore \eqref{eq:conditional-aggregate-sum} implies
\[
 \abs{A_O'(t)}
 \le CNq_N^{r-2}A_O(t)
 =CNq_N^{m-1}A_O(t),
\]
because \(r=m+1\).
\end{proof}

\section{Proof of bulk universality}\label{sec:bulk-proof}

\begin{proof}[Proof of \cref{thm:bulk-universality}]
Fix \(O\in\cO_{N,L}\).  Integrating
\eqref{eq:AO-log-derivative} from zero to one gives
\[
 e^{-CNq_N^{m-1}}\E F_O(G)
 \le\E F_O(J)
 \le e^{CNq_N^{m-1}}\E F_O(G).
\]
The constants are uniform in \(O\), so integrating these inequalities over
the admissible frames and multiplying by \(\omega_N\) proves
\eqref{eq:bulk-univ-bound}.  Division by \(N\), together with \(q_N\to0\),
proves \eqref{eq:pressure-univ-weak} for every \(p\ge3,m\ge2\).  Under
\eqref{eq:moment-threshold}, \(Nq_N^{m-1}\to0\), which proves
\eqref{eq:ratio-one} and \eqref{eq:pressure-univ}.
\end{proof}

\begin{proof}[Proof of \cref{cor:approx-matching}]
Repeat the one-coordinate Taylor expansion in the proof of \cref{lem:one-coordinate-Lindeberg}.  The terms through order \(m\) no longer cancel, but their total absolute contribution is at most
\[
 C F(0)\sum_{r=1}^m\delta_{I,r,N}a_I^r.
\]
The two order-\(m+1\) remainders are still bounded by \(CF(0)a_I^{m+1}\).  The lower bound on both expectations used in that proof converts this additive estimate into
\[
 \left|\log\frac{\E F(X)}{\E F(Y)}\right|
 \le C\left(a_I^{m+1}+\sum_{r=1}^m\delta_{I,r,N}a_I^r\right)
\]
for all sufficiently large \(N\), uniformly in the frozen coordinates.  Telescope over \(I\in[N]^p\) and use \(a_I=C_\ast q_N\).
\end{proof}

\begin{proof}[Proof of \cref{thm:weibull-bulk}]
Choose \(\gamma\) as in the theorem, put \(T_N=N^\gamma\), and couple
\[
 X_I=J_{I,N}\1_{\{\abs{J_{I,N}}\le T_N\}}.
\]
Uniform Weibull tails and exact moment matching before truncation imply,
for \(1\le k\le m\),
\begin{equation}\label{eq:clipped-moment-defect}
 \delta_{I,k,N}:=\abs{\E X_I^k-\E G^k}
 \le \E\abs{J_{I,N}}^k\1_{\{\abs{J_{I,N}}>T_N\}}
 \le C_ke^{-c_kT_N^\alpha}.
\end{equation}
Polynomial factors in the tail integral have been absorbed by decreasing
\(c_k\).

We first show that truncation is invisible on the annealed logarithmic
scale.  Uniformly in \(O\) and in the disorder vector \(z\),
\begin{equation}\label{eq:FO-deterministic-upper}
 F_O(z)\le e^{C_+N},
\end{equation}
because \(V\) is bounded below, \(\kappa_\eps\) and \(f\) are bounded.
There is also a uniform lower bound
\begin{equation}\label{eq:FO-expectation-lower}
 \min\{\E F_O(J),\E F_O(X)\}\ge e^{-C_-N}
\end{equation}
after increasing \(C_-\).  To verify it separately for either array
\(Y=J\) or \(Y=X\), use
\cref{eq:coefficient-orthogonality,eq:C-covariance-identity}.
The coordinates of \(Y\) have second moments at most one, and their means
are zero for \(J\) and at most \(Ce^{-cT_N^\alpha}\) for \(X\).  Hence,
uniformly in the frame,
\[
 \E h_Y^2\le C,\qquad
 \E\norm{g_Y}_2^2\le CN,\qquad
 \E\norm{A_Y}_{\HS}^2\le CN.
\]
For example, a nonzero mean contributes at most
\(N^p\max_I\abs{\E X_I}^2\) times the corresponding coefficient square
sum, and is negligible.  Markov's inequality therefore gives, with
probability bounded below uniformly in \(N,O\),
\[
 \abs{h_Y}\le R,\qquad
 \norm{g_Y}_2^2\le RN,\qquad
 \norm{A_Y}_{\HS}^2\le RN
\]
for a fixed \(R\).  On this event the energy term is at least \(-CN\),
and
\[
 \sum_{a=1}^n\log\kappa_\eps(g_{a,Y})
 =n\log\frac{\eps}{\pi}
  -\sum_{a=1}^n\log(g_{a,Y}^2+\eps^2)\ge-CN
\]
by Jensen's inequality.  Also
\(\Tr f(A_Y)\ge-n\norm f_\infty\), proving
\eqref{eq:FO-expectation-lower}.

The coupling, \eqref{eq:FO-deterministic-upper}, and a union bound give
\[
 \abs{\E F_O(J)-\E F_O(X)}
 \le Ce^{C_+N}N^pe^{-cT_N^\alpha}.
\]
Since \(\alpha\gamma>1\), division by
\eqref{eq:FO-expectation-lower} and
\(\abs{\log u-\log v}\le\abs{u-v}/\min(u,v)\) yield
\begin{equation}\label{eq:truncate-relative}
 \abs{\log\E F_O(J)-\log\E F_O(X)}
 \le Ce^{-cN^{\alpha\gamma}},
\end{equation}
uniformly in \(O\).

It remains to compare \(X\) with \(G\).  Repeat the product-mixture proof
of \cref{prop:aggregate-product-mixture} with
\(\mu_I=\operatorname{Law}(X_I)\).  The only two changes are as follows.
First,
\[
 q_NT_N
 =L^{p-1}N^{\gamma-(p-2)/2}
   (\log N)^{(p-1)/2}\longrightarrow0.
\]
Thus the clipped part of every conditional tilted moment in
\eqref{eq:conditional-tilt-bounds} is uniformly bounded by using
\(\phi_I(s)/\phi_I(0)\le e^{Cq_NT_N}\); the Gaussian part is unchanged.
Second, the Taylor polynomials no longer cancel exactly.  By
\eqref{eq:FO-derivative}, their total defect in coordinate \(I\) is at
most
\[
 C\phi_I(0)\sum_{k=1}^m\delta_{I,k,N}q_N^k.
\]
The order-\(m+1\) remainder and its aggregate summation are unchanged,
because the clipped variables have uniformly bounded fixed moments and
\(q_NT_N=o(1)\).  Consequently, with
\[
 \Delta_N=\sum_{I\in[N]^p}\sum_{k=1}^m
 \delta_{I,k,N}q_N^k,
\]
the analogue of \eqref{eq:AO-log-derivative} is
\[
 \abs{A_O'(t)}
 \le C\left(Nq_N^{m-1}+\Delta_N\right)A_O(t).
\]
By \eqref{eq:clipped-moment-defect},
\[
 \Delta_N\le CN^pe^{-cT_N^\alpha}
 \le Ce^{-c'N^{\alpha\gamma}}.
\]
Integrating in \(t\), combining with
\eqref{eq:truncate-relative}, and then integrating the uniform framewise
comparison over either admissible frame set proves
\eqref{eq:weibull-bulk-bound}.  The remaining conclusions follow exactly
as in the proof of \cref{thm:bulk-universality}.
\end{proof}

\begin{remark}[Why the improved exponent is natural]
The worst one-coordinate influence is \(O(q_N)\), but there are not
\(N^p\) independent worst directions: orthogonality of the Gaussian jet
gives the aggregate budget
\(\sum_I\abs{\partial_I\Psi_O}^2=O(N)\).  A simultaneous product-mixture
path keeps all coordinates under one tilted product law and preserves this
budget.  After matching \(m\) moments, the remaining derivative order is
\(m+1\), so the total relative error is \(O(Nq_N^{m-1})\).  Sequential
replacement remains useful for the approximate-moment statement, where
arbitrary low-order mismatches need not enjoy an aggregate cancellation.
\end{remark}

\section{Gaussian evaluation and the variational formula}\label{sec:gaussian}

\subsection{Reduction to one frame}

Because the Gaussian field is isotropic, the expectation of the integrand in \eqref{eq:Z-def} is independent of \(O\).  By \cref{prop:Gaussian-jet},
\begin{align}
 \E\cZ_{N,L}^G(V,\eps,f)
 ={}&\omega_N\nu_N(\cO_{N,L})a_{\eps,p}^{n}
 \int_\R \sqrt{\frac{N}{2\pi}}e^{-Nu^2/2-NV(u)}\notag\\
 &\hspace{18mm}\times
 \E\exp\left\{\Tr f(\gamma_{N,p}M_n-puI_n)\right\}\dd u.
 \label{eq:Gaussian-Z-exact}
\end{align}

\subsection{GOE linear statistics at exponential scale \texorpdfstring{$N$}{N}}

\begin{lemma}[Annealed GOE linear statistic]\label[lemma]{lem:GOE-linear-stat}
Let \(f\in\cF_m\), or let \(f=\ell_\eta\) for a fixed \(\eta>0\), and
define
\[
 L_{N,f}(u)=\Tr f(\gamma_{N,p}M_n-puI_n).
\]
For every compact \(K\subset\R\),
\begin{equation}\label{eq:linear-stat-uniform}
 \sup_{u\in K}\left|
 \frac1N\log\E e^{L_{N,f}(u)}-\ell_{p,f}(u)
 \right|\longrightarrow0.
\end{equation}
Moreover, uniformly in \(u\in\R\),
\begin{equation}\label{eq:linear-stat-global-bound}
 \frac1N\log\E e^{L_{N,f}(u)}
 \le
 \begin{cases}
  \norm f_\infty,&f\in\cF_m,\\
  C_{p,\eta}+\log(2+\abs u),&f=\ell_\eta.
 \end{cases}
\end{equation}
\end{lemma}

\begin{proof}
For \(f\in\cF_m\), the global bound is immediate from
\(\abs{L_{N,f}(u)}\le n\norm f_\infty\).

Represent the GOE by independent standard Gaussian coordinates with the usual \(n^{-1/2}\) scaling.  The map from these coordinates to \(L_{N,f}(u)\) has Euclidean Lipschitz constant bounded uniformly in \(N,u\): indeed, the matrix gradient is \(\gamma_{N,p}f'(\gamma_{N,p}M_n-puI_n)\), whose Hilbert--Schmidt norm is at most \(C\sqrt n\norm{f'}_\infty\), while the coordinate-to-matrix map has scale \(n^{-1/2}\).  Gaussian concentration therefore yields
\begin{equation}\label{eq:linstat-conc}
 \log\E\exp\{L_{N,f}(u)-\E L_{N,f}(u)\}\le C_f
 \end{equation}
with a constant independent of \(N,u\).  Jensen's inequality and
\eqref{eq:linstat-conc} give
\begin{equation}\label{eq:logmgf-mean}
 0\le\log\E e^{L_{N,f}(u)}-\E L_{N,f}(u)\le C_f.
\end{equation}
Let
\[
 \overline\rho_n=\E\frac1n\sum_{i=1}^n\delta_{\lambda_i(M_n)}.
\]
Then
\[
 \frac1N\E L_{N,f}(u)
 =\frac nN\int f(\gamma_{N,p}\lambda-pu)\overline\rho_n(\dd\lambda).
\]
The semicircle law gives \(\overline\rho_n\Rightarrow\rhoSC\), while
\(n/N\to1\) and
\(\gamma_{N,p}\to\sqrt{2p(p-1)}\).  For \(f=\ell_\eta\), uniform
integrability follows from
\(\ell_\eta(t)\le C_\eta+\log(1+\abs t)\) and the uniformly bounded
second moments of \(\overline\rho_n\).  Hence the last display converges
pointwise to \(\ell_{p,f}(u)\) in both cases.  Since \(f'\) is bounded,
the two sides are equi-Lipschitz in \(u\), so the convergence is
uniform on compact sets.  Dividing \eqref{eq:logmgf-mean} by \(N\)
proves \eqref{eq:linear-stat-uniform}.

It remains to prove the uncut global bound.  Put
\(B=\gamma_{N,p}M_n-puI_n\).  The arithmetic--geometric mean
inequality and then Jensen's inequality give
\[
 \frac1n\E\Tr\ell_\eta(B)
 \le\frac12\E\log\left\{\frac1n\Tr(B^2+\eta^2I_n)\right\}
 \le\frac12\log\E\left\{\frac1n\Tr(B^2+\eta^2I_n)\right\}.
\]
The last expectation is at most
\(C_{p,\eta}(1+u^2)\).  Combining this with
\eqref{eq:logmgf-mean} proves the second line of
\eqref{eq:linear-stat-global-bound}.
\end{proof}

\subsection{Laplace principle}

\begin{lemma}[Sphere-volume asymptotics]\label[lemma]{lem:sphere-volume}
\begin{equation}\label{eq:sphere-volume}
 \lim_{N\to\infty}\frac1N\log\omega_N
 =\frac12\log(2\pi e).
\end{equation}
\end{lemma}

\begin{proof}
Using
\[
 \omega_N=\frac{2\pi^{N/2}N^{(N-1)/2}}{\Gamma(N/2)}
\]
and Stirling's formula, we obtain the result.
\end{proof}

\begin{proof}[Proof of \cref{thm:gaussian-variational}]
Take logarithms in \eqref{eq:Gaussian-Z-exact} and divide by \(N\).  By \cref{lem:Haar-incoherent,lem:sphere-volume}, the first two geometric factors contribute \(\frac12\log(2\pi e)+o(1)\).  The gradient term contributes \((n/N)\log a_{\eps,p}\to\log a_{\eps,p}\).

For the remaining integral, \cref{lem:GOE-linear-stat} gives uniform convergence of its logarithmic integrand on compact sets to
\[
 -\frac{u^2}{2}-V(u)+\ell_{p,f}(u).
\]
Since \(V\) is bounded below, \eqref{eq:linear-stat-global-bound} and
the Gaussian term \(-u^2/2\) give uniform exponential tightness of the
integral outside compact sets, also for \(f=\ell_\eta\).  The
one-dimensional Laplace principle therefore yields the supremum in
\eqref{eq:Gaussian-variational}.  The non-Gaussian conclusions follow
from the corresponding comparison theorems.
\end{proof}

\section{Determinant de-regularization and the Gaussian complexity potential}\label{sec:dereg}

\subsection{Uncut and cutoff determinant approximations}

The uncut smoothed logarithm \(\ell_\eta\) was defined in
\eqref{eq:ell-eta}.  For comparison with a bounded spectral weight, fix
an even \(\chi\in C_c^\infty(\R;[0,1])\) such that \(\chi=1\) on
\([-1,1]\) and \(\chi=0\) outside \([-2,2]\), set
\(\chi_R(t)=\chi(t/R)\), and put
\begin{equation}\label{eq:f-eta-R}
 f_{\eta,R}(t)=\chi_R(t)\ell_\eta(t)
 +(1-\chi_R(t))\ell_\eta(R).
\end{equation}
Then \(f_{\eta,R}-\ell_\eta(R)\in C_c^\infty(\R)\), so \(f_{\eta,R}\in\cF_m\) for every fixed \(\eta,R,m\), and
\begin{equation}\label{eq:f-limit}
 \lim_{R\to\infty}\lim_{\eta\downarrow0}f_{\eta,R}(t)=\log\abs t
\end{equation}
for every \(t\ne0\).

\begin{theorem}[De-regularized Gaussian variational formula]\label[theorem]{thm:dereg-variational}
Let \(V\) be continuous and bounded below.  Then
\begin{align}
 &\lim_{(\eps,\eta)\to(0,0)}
 \left[\frac12\log(2\pi e)+\log a_{\eps,p}
 +\sup_u\left\{-\frac{u^2}{2}-V(u)
 +\ell_{p,\ell_\eta}(u)\right\}\right]\notag\\
 &\hspace{35mm}=\sup_{u\in\R}\{\theta_p(u)-V(u)\},
 \label{eq:uncut-dereg-limit}
\end{align}
where the limit is joint.  The bounded cutoff approximation also
satisfies
\begin{align}
 &\lim_{R\to\infty}\lim_{\eta\downarrow0}\lim_{\eps\downarrow0}
 \left[\frac12\log(2\pi e)+\log a_{\eps,p}
 +\sup_u\left\{-\frac{u^2}{2}-V(u)+\ell_{p,f_{\eta,R}}(u)\right\}\right]
 \notag\\
 &\hspace{35mm}=\sup_{u\in\R}\{\theta_p(u)-V(u)\}.
 \label{eq:dereg-limit}
\end{align}
\end{theorem}

\begin{proof}
Since \(\kappa_\eps\) is an approximate identity and the \(\cN(0,p)\) density is continuous at zero,
\begin{equation}\label{eq:a-eps-limit}
 a_{\eps,p}\longrightarrow\frac1{\sqrt{2\pi p}}.
\end{equation}
For \(0<\eta\le1\) and \(R\ge1\), both the uncut and cutoff weights give
the bounds
\begin{equation}\label{eq:regularizer-uniform-upper}
 f_{\eta,R}(t)\le C+\log(1+\abs t),
 \qquad
 \max\{\ell_{p,f_{\eta,R}}(u),\ell_{p,\ell_\eta}(u)\}
 \le C+\log(2+\abs u),
\end{equation}
with \(C\) independent of \(\eta,R,u\).  Consequently, each objective
function appearing in \eqref{eq:uncut-dereg-limit} or
\eqref{eq:dereg-limit} is bounded above by
\[
 -\frac{u^2}{2}+C\log(2+\abs u)-\inf V,
\]
while the corresponding suprema have a uniform finite lower bound.  Indeed,
\[
 f_{\eta,R}(t)\ge\min\{\log\abs t,0\},
 \qquad 0<\eta\le1,\quad R\ge1,
\]
and \(\ell_\eta(t)\ge\log\abs t\) for \(t\ne0\).
Evaluating at \(u=0\) gives a lower bound independent of \(\eta,R\),
because the logarithmic singularity is integrable against the semicircle
law.  Hence all near-maximizers lie in a common compact set.

For fixed \(u\), dominated convergence against the compactly supported semicircle law gives
\[
 \lim_{R\to\infty}\lim_{\eta\downarrow0}\ell_{p,f_{\eta,R}}(u)
 =\int\log\abs{\sqrt{2p(p-1)}\lambda-pu}\rhoSC(\dd\lambda).
\]
The logarithmic singularity is integrable.  More precisely, if
\(K\subset\R\) is compact and
\[
 R>\sup\left\{\abs{\sqrt{2p(p-1)}\lambda-pu}:
 u\in K,\ \abs\lambda\le\sqrt2\right\},
\]
then the cutoff is inactive and the bounded density of the corresponding
semicircle pushforward gives
\[
 \sup_{u\in K}\left|\int \ell_\eta(\sqrt{2p(p-1)}\lambda-pu)\rhoSC(\dd\lambda)
 -\int\log\abs{\sqrt{2p(p-1)}\lambda-pu}\rhoSC(\dd\lambda)\right|
 \le C_K\eta.
\]
The same estimate without choosing \(R\) gives
\[
 \sup_{u\in K}\left|
 \ell_{p,\ell_\eta}(u)
 -\int\log\abs{\sqrt{2p(p-1)}\lambda-pu}\rhoSC(\dd\lambda)
 \right|\le C_K\eta.
\]
Thus both convergences are locally uniform.  Together with
\eqref{eq:regularizer-uniform-upper}, this proves convergence of the
suprema, jointly in \((\eps,\eta)\) for the uncut family.  Combining
\eqref{eq:a-eps-limit} with \(\frac12\log(2\pi e)\) gives
\(\frac12-\frac12\log p\), which is exactly
\eqref{eq:theta-logpotential}.
\end{proof}

\subsection{Evaluation of the semicircle logarithmic potential}

Let
\begin{equation}\label{eq:Usc}
 U(y)=\int\log\abs{\lambda-y}\rhoSC(\dd\lambda).
\end{equation}

\begin{lemma}[Semicircle log potential]\label[lemma]{lem:semicircle-log-potential}
For \(\abs y\le\sqrt2\),
\begin{equation}\label{eq:U-inside}
 U(y)=\frac{y^2}{2}-\frac12-\frac12\log2.
\end{equation}
For \(\abs y>\sqrt2\),
\begin{equation}\label{eq:U-outside-derivative}
 U'(y)=y-\operatorname{sgn}(y)\sqrt{y^2-2}.
\end{equation}
Consequently, relative to its quadratic continuation from the edge,
\begin{equation}\label{eq:U-edge-cost}
 U(y)=\frac{y^2}{2}-\frac12-\frac12\log2
 -\int_{\sqrt2}^{\abs y}\sqrt{z^2-2}\,\dd z,
 \qquad \abs y>\sqrt2.
\end{equation}
\end{lemma}

\begin{proof}
The Stieltjes transform in our normalization is
\[
 m(z)=\int\frac1{z-\lambda}\rhoSC(\dd\lambda)
 =z-\sqrt{z^2-2},
\]
where the branch is chosen so that \(m(z)\sim z^{-1}\) at infinity.  Since
\[
 U'(y)=\int\frac1{y-\lambda}\rhoSC(\dd\lambda),
\]
the principal value inside the support is \(y\), while outside it equals
\(y-\operatorname{sgn}(y)\sqrt{y^2-2}\).  Equivalently, the derivative of the difference between \(U(y)\) and \(y^2/2\) is zero inside and equals \(-\operatorname{sgn}(y)\sqrt{y^2-2}\) outside.  A direct beta-integral at zero gives
\[
 U(0)=\int\log\abs\lambda\rhoSC(\dd\lambda)
 =-\frac12-\frac12\log2.
\]
Integrating the derivatives gives the stated formulas.
\end{proof}

\begin{proof}[Proof of \cref{thm:theta}]
Factor
\[
 \sqrt{2p(p-1)}\lambda-pu
 =\sqrt{2p(p-1)}\left(\lambda-u\sqrt{\frac{p}{2(p-1)}}\right).
\]
Apply \cref{lem:semicircle-log-potential} with
\[
 y=u\sqrt{\frac{p}{2(p-1)}}.
\]
The edge \(\abs y=\sqrt2\) is exactly \(\abs u=E_\infty\).  Substitution into \eqref{eq:theta-logpotential} gives the quadratic part in \eqref{eq:theta-explicit}; changing variables in the outside integral gives \eqref{eq:Ip-def}.

For the dual formula, \(\Phi_p^{\rm bulk}(V)+V(u)\ge\theta_p(u)\) for every \(V\).  Conversely, choose smooth truncated quadratic penalties \(V_M(x)\) with \(V_M(u)=0\), \(V_M(x)=M(x-u)^2\) on \(\abs{x-u}\le1\), and \(V_M\) constant outside a slightly larger interval.  Continuity and the quadratic decay of \(\theta_p\) imply
\[
 \sup_x\{\theta_p(x)-V_M(x)\}\downarrow\theta_p(u).
\]
Taking the infimum proves \eqref{eq:dual-theta}.
\end{proof}

\begin{corollary}[Large-energy asymptotics]\label[corollary]{cor:theta-large}
As \(\abs u\to\infty\),
\begin{equation}\label{eq:theta-large}
 \theta_p(u)=-\frac{u^2}{2}+\log(p\abs u)
 +\frac12-\frac12\log p+O(u^{-2}).
\end{equation}
In particular, \(\theta_p(u)=-u^2/2+O(\log\abs u)\).
\end{corollary}

\begin{proof}
Uniformly for \(\abs\lambda\le\sqrt2\),
\[
 \log\abs{pu-\sqrt{2p(p-1)}\lambda}
 =\log(p\abs u)+O(u^{-1}),
\]
and the first-order term integrates to zero by symmetry; the remainder is \(O(u^{-2})\).  Substitute in \eqref{eq:theta-logpotential}.
\end{proof}

\subsection{An index-generating extension}

The bulk theorem applies to compactly supported smooth spectral weights.  Let
\(\chi_{\delta,R}\in C_c^\infty(\R;[0,1])\) equal one on
\([-R,-\delta]\), vanish on \([\delta,\infty)\) and outside
\([-2R,2R]\), and interpolate smoothly elsewhere.  Replace \(f\) by
\begin{equation}\label{eq:index-f}
 f_{\eta,R,s,\delta}(t)=f_{\eta,R}(t)+s\chi_{\delta,R}(t).
\end{equation}
Then \(\Tr f_{\eta,R,s,\delta}\) is a regularized log determinant plus
\(s\) times a truncated, smoothed index.  Define
\[
 \alpha_p(u)=\int
 \1_{\{\sqrt{2p(p-1)}\lambda-pu<0\}}\rhoSC(\dd\lambda).
\]

\begin{theorem}[Regularized extensive-index generating pressure]\label[theorem]{thm:index-generating}
Fix \(s\in\R\), \(L>L_0\), and a continuous \(V\) bounded below.  Then
\begin{align}
 &\lim_{R\to\infty}\lim_{\delta\downarrow0}
 \lim_{\eta\downarrow0}\lim_{\eps\downarrow0}
 \lim_{N\to\infty}
 \Phi_{N,L}^{G}(V,\eps,f_{\eta,R,s,\delta})\notag\\
 &\qquad=\sup_u\{\theta_p(u)-V(u)+s\alpha_p(u)\}.
 \label{eq:index-generating-formal}
\end{align}
If \(V\in\cV_m\) and the array \(J\) satisfies the hypotheses of either
\cref{thm:bulk-universality} or \cref{thm:weibull-bulk}, the same
iterated limit holds with \(G\) replaced by \(J\).
\end{theorem}

\begin{proof}
For fixed \(\eta,R,\delta\), apply \cref{thm:gaussian-variational} to \eqref{eq:index-f}.  The index summand is bounded by \(\abs s\), so the common-compact near-maximizer argument in the proof of \cref{thm:dereg-variational} is unchanged.  On every compact \(u\)-set, the bounded density of the affine semicircle pushforward shows that smoothing the sign change contributes \(O(\delta)\), uniformly in \(u\); taking \(R\to\infty\) removes the remote cutoff.  The determinant and delta limits are those of \cref{thm:dereg-variational}.  This proves \eqref{eq:index-generating-formal}.
For every fixed regularizer, the applicable one of
\cref{thm:bulk-universality,thm:weibull-bulk} transfers the limiting
pressure, proving the non-Gaussian statement.
\end{proof}

\begin{remark}
The theorem concerns a regularized index-generating functional with the \(N\to\infty\) limit taken first.  Identifying it with exact non-Gaussian index counts, or controlling fixed index, still requires exponential de-regularization and edge estimates.
\end{remark}

\section{Localized obstructions and unregularized non-universality}
\label{sec:tail}

\subsection{A deterministic cap lemma}

\begin{lemma}[A large diagonal monomial creates a local maximum]\label[lemma]{lem:cap}
Let \(r:\mathbb S^{N-1}(1)\to\R\) be continuous with \(\norm r_\infty\le M\), and let
\begin{equation}\label{eq:cap-function}
 F(x)=a x_1^p+r(x),
 \qquad a>2M.
\end{equation}
Choose \(\delta\in(0,1)\) so that
\begin{equation}\label{eq:cap-gap}
 a\bigl[1-(1-\delta)^p\bigr]>2M.
\end{equation}
Then \(F\) attains a local maximum at an interior point of the cap
\[
 \cK_\delta=\{x\in\mathbb S^{N-1}(1):x_1\ge1-\delta\}.
\]
If \(r\) is \(C^1\), this point is a critical point of \(F\) on the sphere.  Its value lies in \([a-M,a+M]\).
\end{lemma}

\begin{proof}
At the pole \(e_1\), \(F(e_1)\ge a-M\).  On the boundary of the cap,
\[
 F(x)\le a(1-\delta)^p+M<a-M
\]
by \eqref{eq:cap-gap}.  Therefore a maximizer of \(F\) on the compact cap cannot lie on its boundary.  It is an interior local maximum on the sphere and hence, when \(r\) is differentiable, a critical point.  Since \(0\le x_1^p\le1\) throughout the cap, its value is at most \(a+M\), while it is at least \(F(e_1)\ge a-M\).
\end{proof}

\subsection{Tightness of the remainder for sub-Gaussian disorder}

Separate one diagonal entry:
\begin{equation}\label{eq:spike-decomposition}
 h_N^J(x)=\frac{J_{1\cdots1}}{\sqrt N}x_1^p+R_N(x),
 \qquad
 R_N(x)=N^{-1/2}\sum_{I\ne(1,\ldots,1)}J_Ix_I.
\end{equation}

\begin{lemma}[Uniform sub-Gaussian remainder bound]\label[lemma]{lem:remainder-tight}
Suppose, without requiring independence, that the coefficient vector
satisfies the joint linear-form estimate
\begin{equation}\label{eq:joint-K-subg}
 \E\exp\left\{\sum_{I\in[N]^p}t_IJ_I\right\}
 \le\exp\left\{\frac{K^2}{2}\sum_{I\in[N]^p}t_I^2\right\},
 \qquad t\in\R^{N^p}.
\end{equation}
Then
\begin{equation}\label{eq:remainder-expectation}
 \E\sup_{x\in\mathbb S^{N-1}(1)}\abs{R_N(x)}\le C_{p,K}
\end{equation}
uniformly in \(N\).  Consequently, there is \(M=M(p,K)\) such that
\begin{equation}\label{eq:remainder-prob}
 \inf_N\Pp\left(\norm{R_N}_\infty\le M\right)\ge\frac12.
\end{equation}
\end{lemma}

\begin{proof}
The joint bound \eqref{eq:joint-K-subg} gives the standard
Orlicz-norm estimate
\[
 \left\|\sum_I c_IJ_I\right\|_{\psi_2}
 \le CK\left(\sum_Ic_I^2\right)^{1/2}
\]
and shows that \(R_N\) is sub-Gaussian with increment metric
\begin{align*}
 d(x,y)
 &\le \frac{CK}{\sqrt N}
 \left(\sum_I(x_I-y_I)^2\right)^{1/2}\\
 &=\frac{CK}{\sqrt N}\norm{x^{\tensor p}-y^{\tensor p}}_2
 \le\frac{CKp}{\sqrt N}\norm{x-y}_2.
\end{align*}
The last inequality follows by telescoping the tensor product and using \(\norm{x}=\norm{y}=1\).  Consequently,
\[
 \log\Ncov(\mathbb S^{N-1}(1),d,\varepsilon)
 \le N\log\left(1+\frac{CKp}{\sqrt N\,\varepsilon}\right).
\]
Since \(R_N(e_1)=0\), Dudley's entropy bound \cite[Chapter~1]{TalagrandChaining} gives
\begin{align*}
 \E\sup_xR_N(x)
 &\le C\int_0^{Cp/\sqrt N}
 \sqrt{N\log\left(1+\frac{C}{\sqrt N\,\varepsilon}\right)}\dd\varepsilon
 \le C_{p,K}.
\end{align*}
The last estimate follows after the change of variables \(\varepsilon=(CKp/\sqrt N)t\).
Apply the same argument to \(-R_N\), and use
\(\sup_x\abs{R_N(x)}\le\sup_xR_N(x)+\sup_x(-R_N(x))\), to obtain \eqref{eq:remainder-expectation}.  Markov's inequality with a suitable \(M=O(C_{p,K})\) gives \eqref{eq:remainder-prob}.
\end{proof}

\subsection{A mesoscopic coherent-block branch}

The next lower bound uses only coefficients of fixed size.  Its
speed-\(N\) cost comes from asking order \(N\) entries in a block on
order \(N^{1/p}\) coordinates to align.

\begin{theorem}[Mesoscopic coherent-block lower bound]
\label[theorem]{thm:mesoscopic-block-lower}
Suppose the ordered entries are i.i.d., centered, and
\(K\)-sub-Gaussian with common law \(\mu\).  Let \(b,w,c>0\), put
\[
 \pi=\mu([b,b+w]),\qquad
 A=bc^{p/2},\qquad D=wc^{p/2},
\]
and assume \(\pi>0\).  There is \(M=M(p,K)<\infty\) with the following
property.  If some \(t\in(0,1)\) satisfies
\begin{equation}\label{eq:mesoscopic-cap-condition}
 A(1-t^p)>2(M+D),
\end{equation}
then, for
\begin{equation}\label{eq:mesoscopic-window}
 B=[A-(M+D+1),A+(M+D+1)],
\end{equation}
one has
\begin{equation}\label{eq:mesoscopic-block-lower}
 \liminf_{N\to\infty}\frac1N\log
 \E_\mu\Crt_{N,\max}^J(B)
 \ge c^p\log\pi.
\end{equation}
The same lower bound therefore holds for the total critical-point
count.
\end{theorem}

\begin{proof}
Let
\[
 k_N=\lfloor cN^{1/p}\rfloor,\qquad
 S_N=\{1,\ldots,k_N\},\qquad
 u_N=k_N^{-1/2}\sum_{i\in S_N}e_i.
\]
For all sufficiently large \(N\), \(1\le k_N\le N\), and
\begin{equation}\label{eq:block-scaling}
 \frac{k_N^p}{N}\longrightarrow c^p,\qquad
 A_N:=\frac{bk_N^{p/2}}{\sqrt N}\longrightarrow A,\qquad
 D_N:=\frac{wk_N^{p/2}}{\sqrt N}\longrightarrow D.
\end{equation}
Delete the full ordered block \(S_N^p\) and write
\[
 R_{N,S}(x)=\frac1{\sqrt N}\sum_{I\notin S_N^p}J_Ix_I.
\]
The canonical increment metric of this process is bounded by the one
in the proof of \cref{lem:remainder-tight}, and
\(R_{N,S}(u_N)=0\).  The same Dudley and Markov argument therefore
gives a constant \(M=M(p,K)\), independent of \(N,c\), such that
\begin{equation}\label{eq:deleted-block-remainder}
 \Pp\bigl(\norm{R_{N,S}}_\infty\le M\bigr)\ge\frac12.
\end{equation}

Consider the block-alignment event
\[
 \mathcal E_N=\{J_I\in[b,b+w]\ \text{for every }I\in S_N^p\}.
\]
It is independent of \(R_{N,S}\) and has probability
\(\pi^{k_N^p}\).  On \(\mathcal E_N\), decompose
\begin{align}
 h_N^J(x)
 &=A_N\ip{x}{u_N}^{p}+\widetilde R_N(x),\label{eq:block-decomposition}\\
 \widetilde R_N(x)
 &=R_{N,S}(x)
   +\frac1{\sqrt N}\sum_{I\in S_N^p}(J_I-b)x_I.\notag
\end{align}
By Cauchy--Schwarz,
\[
 \sup_{\norm x_2=1}
 \sum_{I\in S_N^p}|x_I|
 =\sup_{\norm x_2=1}\left(\sum_{i\in S_N}|x_i|\right)^p
 \le k_N^{p/2}.
\]
Consequently, on the intersection of \(\mathcal E_N\) and the event in
\eqref{eq:deleted-block-remainder},
\begin{equation}\label{eq:block-error}
 \norm{\widetilde R_N}_\infty\le M+D_N.
\end{equation}

Take the spherical cap
\[
 \mathcal K_{N,t}
 =\{x\in\mathbb S^{N-1}(1):\ip{x}{u_N}\ge t\}.
\]
The profile \(P_N(x)=\ip{x}{u_N}^p\) equals one at \(u_N\), equals
\(t^p\) on the boundary, and is at most one throughout the cap.
By \eqref{eq:mesoscopic-cap-condition} and \eqref{eq:block-scaling},
for all large \(N\),
\[
 A_N(1-t^p)>2(M+D_N).
\]
Apply \cref{lem:strict-profile} to
\eqref{eq:block-decomposition}.  It produces an interior local maximum
whose value belongs to
\[
 [A_N-(M+D_N),A_N+(M+D_N)]\subset B
\]
for all sufficiently large \(N\).  Hence
\[
 \E_\mu\Crt_{N,\max}^J(B)
 \ge\frac12\pi^{k_N^p}.
\]
Taking logarithms, dividing by \(N\), and using
\eqref{eq:block-scaling} proves \eqref{eq:mesoscopic-block-lower}.
\end{proof}

\begin{corollary}[Quantitative separation from the mesoscopic block]
\label[corollary]{cor:block-growing-cutoff-separation}
Retain the notation of \cref{thm:mesoscopic-block-lower}.  Every
\(x\in\mathcal K_{N,t}\) satisfies
\begin{equation}\label{eq:block-cap-coordinate-lower}
 \norm{x}_\infty\ge\frac{t}{\sqrt{k_N}}
 =\left(\frac{t}{\sqrt c}+o(1)\right)N^{-1/(2p)}.
\end{equation}
Consequently, if a point cutoff \(D_{N,r_N}\) from
\cref{thm:growing-point-cutoff} intersects \(\mathcal K_{N,t}\), then
\begin{equation}\label{eq:block-cutoff-influence-lower}
 q_N^{\rm gr}=\sqrt N\,r_N^{p-1}
 \ge
 \left(t^{p-1}c^{-(p-1)/2}+o(1)\right)N^{1/(2p)}.
\end{equation}
In particular, every growing cutoff for which
\(q_N^{\rm gr}=O(1)\), and hence every cutoff in the pressure-universal
regime \(q_N^{\rm gr}\to0\), is eventually disjoint from the entire cap
\(\mathcal K_{N,t}\) that creates the coherent-block critical point.
\end{corollary}

\begin{proof}
For \(x\in\mathcal K_{N,t}\), H\"older's inequality gives
\[
 t\le\ip{x}{u_N}
 =\frac1{\sqrt{k_N}}\sum_{i\in S_N}x_i
 \le\frac1{\sqrt{k_N}}\sum_{i\in S_N}\abs{x_i}
 \le\sqrt{k_N}\norm{x}_\infty.
\]
This proves \eqref{eq:block-cap-coordinate-lower}.  If
\(x\in D_{N,r_N}\cap\mathcal K_{N,t}\), then
\(r_N\ge t/\sqrt{k_N}\).  Raising this inequality to the power
\(p-1\), multiplying by \(\sqrt N\), and using
\(k_N=cN^{1/p}+o(N^{1/p})\) proves
\eqref{eq:block-cutoff-influence-lower}.
\end{proof}

\begin{lemma}[Square-free block approximation]
\label[lemma]{lem:squarefree-rank-one}
Let \(S\subset[N]\) have cardinality \(k\ge p\), put
\[
 v_S=k^{-1/2}\sum_{i\in S}e_i,\qquad
 P_{k,S}(x)=\frac{p!}{k^{p/2}}
 \sum_{\substack{A\subset S\\ |A|=p}}\prod_{i\in A}x_i.
\]
Then
\begin{equation}\label{eq:squarefree-profile-error}
 \sup_{x\in\mathbb S^{N-1}(1)}
 \abs{P_{k,S}(x)-\ip{x}{v_S}^p}
 \le\frac{\binom p2}{k}.
\end{equation}
\end{lemma}

\begin{proof}
In the expansion of \((\sum_{i\in S}x_i)^p\), the ordered tuples with
distinct indices contribute
\[
 p!\sum_{\substack{A\subset S\\ |A|=p}}\prod_{i\in A}x_i.
\]
For any fixed pair
of positions, the absolute sum of tuples in which those positions
coincide is at most
\[
 \left(\sum_{i\in S}x_i^2\right)
 \left(\sum_{i\in S}|x_i|\right)^{p-2}
 \le k^{(p-2)/2}.
\]
A union bound over the \(\binom p2\) pairs bounds the contribution of
all colliding tuples.  Divide by \(k^{p/2}\).
\end{proof}

\subsection{Quadratic-scale tail cost and the one-spike branch}

For \(a>0\) and \(\eta>0\), define the upper quadratic-scale local cost
\begin{equation}\label{eq:tail-cost}
 \overline I_\mu(a,\eta)
 =\limsup_{N\to\infty}-\frac1N
 \log\Pp_\mu\left(\frac\xi{\sqrt N}\in[a,a+\eta]\right).
\end{equation}

\begin{theorem}[One-spike lower bound]\label[theorem]{thm:one-spike}
Assume that \(J_{1\cdots1}\) has law \(\mu\).  Fix \(M<\infty\),
\(a>2M\), and \(\eta>0\).  Assume that the spike event below has
positive probability for all sufficiently large \(N\) and that
\begin{equation}\label{eq:conditional-remainder-cost}
 \lim_{N\to\infty}\frac1N
 \log\Pp\left(
 \left.\norm{R_N}_\infty\le M\,\right|\,
 \frac{J_{1\cdots1}}{\sqrt N}\in[a,a+\eta]\right)=0.
\end{equation}
No assumptions of independence, centering, identical distribution, or
existence of a density are needed beyond this conditional-rate
hypothesis.  Then
\begin{equation}\label{eq:one-spike-max-lower}
 \liminf_{N\to\infty}\frac1N\log
 \E\Crt_{N,\max}^J([a-M,a+\eta+M])
 \ge-\overline I_\mu(a,\eta).
\end{equation}
Consequently,
\begin{equation}\label{eq:one-spike-lower}
 \liminf_{N\to\infty}\frac1N\log
 \E\Crt_N^J([a-M,a+\eta+M])
 \ge-\overline I_\mu(a,\eta).
\end{equation}
\end{theorem}

\begin{proof}
Let
\[
 b_N=\frac{J_{1\cdots1}}{\sqrt N},\qquad
 E_N=\{b_N\in[a,a+\eta]\}
 \cap\{\norm{R_N}_\infty\le M\}.
\]
Choose \(\delta\) so that \(a[1-(1-\delta)^p]>2M\).  On \(E_N\),
\[
 b_N[1-(1-\delta)^p]\ge a[1-(1-\delta)^p]>2M,
\]
so \cref{lem:cap} applies to \eqref{eq:spike-decomposition} with its actual coefficient \(b_N\).  It gives a critical value in \([b_N-M,b_N+M]\subset[a-M,a+\eta+M]\).  Thus
\begin{align*}
 &\E\Crt_{N,\max}^J([a-M,a+\eta+M])
 \ge\Pp(E_N)\\
 &\quad=\Pp\left(\frac\xi{\sqrt N}\in[a,a+\eta]\right)
 \Pp\left(
 \left.\norm{R_N}_\infty\le M\,\right|\,
 \frac{J_{1\cdots1}}{\sqrt N}\in[a,a+\eta]\right).
\end{align*}
Take logarithms, divide by \(N\), use
\eqref{eq:conditional-remainder-cost}, and pass to the liminf.  Since every
local maximum is a critical point, \eqref{eq:one-spike-lower} follows as
well.
\end{proof}

\begin{corollary}[General quadratic-tail obstruction]\label[corollary]{cor:general-tail-obstruction}
Assume that the distinguished coordinate has law \(\mu\).  Suppose that
for some \(M<\infty\), \(\eta>0\), \(\delta>0\), and a sequence
\(a_k\to\infty\), \eqref{eq:conditional-remainder-cost} holds with
\(a=a_k\) for every sufficiently large \(k\), and
\begin{equation}\label{eq:subquadratic-local-cost}
 \overline I_\mu(a_k,\eta)\le\left(\frac12-\delta\right)a_k^2
\end{equation}
for all sufficiently large \(k\), then, for every sufficiently large
\(k\), setting \(B_k=[a_k-M,a_k+\eta+M]\),
\begin{align}
 \liminf_{N\to\infty}\frac1N\log\E_\mu\Crt_{N,\max}^J(B_k)
 &>\limsup_{N\to\infty}\frac1N\log\E\Crt_{N,\max}^G(B_k),\label{eq:general-tail-max}\\
 \liminf_{N\to\infty}\frac1N\log\E_\mu\Crt_N^J(B_k)
 &>\lim_{N\to\infty}\frac1N\log\E\Crt_N^G(B_k).
 \label{eq:general-tail-total}
\end{align}
If the joint coefficient array is invariant in law under global sign
reversal \(J\mapsto-J\), the analogous strict inequality holds for local
minima in \(-B_k\).
\end{corollary}

\begin{proof}
By \cref{thm:one-spike}, the non-Gaussian lower rate is at least \(-\overline I_\mu(a_k,\eta)\).  On the other hand, \cref{thm:Gaussian-local-complexity,cor:theta-large} give
\[
 \lim_{N\to\infty}\frac1N\log\E\Crt_N^G(B_k)
 \le-\frac12(a_k-M)^2+O(\log a_k).
\]
The quadratic gap \(\delta a_k^2\) dominates the
\(O(a_k+\log a_k)\) error.  Since the Gaussian maximum count is bounded
by the total count, both strict inequalities follow.  Global sign
invariance under \(J\mapsto-J\) exchanges maxima in \(B_k\) with minima
in \(-B_k\).
\end{proof}

\begin{lemma}[Strict finite-support profile]\label[lemma]{lem:strict-profile}
Let \(K\) be a compact neighborhood in \(\mathbb S^{N-1}(1)\), let \(x_0\) lie in its relative interior, and suppose a continuous function \(P\) satisfies
\[
 P(x_0)=\max_KP,
 \qquad
 \Delta:=P(x_0)-\max_{\partial K}P>0.
\]
If \(\norm r_\infty\le M\) and \(b\Delta>2M\), then \(bP+r\) has a local maximum in the interior of \(K\), with value in
\[
 [bP(x_0)-M,bP(x_0)+M].
\]
If \(P,r\) are \(C^1\), this point is a spherical critical point.
\end{lemma}

\begin{proof}
At \(x_0\), the value is at least \(bP(x_0)-M\), while on \(\partial K\) it is at most \(b(P(x_0)-\Delta)+M<bP(x_0)-M\).  A maximizer on \(K\) therefore lies in its interior.  The value bounds and the critical-point assertion are immediate.
\end{proof}

\begin{corollary}[Diagonal-free high-energy one-spike lower bound]\label[corollary]{cor:diagonal-free-spike}
Consider the diagonal-free ordered model
\[
 h_{N,\mathrm{df}}^J(x)=N^{-1/2}
 \sum_{\substack{I\in[N]^p:\ i_1,\ldots,i_p\ \mathrm{distinct}}}J_Ix_I,
 \qquad x\in\mathbb S^{N-1}(1),
\]
with i.i.d. centered \(K\)-sub-Gaussian entries having common law
\(\mu\).  There is \(M=M(p,K)\) such that, for all sufficiently large
\(a\) and every \(\eta>0\),
\begin{align}\label{eq:diagonal-free-spike}
 &\liminf_{N\to\infty}\frac1N\log
 \E\Crt_{N,\max}^{J,\mathrm{df}}
 \left([ap^{-p/2}-M,(a+\eta)p^{-p/2}+M]\right)\notag\\
 &\hspace{45mm}\ge-\overline I_\mu(a,\eta).
\end{align}
Here \(\Crt_{N,\max}^{J,\mathrm{df}}\) denotes the local-maximum count for \(h_{N,\mathrm{df}}^J\).  Thus the localized one-spike lower-bound mechanism does not rely on repeated tensor indices; this corollary does not compare its exponent with Gaussian diagonal-free complexity.
\end{corollary}

\begin{proof}
Separate the coefficient \(J_{1,2,\ldots,p}\) and put \(P(x)=x_1\cdots x_p\).  Deleting terms from the canonical metric in the proof of \cref{lem:remainder-tight} can only improve its bound, so the remaining field has sup norm at most a fixed \(M\) with probability at least \(1/2\), independently of this coefficient.  By the arithmetic--geometric mean inequality,
\[
 \max_{\norm x_2=1,\ x_i\ge0}P(x)=p^{-p/2}
\]
at the strict maximizer \(x_\star=p^{-1/2}(e_1+\cdots+e_p)\).  Fix a sufficiently small \(\rho>0\) and take
\[
 K_{N,\rho}=\{x\in\mathbb S^{N-1}(1):\norm{x-x_\star}_2\le\rho\}.
\]
We take \(\rho<p^{-1/2}\), so the first \(p\) coordinates are positive
throughout this cap and the preceding positive-orthant maximization
applies.
The boundary gap
\[
 \Delta_{p,\rho}=p^{-p/2}-\sup_{x\in\partial K_{N,\rho}}x_1\cdots x_p
\]
is positive uniformly in \(N\).  Indeed, writing \(t=\norm{(x_{p+1},\ldots,x_N)}_2\) reduces the maximization to a compact subset of \(\R^{p+1}\) independent of \(N\), and equality in the arithmetic--geometric mean inequality occurs only at \((x_\star,t=0)\), which is not on the boundary.  On \(J_{1,2,\ldots,p}/\sqrt N\in[a,a+\eta]\), \cref{lem:strict-profile} applies once \(a\Delta_{p,\rho}>2M\), and gives the displayed energy interval.  Independence and the definition of \(\overline I_\mu\) finish the proof.
\end{proof}

\begin{remark}[Structural zero-energy critical set]\label[remark]{rem:df-zero-critical}
After summing the ordered coefficients over permutation orbits, every
diagonal-free degree-\(p\) Hamiltonian has the square-free form
\[
 h(x)=N^{-1/2}\sum_{\substack{S\subset[N]\\ |S|=p}}
 \widehat J_S\prod_{i\in S}x_i.
\]
If \(T(x)=\{i:x_i\ne0\}\) has \(\abs{T(x)}\le p-2\), then \(h(x)=0\) and
\[
 \partial_jh(x)=N^{-1/2}
 \sum_{\substack{S\ni j\\ |S|=p}}
 \widehat J_S\prod_{i\in S\setminus\{j\}}x_i=0,
 \qquad j\in[N],
\]
because every product in a first derivative requires \(p-1\) nonzero
coordinates.  Hence every such point on the sphere is a spherical critical
point at energy zero.  For \(p\ge4\), every coordinate
\((p-2)\)-plane contributes a sphere \(\mathbb S^{p-3}\) of such points;
hence the diagonal-free field is structurally non-Morse, and a full
critical-point count whose energy window contains
zero is infinite; in particular, \cref{lem:Morse} and the generic
algebraic ceiling for the full coefficient space do not apply to this
specialization.  For \(p=3\), the forced set consists only of the
coordinate poles.

This does not affect \cref{cor:diagonal-free-spike}: after increasing the
lower threshold on \(a\), its displayed energy interval is contained in
\((0,\infty)\), and the corollary asserts only a high-energy local-maximum
lower bound.  Any finite full-count formulation for the diagonal-free
model must either isolate nonzero-energy critical points and establish
generic finiteness, or add a generic perturbation.
\end{remark}

\begin{remark}\label[remark]{rem:distinct-index-spike}
In a symmetrized diagonal-free convention, the coefficient of \(x_1\cdots x_p\) includes the relevant orbit-multiplicity factor; only the deterministic energy scale in \eqref{eq:diagonal-free-spike} changes.
\end{remark}

\begin{remark}[Beyond sub-Gaussian tails]
If the spike coordinate is independent of \(R_N\), then
\eqref{eq:conditional-remainder-cost} follows from
\(\Pp(\norm{R_N}_\infty\le M)=e^{-o(N)}\).  In the standard
symmetric-tensor convention, Sawhney and Sellke prove the stronger
thermodynamic tightness needed here under the sharp finite-\(2p\)-moment
hypothesis \cite{SawhneySellke}.  In the ordered convention, one may first
group entries over permutation orbits and then use their uniformly
bounded-moment theorem.  Its hypotheses hold for the Weibull-tail
construction below.  We retain the elementary argument for sub-Gaussian
disorder because it is self-contained and is exactly the regime used in
\cref{thm:counterexample}.
\end{remark}

\subsection{Construction of a smooth moment-matching law}

\begin{proposition}[Moment correction around an arbitrary tail law]\label[proposition]{prop:general-moment-correction}
Fix \(q\ge1\).  Let \(Q\) be any symmetric probability law with finite
moments through order \(2q\).  For all sufficiently small
\(\varepsilon>0\), there exists a symmetric compactly supported
\(C^\infty\) density \(\nu_\varepsilon\) such that
\begin{equation}\label{eq:general-mixture}
 \mu_\varepsilon=\varepsilon Q+(1-\varepsilon)\nu_\varepsilon
\end{equation}
matches the standard Gaussian moments through order \(2q\).  If \(\nu_\varepsilon\) is supported in \([-R,R]\), then
\[
 \mu_\varepsilon(A)=\varepsilon Q(A)
 \quad\text{for every Borel }A\subset\R\setminus[-R,R].
\]
In particular, \(\mu_\varepsilon\) and \(Q\) have identical logarithmic tail asymptotics.
If \(Q\) has a strictly positive \(C^\infty\) density, then so does
\(\mu_\varepsilon\).
\end{proposition}

\begin{proof}
Let \(q_k=\int x^{2k}Q(\dd x)\) and \(m_k=(2k-1)!!\).  The compact component must have even moments
\[
 \widetilde m_k(\varepsilon)=\frac{m_k-\varepsilon q_k}{1-\varepsilon},
 \qquad k=0,\ldots,q.
\]
This vector converges to \((m_0,\ldots,m_q)\) as
\(\varepsilon\downarrow0\).  The interior moment representation of
\cref{lem:moment-interior} therefore supplies a symmetric compactly
supported smooth density with these moments for all sufficiently small
\(\varepsilon\).  Odd moments vanish by symmetry.  The displayed
identity outside the support of \(\nu_\varepsilon\) is immediate from
\eqref{eq:general-mixture}.  If \(Q\) has a strictly positive smooth
density, the same is true of the mixture.
\end{proof}

\begin{proposition}[Bounded moment matching with a remote bump]
\label[proposition]{prop:bounded-moment-law}
Fix an integer \(q\ge1\).  For every sufficiently large \(b\), there is
a symmetric, compactly supported probability law \(\mu_b\) with a
\(C^\infty\) density such that
\begin{equation}\label{eq:bounded-moment-match}
 \int x^j\mu_b(\dd x)=\E G^j,
 \qquad 0\le j\le2q,
\end{equation}
and
\begin{equation}\label{eq:remote-bump-mass}
 \mu_b([b,b+1])=\frac12b^{-(2q+2)}.
\end{equation}
\end{proposition}

\begin{proof}
Fix a nonnegative \(\rho\in C_c^\infty((0,1))\) with
\(\int\rho=1\), and let \(Q_b\) have the symmetric density
\[
 q_b(x)=\frac12\rho(x-b)+\frac12\rho(-x-b).
\]
Thus \(Q_b\) assigns mass \(1/2\) to each of \((b,b+1)\) and
\((-b-1,-b)\).  Put
\[
 \varepsilon_b=b^{-(2q+2)}
\]
and denote the even moments of \(Q_b\) by \(s_{b,k}\).  A compact
component \(\nu_b\) in
\[
 \mu_b=\varepsilon_bQ_b+(1-\varepsilon_b)\nu_b
\]
must have even moments
\begin{equation}\label{eq:bounded-target-moments}
 \widetilde m_{b,k}
 =\frac{m_k-\varepsilon_bs_{b,k}}{1-\varepsilon_b},
 \qquad m_k=(2k-1)!!,\quad 0\le k\le q.
\end{equation}
Since \(s_{b,k}\le(b+1)^{2k}\),
\[
 \max_{0\le k\le q}
 \abs{\widetilde m_{b,k}-m_k}\longrightarrow0.
\]
The final assertion of \cref{lem:moment-interior} therefore supplies,
for all sufficiently large \(b\), a symmetric \(C^\infty\) density
\(\nu_b\), supported in one fixed compact set, with the moments
\eqref{eq:bounded-target-moments}.  Symmetry gives all the odd moments
in \eqref{eq:bounded-moment-match}.

Increase \(b\), if necessary, so that the support of \(\nu_b\) is
disjoint from \([b,b+1]\).  Then the latter interval receives only the
positive bump of \(Q_b\), and its \(\mu_b\)-mass is
\(\varepsilon_b/2\).  The mixture has a compactly supported
\(C^\infty\) density and satisfies all the claimed properties.
\end{proof}

\begin{proposition}[Moment matching with a wider Gaussian tail]\label[proposition]{prop:moment-law}
Fix an integer \(q\ge1\) and \(\sigma>1\).  There exists \(\varepsilon_0>0\) such that, for every \(0<\varepsilon<\varepsilon_0\), one can find a symmetric compactly supported \(C^\infty\) density \(\nu\) for which
\begin{equation}\label{eq:mu-mixture}
 \mu=\varepsilon\cN(0,\sigma^2)+(1-\varepsilon)\nu
\end{equation}
has the same moments as \(\cN(0,1)\) through order \(2q\).  The density of \(\mu\) is strictly positive and \(C^\infty\), and \(\mu\) is sub-Gaussian.
\end{proposition}

\begin{proof}
The construction is proved in full in \cref{app:moment-construction}.
The idea is to represent the moments of \(G^2\) through order \(q\) by a
positive quadrature rule on \((0,\infty)\), replace its atoms by narrow
symmetric smooth bumps, and solve the resulting invertible moment system
for nearby positive weights.  For small \(\varepsilon\), the moments that
the compact component must supply are a small perturbation of the
Gaussian moment vector, so positivity is preserved.  The Gaussian
mixture component gives a strictly positive smooth density and a
sub-Gaussian tail.
\end{proof}

\begin{lemma}[Tail cost of the constructed law]\label[lemma]{lem:mixture-tail}
For the law in \cref{prop:moment-law}, every \(a>0\) and \(\eta>0\) satisfy
\begin{equation}\label{eq:mixture-tail-cost}
 \lim_{N\to\infty}-\frac1N\log
 \Pp_\mu\left(\frac\xi{\sqrt N}\in[a,a+\eta]\right)
 =\frac{a^2}{2\sigma^2}.
\end{equation}
\end{lemma}

\begin{proof}
For all sufficiently large \(N\), the interval \([a\sqrt N,(a+\eta)\sqrt N]\) is disjoint from the support of the compact component.  Its \(\mu\)-probability is therefore exactly \(\varepsilon\) times its \(\cN(0,\sigma^2)\)-probability.  The Gaussian local large-deviation estimate gives \eqref{eq:mixture-tail-cost}.
\end{proof}

\subsection{Proof of the counterexample}

\begin{proof}[Proof of \cref{thm:counterexample}]
Apply \cref{prop:moment-law} with \(q=p\), obtaining a smooth sub-Gaussian law matching moments through order \(2p\).  Let \(M\) be given by \cref{lem:remainder-tight}.  Since \(R_N\) omits \(J_{1\cdots1}\), independence makes the same remainder bound valid conditional on the spike event and verifies \eqref{eq:conditional-remainder-cost}.  Fix a small \(\eta>0\), and for \(a>2M\) set
\[
 B_a=[a-M,a+\eta+M].
\]
By \cref{thm:one-spike,lem:mixture-tail},
\begin{equation}\label{eq:counter-lower}
 \liminf_{N\to\infty}\frac1N\log\E_\mu\Crt_N^J(B_a)
 \ge-\frac{a^2}{2\sigma^2}.
\end{equation}

For the Gaussian spherical \(p\)-spin model, \cref{thm:Gaussian-local-complexity} gives
\begin{equation}\label{eq:Gaussian-interval-complexity}
 \lim_{N\to\infty}\frac1N\log\E\Crt_N^G(B_a)
 =\sup_{u\in B_a}\theta_p(u).
\end{equation}
By \cref{cor:theta-large},
\begin{equation}\label{eq:Gaussian-high-bound}
 \sup_{u\in B_a}\theta_p(u)
 \le-\frac12(a-M)^2+C\log a
\end{equation}
for all large \(a\).  Since \(\sigma>1\),
\[
 \frac12-\frac1{2\sigma^2}>0.
\]
Therefore, for sufficiently large \(a\),
\[
 -\frac12(a-M)^2+C\log a
 <-\frac{a^2}{2\sigma^2}.
\]
Combining this strict inequality with \eqref{eq:counter-lower}--\eqref{eq:Gaussian-high-bound} proves \eqref{eq:counter-strict} with \(B=B_a\).  The stronger lower bound \eqref{eq:one-spike-max-lower}, together with \(\Crt_{N,\max}^G(B_a)\le\Crt_N^G(B_a)\), proves \eqref{eq:counter-max}.  Since both \(\mu\) and the Gaussian law are symmetric, replacing \(J\) by \(-J\) exchanges maxima in \(B_a\) with minima in \(-B_a\), proving \eqref{eq:counter-min}.

Finally, because the coefficient law has a smooth density and the non-Morse tensors form a proper algebraic discriminant set of Lebesgue measure zero, the restricted polynomial is almost surely Morse.  Hence the critical-point counts used above are finite almost surely.
\end{proof}

\begin{proof}[Proof of \cref{thm:bounded-block-counterexample}]
Let \(q=\lceil r/2\rceil\) and take the laws \(\mu_b\) from
\cref{prop:bounded-moment-law}.  Write
\[
 \pi_b=\mu_b([b,b+1])=\frac12b^{-(2q+2)},\qquad
 \lambda_b=\frac{-\log\pi_b}{b^2}
 =\frac{(2q+2)\log b+\log2}{b^2}.
\]
Choose \(b\) sufficiently large that \(\mu_b\) exists and
\begin{equation}\label{eq:bounded-block-cost-gap}
 \lambda_b<\frac12\left(1-\frac1b\right)^2,
 \qquad
 1-2^{-p}>\frac2b.
\end{equation}
Fix this \(b\) and this law.  It is centered, has variance one, is
sub-Gaussian, and matches all Gaussian moments through order
\(2q\ge r\).

Let \(M=M(p,K)\) be the constant in
\cref{thm:mesoscopic-block-lower} for a sub-Gaussian parameter \(K\)
of this fixed law.  For \(c>0\), set
\[
 A=bc^{p/2},\qquad D=c^{p/2}=\frac Ab,\qquad
 W=M+D+1.
\]
The second inequality in \eqref{eq:bounded-block-cost-gap} allows us
to choose \(c\) so large that
\[
 A(1-2^{-p})>2(M+D).
\]
It also gives \(A-W>0\) for all sufficiently large \(c\).
For the fixed interval
\[
 B=[A-W,A+W],
\]
\cref{thm:mesoscopic-block-lower}, with \(w=1\) and \(t=1/2\),
therefore yields
\begin{equation}\label{eq:bounded-proof-lower}
 \liminf_{N\to\infty}\frac1N\log
 \E_{\mu_b}\Crt_{N,\max}^J(B)
 \ge c^p\log\pi_b=-\lambda_bA^2.
\end{equation}

We next increase the fixed constant \(c\), if necessary.  Since
\(A-W=A(1-b^{-1})-M-1\), \cref{cor:theta-large} implies
\begin{equation}\label{eq:bounded-proof-Gaussian}
 \sup_{u\in B}\theta_p(u)
 \le-\frac12\{A(1-b^{-1})-M-1\}^2
 +C_p\log(A+W)
\end{equation}
once \(c\) is large.  The strict first inequality in
\eqref{eq:bounded-block-cost-gap} shows that the right-hand side of
\eqref{eq:bounded-proof-Gaussian} is strictly smaller than
\(-\lambda_bA^2\) for all sufficiently large \(c\).  The exact
Gaussian interval formula \eqref{eq:Gaussian-local-complexity},
together with
\(\Crt_{N,\max}^G(B)\le\Crt_N^G(B)\), now proves
\eqref{eq:bounded-block-max} and \eqref{eq:bounded-block-total}.

The law is symmetric, so \(J\mapsto-J\) gives the reflected
local-minimum inequality.  Its joint coefficient law is absolutely
continuous; hence \cref{lem:Morse} and the algebraic ceiling in
\cref{rem:tensor-convention} ensure that all the counts used here are
finite almost surely.  Finally, if
\(\supp\mu_b\subset[-R_b,R_b]\), then
\(R_b<N^{1/2-\varepsilon}\) for every sufficiently large \(N\), which
proves \eqref{eq:bounded-cutoff-automatic}.
\end{proof}

\begin{proof}[Proof of \cref{cor:symmetric-bounded-block}]
Let \(q=\lceil r/2\rceil\).  Choose a large number \(L\) and apply
\cref{prop:bounded-moment-law} with remote positive bump
\([L,L+1]\).  Thus
\[
 \pi_L:=\mu_L([L,L+1])=\frac12L^{-(2q+2)}.
\]
Put
\[
 b=\frac{L}{\sqrt{p!}},\qquad
 \gamma_p=\frac2{\sqrt{p!}},\qquad
 d_p=1-2^{-p},\qquad
 \ell_L=-\log\pi_L.
\]
Since \(\ell_L=O_q(\log L)\), we may fix \(L\) so large that
\begin{equation}\label{eq:symmetric-block-cost-gap}
 bd_p>4\gamma_p,\qquad
 \frac{\ell_L}{p!}<\frac12(b-\gamma_p)^2.
\end{equation}
Fix the resulting law \(\mu=\mu_L\), a sub-Gaussian parameter \(K\),
and a constant \(c>0\) to be chosen below.  Let
\[
 k_N=\lfloor cN^{1/p}\rfloor,\qquad
 S_N=\{1,\ldots,k_N\},\qquad v_N=k_N^{-1/2}
 \sum_{i\in S_N}e_i,
\]
and let \(P_N=P_{k_N,S_N}\) be the profile in
\cref{lem:squarefree-rank-one}.  Expose only the square-free orbit
coordinates in \(S_N\):
\[
 \mathcal E_N^{\mathrm{sym}}
 =\bigcap_{\substack{A\subset S_N\\ |A|=p}}
 \{\xi_A\in[L,L+1]\}.
\]
These coordinates are independent, so
\begin{equation}\label{eq:symmetric-block-event}
 \Pp(\mathcal E_N^{\mathrm{sym}})
 =\pi_L^{\binom{k_N}{p}},\qquad
 \frac1N\binom{k_N}{p}\longrightarrow\frac{c^p}{p!}.
\end{equation}

For a square-free orbit \(A\), \(d_A=p!\).  Hence on
\(\mathcal E_N^{\mathrm{sym}}\), the symmetric-model energy density
decomposes as
\begin{equation}\label{eq:symmetric-block-decomposition}
 h_{N,\mathrm{sym}}^\xi(x)
 =A_NP_N(x)+U_N(x)+R_N(x),\qquad
 A_N=\frac{bk_N^{p/2}}{\sqrt N},
\end{equation}
where \(R_N\) deletes the exposed square-free block and is independent
of \(\mathcal E_N^{\mathrm{sym}}\).  The within-block error satisfies
\begin{align}
 \norm{U_N}_\infty
 &\le\frac{\sqrt{p!}}{\sqrt N}
 \sup_{\norm x_2=1}
 \sum_{\substack{A\subset S_N\\ |A|=p}}\prod_{i\in A}|x_i|\notag\\
 &\le\frac1{\sqrt{p!N}}
 \sup_{\norm x_2=1}\left(\sum_{i\in S_N}|x_i|\right)^p
 \le\frac{k_N^{p/2}}{\sqrt{p!N}}
 \le\gamma_pc^{p/2}\label{eq:symmetric-block-error}
\end{align}
for all large \(N\).

Deleting coordinates only decreases the canonical increment metric,
because
\[
 \sum_{\mathfrak a}d_{\mathfrak a}
 \bigl(x^{\mathfrak a}-y^{\mathfrak a}\bigr)^2
 =\norm{x^{\tensor p}-y^{\tensor p}}_2^2.
\]
Moreover, \(\E|R_N(v_N)|\le CK/\sqrt N\), since its pointwise
variance is at most \(CK^2/N\).  Apply the Dudley argument in
\cref{lem:remainder-tight} to the centered increment process
\(R_N(x)-R_N(v_N)\), and then add \(\E|R_N(v_N)|\).  This gives an
\(M=M(p,K)\) such that
\begin{equation}\label{eq:symmetric-remainder-event}
 \Pp(\norm{R_N}_\infty\le M)\ge\frac12
\end{equation}
for every sufficiently large \(N\).

Consider the cap
\[
 \mathcal K_N=\{x\in\mathbb S^{N-1}(1):\ip{x}{v_N}\ge1/2\}.
\]
By \cref{lem:squarefree-rank-one},
\[
 P_N(v_N)-\sup_{\partial\mathcal K_N}P_N
 \ge d_p-\frac{2\binom p2}{k_N}\ge\frac{d_p}{2}
\]
for all large \(N\).  The first inequality in
\eqref{eq:symmetric-block-cost-gap} allows us to fix \(c\) so large
that, for all large \(N\),
\[
 A_N\frac{d_p}{2}>
 2\bigl(M+\gamma_pc^{p/2}\bigr).
\]
Thus, on the intersection of the events in
\eqref{eq:symmetric-block-event} and
\eqref{eq:symmetric-remainder-event}, the
value of \eqref{eq:symmetric-block-decomposition} at \(v_N\) is
strictly larger than every boundary value of \(\mathcal K_N\).  A
maximizer over this cap is therefore an interior local maximum.
Using \eqref{eq:squarefree-profile-error} and
\eqref{eq:symmetric-block-error},
as well as
\[
 P_N(v_N)=1+O_p(k_N^{-1}),\qquad
 \sup_{\mathcal K_N}P_N\le1+O_p(k_N^{-1}),
\]
its value belongs, for all large \(N\), to the fixed interval
\begin{equation}\label{eq:symmetric-block-window}
 B_c=[(b-\gamma_p)c^{p/2}-M-2,\,
       (b+\gamma_p)c^{p/2}+M+2].
\end{equation}
Increase \(c\), if necessary, so that \(B_c\subset(0,\infty)\).
Equations
\eqref{eq:symmetric-block-event} and
\eqref{eq:symmetric-remainder-event} imply
\begin{equation}\label{eq:symmetric-block-lower}
 \liminf_{N\to\infty}\frac1N\log
 \E_\mu\Crt_{N,\max}^{\mathrm{sym}}(B_c)
 \ge-\frac{c^p}{p!}\ell_L.
\end{equation}

By \cref{cor:theta-large}, the Gaussian total-count exponent on \(B_c\)
is at most
\[
 -\frac12\{(b-\gamma_p)c^{p/2}-M-2\}^2+O_p(\log c).
\]
The second strict inequality in
\eqref{eq:symmetric-block-cost-gap} makes this smaller than the
right-hand side of \eqref{eq:symmetric-block-lower} when \(c\) is
sufficiently large.  Since the Gaussian symmetric-coordinate field
has the same law as the ordered Gaussian field, this proves the
strict total and local-maximum comparisons.  Symmetry gives the
reflected minimum comparison, absolute continuity gives the Morse
property, and compact support gives the asserted entrywise cutoff.
\end{proof}

\begin{proof}[Proof of \cref{thm:subexp-counterexample}]
Let \(Q_\alpha\) have density
\begin{equation}\label{eq:qalpha-density}
 q_\alpha(x)=Z_\alpha^{-1}
 \exp\left\{-(1+x^2)^{\alpha/2}\right\}.
\end{equation}
This density is symmetric, strictly positive, smooth, and has moments of every order; it is subexponential for \(\alpha\ge1\), and it satisfies \cref{def:stretched-tail} with \(c_\xi=1\).  Apply \cref{prop:general-moment-correction} with \(q=p\) and \(Q=Q_\alpha\).  The resulting law \(\mu\) matches Gaussian moments through order \(2p\) and has the same Weibull-tail asymptotics as \(Q_\alpha\).

Group the ordered coefficients according to their index multisets.  If \(\mathfrak a\) is a multiset and \(c_{\mathfrak a}\le p!\) is the size of its permutation orbit, set
\[
 \widetilde J_{\mathfrak a}=\frac1{c_{\mathfrak a}}
 \sum_{I\in\operatorname{Orb}(\mathfrak a)}J_I.
\]
Then distinct \(\widetilde J_{\mathfrak a}\)'s are independent and
\(\operatorname{Var}(\widetilde J_{\mathfrak a})=c_{\mathfrak a}^{-1}\).  If the multiplicities in \(\mathfrak a\) are \(n_1,n_2,\ldots\), then
\[
 c_{\mathfrak a}^{-1}=\frac{\prod_j n_j!}{p!},
\]
which is exactly the variance normalization in the symmetric-tensor convention of Sawhney and Sellke.  Extending \(\widetilde J_{\mathfrak a}\) over its orbit recovers the ordered Hamiltonian.  Fix \(\varepsilon_0\in(0,1)\).  Since \(c_{\mathfrak a}\le p!\) and \(\mu\) has moments of every order,
\[
 \zeta_{\mathfrak a}:=\sqrt{c_{\mathfrak a}}\,
 \widetilde J_{\mathfrak a}
 =c_{\mathfrak a}^{-1/2}
 \sum_{I\in\operatorname{Orb}(\mathfrak a)}J_I
\]
are independent, centered, and have variance one, and
\[
 \sup_{N,\mathfrak a}\E
 \abs{\zeta_{\mathfrak a}}^{2p+\varepsilon_0}<\infty.
\]
Moreover, the grouped Hamiltonian is exactly
\[
 N^{-(p-1)/2}\sum_{\mathfrak a}
 \sqrt{c_{\mathfrak a}}\,\zeta_{\mathfrak a}\sigma^{\mathfrak a}.
\]
Thus the uniformly bounded-moment hypothesis of
\cite[Theorem~1.4 and Corollary~6.5]{SawhneySellke} holds.  Applying the
ground-state conclusion at \(\beta=\infty\) to the disorder and its
negative gives tightness of
\[
 \sup_{x\in\mathbb S^{N-1}(1)}\abs{h_N^J(x)}.
\]
Since
\[
 \norm{R_N}_\infty\le \sup_x\abs{h_N^J(x)}+\frac{\abs{J_{1\cdots1}}}{\sqrt N}
\]
and \(J_{1\cdots1}/\sqrt N\to0\) in probability, the sequence \(\norm{R_N}_\infty\) is tight.  Choose \(M\) so that \(\Pp(\norm{R_N}_\infty\le M)\ge1/2\) for all sufficiently large \(N\), and enlarge it to cover the finitely many remaining \(N\).  This proves \eqref{eq:remainder-prob}.  Since \(R_N\) omits \(J_{1\cdots1}\), independence also turns this bound into the conditional estimate \eqref{eq:conditional-remainder-cost}.

Let the compact correction in \cref{prop:general-moment-correction} be supported in \([-R,R]\).  For large \(N\), the interval \([a\sqrt N,(a+\eta)\sqrt N]\) lies outside this support, and a unit subinterval gives
\[
 \Pp_\mu\left(\frac\xi{\sqrt N}\in[a,a+\eta]\right)
 \ge c\exp\left\{-\bigl(1+(a\sqrt N+1)^2\bigr)^{\alpha/2}\right\}.
\]
Therefore, for every fixed \(a>0\) and \(\eta>0\),
\begin{equation}\label{eq:stretched-cost-zero}
 0\le-\frac1N\log\Pp_\mu\left(\frac\xi{\sqrt N}\in[a,a+\eta]\right)
 \le C N^{\alpha/2-1}+O(N^{-1})\longrightarrow0.
\end{equation}
Choose \(a>2M\) so large that \(\sup_{u\in[a-M,a+\eta+M]}\theta_p(u)<0\), which is possible by \cref{cor:theta-large}.  The one-spike theorem and \eqref{eq:stretched-cost-zero} give a non-Gaussian local-maximum exponent at least zero, whereas the Gaussian total interval exponent is the displayed negative supremum by \cref{thm:Gaussian-local-complexity}.  This proves \eqref{eq:subexp-counter-strict} and the stated maximum inequality.  Symmetry again gives the minimum statement on the reflected interval.
\end{proof}

\begin{definition}[Sharp sub-Gaussianity]\label[definition]{def:sharp-subg}
A centered, variance-one law is sharp sub-Gaussian if
\begin{equation}\label{eq:sharp-subg}
 \E e^{t\xi}\le e^{t^2/2},
 \qquad t\in\R.
\end{equation}
\end{definition}

\begin{definition}[Profile-sharp sub-Gaussianity]
\label[definition]{def:profile-sharp-subg}
We say that the random field satisfies \(({\rm PSG})_p\) if, for every
\(N\), \(x\in\mathbb S^{N-1}(1)\), and \(s\in\R\),
\begin{equation}\label{eq:profile-sharp-subg}
 \E e^{s h_N^J(x)}\le e^{s^2/(2N)}.
\end{equation}
This rank-one/profile-direction bound is a convenient finite-\(N\)
sufficient condition for the weaker tail-rate input below.
\end{definition}

\begin{definition}[Uniform Gaussian-rate profile tails]
\label[definition]{def:uniform-profile-tail}
We say that the random field satisfies \(({\rm PT})_p\) if, for every
\(a>0\),
\begin{equation}\label{eq:uniform-profile-tail}
 \limsup_{N\to\infty}\frac1N\log
 \sup_{x\in\mathbb S^{N-1}(1)}
 \Pp\{\abs{h_N^J(x)}\ge a\}
 \le-\frac{a^2}{2}.
\end{equation}
This asymptotic one-profile tail rate, rather than a moment-generating
function bound, is the exact probabilistic input used by the
sparse-profile upper bounds.
\end{definition}

A stronger sufficient condition in the ordered coefficient space is the
joint linear-form bound
\begin{equation}\label{eq:joint-sharp-subg}
 \E\exp\left\{\sum_{I\in[N]^p}t_IJ_I\right\}
 \le
 \exp\left\{\frac12\sum_{I\in[N]^p}t_I^2\right\},
 \qquad t\in\R^{N^p}.
\end{equation}
Independent sharp sub-Gaussian coordinates imply
\eqref{eq:joint-sharp-subg}, which in turn implies
\eqref{eq:profile-sharp-subg} by taking
\(t_I=sx_I/\sqrt N\).  Neither independence nor the full joint bound is
required.  Chernoff's inequality shows that \(({\rm PSG})_p\) implies
\(({\rm PT})_p\).  The analogous orbit-coordinate joint bound is
likewise sufficient for both conditions in the standard
symmetric-coordinate model.

\begin{proposition}[Sharp sub-Gaussianity removes the one-entry advantage]\label[proposition]{prop:sharp-one-spike}
If \(\xi\) is sharp sub-Gaussian, then for every \(a>0\) and \(\eta>0\),
\begin{equation}\label{eq:sharp-tail-cost}
 \liminf_{N\to\infty}-\frac1N\log
 \Pp\left(\frac\xi{\sqrt N}\in[a,a+\eta]\right)
 \ge\frac{a^2}{2}.
\end{equation}
Hence a single localized tensor entry cannot occur at a smaller exponential cost than in the standard Gaussian model.
\end{proposition}

\begin{proof}
Chernoff's inequality and \eqref{eq:sharp-subg} give
\[
 \Pp(\xi\ge a\sqrt N)
 \le\inf_{t>0}\exp\{-ta\sqrt N+t^2/2\}
 =e^{-a^2N/2}.
\]
The local interval probability is bounded by this upper tail.
\end{proof}

\begin{proposition}[Gaussian-rate energy tails on sparse profiles]
\label[proposition]{prop:sharp-profile-energy}
Suppose that the random field satisfies \(({\rm PT})_p\).
Put
\[
 \pi_N(a):=
 \sup_{x\in\mathbb S^{N-1}(1)}
 \Pp\{\abs{h_N^J(x)}\ge a\}.
\]
Put
\[
 \mathcal S_{N,k}
 =\{x\in\mathbb S^{N-1}(1):
       \abs{\operatorname{supp}x}\le k\},
 \qquad
 \Lambda_p=\frac{p^{p+1}}{p!}.
\]
For \(1\le k\le N\), \(a>0\), and
\(0<\delta<\Lambda_p^{-1}\),
\begin{align}
 &\Pp\left\{
 \sup_{x\in\mathcal S_{N,k}}\abs{h_N^J(x)}\ge a
 \right\}\notag\\
 &\qquad\le
 \binom Nk\left(1+\frac2\delta\right)^k
 \pi_N\bigl((1-\Lambda_p\delta)a\bigr).
 \label{eq:sparse-profile-net-bound}
\end{align}
Consequently, if \(1\le k_N\le N\), \(k_N=o(N)\), and
\(a_N\to a>0\), then
\begin{equation}\label{eq:sparse-profile-rate}
 \limsup_{N\to\infty}\frac1N\log
 \Pp\left\{
 \sup_{x\in\mathcal S_{N,k_N}}\abs{h_N^J(x)}
 \ge a_N\right\}
 \le-\frac{a^2}{2}.
\end{equation}
The same conclusions hold for the standard symmetric-coordinate model
whenever its field satisfies \(({\rm PT})_p\); the profile MGF and
orbit-coordinate joint bounds described above are sufficient.
\end{proposition}

\begin{proof}
We first record the deterministic net estimate.  If \(P\) is a real
\(p\)-homogeneous polynomial on a Euclidean space and \(B\) is its
symmetric \(p\)-linear polarization, the real polarization identity
gives
\[
 \norm B_{\mathrm{op}}\le\frac{p^p}{p!}\norm P_\infty,
 \qquad
 \norm P_\infty=\sup_{\norm{x}_2=1}\abs{P(x)}.
\]
Indeed,
\[
 B(x_1,\ldots,x_p)
 =\frac1{2^pp!}
 \sum_{\varepsilon\in\{-1,1\}^p}
 \left(\prod_{j=1}^p\varepsilon_j\right)
 P\left(\sum_{j=1}^p\varepsilon_jx_j\right),
\]
and homogeneity gives the displayed norm bound.  Telescoping the
\(p\) arguments of \(B\) therefore yields, for unit \(x,y\),
\begin{equation}\label{eq:homogeneous-Lipschitz}
 \abs{P(x)-P(y)}
 \le\Lambda_p\norm P_\infty\norm{x-y}_2.
\end{equation}
Thus, if \(\mathcal N\) is a Euclidean \(\delta\)-net of the unit
sphere and \(\delta<\Lambda_p^{-1}\),
\begin{equation}\label{eq:polynomial-net}
 \norm P_\infty
 \le(1-\Lambda_p\delta)^{-1}
 \max_{y\in\mathcal N}\abs{P(y)}.
\end{equation}
This uses the symmetric polarization of the polynomial
\(x\mapsto h_N^J(x)\); the original ordered coefficient tensor need
not be symmetric.

For every \(S\subset[N]\) with \(\abs S=k\), choose a
\(\delta\)-net \(\mathcal N_S\) of the unit sphere in \(\R^S\) with
\(\abs{\mathcal N_S}\le(1+2/\delta)^k\).  Every vector supported on at
most \(k\) coordinates lies in at least one such \(k\)-dimensional
coordinate subspace.  Apply \eqref{eq:polynomial-net} on each
subspace, followed by the definition of \(\pi_N\) and a union bound, to
prove \eqref{eq:sparse-profile-net-bound}.

For \eqref{eq:sparse-profile-rate}, take
\(\delta_N=(k_N/N)^{1/2}\) for all sufficiently large \(N\).  Then
\(\delta_N\to0\), and
\[
 \frac1N\log\binom N{k_N}
 \le\frac{k_N}{N}\log\frac{eN}{k_N}\longrightarrow0,
 \qquad
 \frac{k_N}{N}\log\left(1+\frac2{\delta_N}\right)
 \longrightarrow0.
\]
If \(a_N\to a>0\), then for every \(b<a\),
\((1-\Lambda_p\delta_N)a_N\ge b\) for all sufficiently large \(N\).
Apply \eqref{eq:uniform-profile-tail} at \(b\), take logarithms in
\eqref{eq:sparse-profile-net-bound}, divide by \(N\), and then let
\(b\uparrow a\).  The deterministic polynomial-net argument is
unchanged in the symmetric-coordinate model.
\end{proof}

\begin{corollary}[Coherent block averages in the ordered model]
\label[corollary]{cor:sharp-block-average}
In the ordered model, assume the stronger finite-\(N\) profile MGF
condition \(({\rm PSG})_p\).  Let
\(S\subset[N]\), \(\abs S=k\), and set
\[
 \overline J_S=\frac1{k^p}\sum_{I\in S^p}J_I,
 \qquad
 u_S=k^{-1/2}\sum_{i\in S}e_i.
\]
For every \(b,a\ge0\),
\begin{align}
 \Pp\{\overline J_S\ge b\}
 &\le\exp\{-k^pb^2/2\},\label{eq:block-average-Chernoff}\\
 \Pp\{h_N^J(u_S)\ge a\}
 &\le e^{-Na^2/2}.\label{eq:block-energy-Chernoff}
\end{align}
Moreover,
\begin{equation}\label{eq:any-block-Chernoff}
 \Pp\left\{\max_{\abs S=k}h_N^J(u_S)\ge a\right\}
 \le\binom Nk e^{-Na^2/2}.
\end{equation}
Hence the right-hand side has exponential rate at most
\(-a^2/2\) whenever \(k=o(N)\).  If
\(k_N=N^{\gamma+o(1)}\), the coefficient level
\[
 b_{N,k_N}(a)=\frac{a\sqrt N}{k_N^{p/2}}
 =aN^{(1-p\gamma)/2+o(1)}
\]
interpolates, for \(0\le\gamma\le1/p\), between a quadratic-scale
spike (with endpoint \(k=1\)) and a fixed-amplitude mesoscopic coherent
block.
\end{corollary}

\begin{proof}
Since
\[
 \sum_{I\in S^p}J_I=\sqrt N\,k^{p/2}h_N^J(u_S),
\]
\eqref{eq:profile-sharp-subg} gives, for \(s>0\),
\[
 \E\exp\left\{s\sum_{I\in S^p}J_I\right\}
 \le\exp\{k^ps^2/2\}.
\]
Chernoff's inequality optimized at \(s=b\) proves
\eqref{eq:block-average-Chernoff}.  Since
\[
 h_N^J(u_S)
 =\frac{k^{p/2}}{\sqrt N}\overline J_S,
\]
apply \eqref{eq:block-average-Chernoff} with
\(b=a\sqrt N/k^{p/2}\) to obtain
\eqref{eq:block-energy-Chernoff}; a union bound over supports proves
\eqref{eq:any-block-Chernoff}.
\end{proof}

\begin{corollary}[Profiles asymptotically supported on \(o(N)\)
coordinates]\label[corollary]{cor:sharp-approximately-sparse}
Under the hypotheses of \cref{prop:sharp-profile-energy}, let
\(1\le k_N\le N\), \(k_N=o(N)\), and for \(\rho>0\) put
\[
 \mathcal A_N(k_N,\rho)
 =\left\{x\in\mathbb S^{N-1}(1):
 \text{some \(S\subset[N]\), \(\abs S\le k_N\), satisfies }
 \norm{x_{S^c}}_2\le\rho\right\}.
\]
Then, for every \(a>0\),
\begin{equation}\label{eq:approximately-sparse-rate}
 \lim_{\rho\downarrow0}\limsup_{N\to\infty}\frac1N\log
 \Pp\left\{
 \sup_{x\in\mathcal A_N(k_N,\rho)}
 \abs{h_N^J(x)}\ge a\right\}
 \le-\frac{a^2}{2}.
\end{equation}
The same conclusion holds in the standard symmetric-coordinate model
whenever that field satisfies \(({\rm PT})_p\); the orbit-coordinate
joint MGF bound is sufficient.
\end{corollary}

\begin{proof}
Write
\(\norm{h_N^J}_\infty
=\sup_{\norm x_2=1}\abs{h_N^J(x)}\).
Taking \(\delta_0=(2\Lambda_p)^{-1}\) in
\eqref{eq:polynomial-net} on the full sphere gives, for every \(M>0\),
\begin{equation}\label{eq:global-polynomial-tail}
 \Pp\{\norm{h_N^J}_\infty>M\}
 \le(1+4\Lambda_p)^N\pi_N(M/2),
\end{equation}
and hence its upper exponential rate is at most
\(\log(1+4\Lambda_p)-M^2/8\).
Fix \(x\in\mathcal A_N(k_N,\rho)\) and a witnessing set \(S\).  Put
\(r=\norm{x_S}_2\) and \(y=x_S/r\).  For \(\rho<1\),
\[
 \norm{x-y}_2
 \le\norm{x_{S^c}}_2+(1-r)
 \le\rho+\rho^2\le2\rho.
\]
On \(\{\norm{h_N^J}_\infty\le M\}\),
\eqref{eq:homogeneous-Lipschitz} therefore gives
\[
 \abs{h_N^J(x)-h_N^J(y)}
 \le2\Lambda_pM\rho.
\]
Consequently,
\begin{align*}
 &\Pp\left\{
 \sup_{x\in\mathcal A_N(k_N,\rho)}
 \abs{h_N^J(x)}\ge a\right\}\\
 &\quad\le
 \Pp\{\norm{h_N^J}_\infty>M\}
 +\Pp\left\{
 \sup_{y\in\mathcal S_{N,k_N}}\abs{h_N^J(y)}
 \ge a-2\Lambda_pM\rho\right\}.
\end{align*}
Choose \(M\) so large that the exponential rate in
\eqref{eq:global-polynomial-tail} is strictly below \(-a^2/2\),
keep \(M\) fixed, apply \eqref{eq:sparse-profile-rate}, and then let
\(\rho\downarrow0\).
\end{proof}

\begin{corollary}[Ordered-model high-energy critical points near
sublinear supports]
\label[corollary]{cor:sharp-sparse-critical-upper}
In the ordered model, assume the hypotheses of
\cref{prop:sharp-profile-energy} and the critical-count ceiling
\eqref{eq:critical-count-ceiling}.  Let
\[
 \Crt_{N,k_N,\rho}^J([a,\infty))
 =\#\left\{x:\nabla_SH_N^J(\sqrt N x)=0,\ 
 h_N^J(x)\ge a,\ 
 x\in\mathcal A_N(k_N,\rho)\right\}.
\]
For every integer sequence \(1\le k_N\le N\), \(k_N=o(N)\), and every
\(a>0\),
\begin{equation}\label{eq:sharp-sparse-critical-rate}
 \lim_{\rho\downarrow0}\limsup_{N\to\infty}\frac1N\log
 \E\Crt_{N,k_N,\rho}^J([a,\infty))
 \le\log(p-1)-\frac{a^2}{2}.
\end{equation}
Consequently, for every sufficiently large fixed \(a\) and every such
sequence \(k_N\), there is
\(\rho_0=\rho_0(p,a)>0\) such that, for every fixed
\(0<\rho<\rho_0\),
\begin{equation}\label{eq:sharp-sparse-below-Gaussian}
 \limsup_{N\to\infty}\frac1N\log
 \E\Crt_{N,k_N,\rho}^J([a,a+1])
 <
 \lim_{N\to\infty}\frac1N\log
 \E\Crt_N^G([a,a+1]).
\end{equation}
\end{corollary}

\begin{proof}
Let \(D_{N,p}\) be the ceiling in
\eqref{eq:critical-count-ceiling}.  The restricted count vanishes unless
\[
 \sup_{\mathcal A_N(k_N,\rho)}h_N^J\ge a,
\]
and is at most \(D_{N,p}\) otherwise.  Take expectations and apply the
profile estimate in \cref{cor:sharp-approximately-sparse}.
More explicitly, the proof of that corollary gives, for every fixed
\(M>0\) and \(\rho>0\),
\begin{align}
 &\limsup_{N\to\infty}\frac1N\log
 \E\Crt_{N,k_N,\rho}^J([a,\infty))\notag\\
 &\quad\le\log(p-1)+
 \max\left\{
 \log(1+4\Lambda_p)-\frac{M^2}{8},\
 -\frac12\bigl(a-2\Lambda_pM\rho\bigr)_+^2
 \right\}.
 \label{eq:sharp-sparse-explicit-critical}
\end{align}
This estimate is uniform over the particular sequence \(k_N=o(N)\).

By \cref{cor:theta-large},
\[
 \theta_p(a)
 -\left\{\log(p-1)-\frac{a^2}{2}\right\}
 \longrightarrow+\infty.
\]
Moreover, \(\theta_p\) is strictly decreasing on
\([E_\infty,\infty)\) by \eqref{eq:theta-explicit}.  Hence, for all
sufficiently large \(a\),
\[
 \log(p-1)-\frac{a^2}{2}<\theta_p(a)
 =\sup_{u\in[a,a+1]}\theta_p(u).
\]
Choose \(M\) so that the first term in the maximum in
\eqref{eq:sharp-sparse-explicit-critical}, after adding
\(\log(p-1)\), is below \(\theta_p(a)\).  Then choose
\(\rho_0=\rho_0(p,a)>0\) so that the same is true of the second term
whenever \(0<\rho<\rho_0\).  Finally,
\cref{thm:Gaussian-local-complexity} proves
\eqref{eq:sharp-sparse-below-Gaussian}.
\end{proof}

For a deterministic Borel set
\(E_N\subset\mathbb S^{N-1}(1)\), write
\[
 \mathscr C_N^J(V;E_N)
 :=
 \E\sum_{\substack{x:\nabla_SH_N^J(\sqrt N x)=0\\x\in E_N}}
 e^{-NV(h_N^J(x))}.
\]

\begin{proposition}[Weighted sparse-profile bound]
\label[proposition]{prop:weighted-sharp-sparse}
Assume \(({\rm PT})_p\) and the critical-count ceiling
\eqref{eq:critical-count-ceiling}.
If \(V:\R\to\R\) is Borel measurable and bounded below and \(k_N=o(N)\),
then
\begin{equation}\label{eq:weighted-sharp-sparse}
 \lim_{\rho\downarrow0}\limsup_{N\to\infty}\frac1N
 \log\mathscr C_N^J\bigl(V;\mathcal A_N(k_N,\rho)\bigr)
 \le\Xi_p^{\rm sp}(V),
\end{equation}
where
\begin{equation}\label{eq:Xi-sparse}
 \Xi_p^{\rm sp}(V)
 :=
 \log(p-1)+
 \sup_{u\in\R}\left\{-V(u)-\frac{u^2}{2}\right\}.
\end{equation}
\end{proposition}

\begin{proof}
The critical-count ceiling gives deterministic \(D_{N,p}\) with
\[
 \Crt_N^J(\R)\le D_{N,p},
 \qquad
 \limsup_N\frac1N\log D_{N,p}\le\log(p-1).
\]
Fix \(M,\delta>0\), partition \([-M,M]\) into finitely many intervals
\(I_j\) of length at most \(\delta\), and put
\(d_j=\inf_{u\in I_j}\abs u\).  Then
\begin{align*}
 &\E\sum_{\substack{x\text{ critical},\
 x\in\mathcal A_N(k_N,\rho)\\ h_N^J(x)\in I_j}}
 e^{-NV(h_N^J(x))}\\
 &\qquad\le
 D_{N,p}e^{-N\inf_{I_j}V}
 \Pp\left\{
 \sup_{x\in\mathcal A_N(k_N,\rho)}
 \abs{h_N^J(x)}\ge d_j\right\}.
\end{align*}
For \(d_j>0\), \cref{cor:sharp-approximately-sparse} bounds the last
probability at rate at most \(-d_j^2/2\); for \(d_j=0\), the trivial
bound one gives the same statement.  The contribution from
\(\abs{h_N^J}>M\) is at most
\[
 D_{N,p}e^{-N\inf_\R V}
 \Pp\left\{
 \sup_{\mathcal A_N(k_N,\rho)}\abs{h_N^J}>M
 \right\},
\]
whose rate, after \(\rho\downarrow0\), is at most
\(\log(p-1)-\inf_\R V-M^2/2\).
The number of intervals does not depend on \(N\), so addition takes the
maximum of their exponential rates.  Thus the preceding finite
decomposition gives
\begin{align*}
 &\lim_{\rho\downarrow0}\limsup_{N\to\infty}\frac1N
 \log\mathscr C_N^J\bigl(V;\mathcal A_N(k_N,\rho)\bigr)\\
 &\quad\le
 \log(p-1)+
 \max\left\{
 \max_j\left(-\inf_{I_j}V-\frac{d_j^2}{2}\right),
 -\inf_\R V-\frac{M^2}{2}
 \right\}.
\end{align*}
For every \(I_j\subset[-M,M]\) of length at most \(\delta\),
\[
 -\inf_{I_j}V-\frac{d_j^2}{2}
 \le
 \sup_{u\in I_j}\left\{-V(u)-\frac{u^2}{2}\right\}
+M\delta+\frac{\delta^2}{2}.
\]
Let first \(\delta\downarrow0\), and then \(M\to\infty\).  The global
lower bound on \(V\) sends the tail term to \(-\infty\), and the
remaining terms give \eqref{eq:weighted-sharp-sparse}.
\end{proof}

\begin{remark}[Scope of the sharp localized upper bound]
\Cref{cor:sharp-sparse-critical-upper} controls critical points whose
\(\ell^2\)-mass outside \(o(N)\) coordinates vanishes.  It includes
one-coordinate, finite-support, and genuinely mesoscopic supports.
It does not control mixed profiles carrying a nonvanishing amount of
mass in both a localized component and a delocalized background, nor
does it compare the determinant under conditioning on such a profile.
Those regimes require the multiscale deformed Kac--Rice upper bound
described in \cref{sec:conjecture}.
The Gaussian quantity in
\eqref{eq:sharp-sparse-below-Gaussian} is the unrestricted total-count
exponent; no same-profile universality or determinant comparison is
asserted.
\end{remark}

\subsection{Entry cutoffs and normalized sample log-counts}

For a Borel set \(B\), define the normalized sample observable
\begin{equation}\label{eq:sample-log-count}
 Q_N^J(B)=\frac1N\log\bigl(1+\Crt_N^J(B)\bigr).
\end{equation}
Here and below, we use the convention \(\log(1+\infty)=\infty\).  The \(1+\Crt\)
convention keeps the observable finite when the window is empty; it is
not intended to identify the logarithm of a positive typical count in
every energy regime.  Algebraic genericity makes it uniformly bounded,
but is not needed for the distributional cutoff comparison below.

\begin{proposition}[Cutoff equivalence for sample log-counts]
\label[proposition]{prop:quenched-cutoff}
Fix \(0\le\varepsilon<1/2\) and a Borel set \(B\subset\R\), with no
independence, identical-distribution, density, or genericity condition.
Set
\begin{equation}\label{eq:cutoff-event}
 \mathcal T_{N,\varepsilon}
 =\left\{\max_{I\in[N]^p}|J_I|
 \le N^{1/2-\varepsilon}\right\},\qquad
 \delta_{N,\varepsilon}
 =\Pp(\mathcal T_{N,\varepsilon}^c),
\end{equation}
and, whenever the conditioning event has positive probability, write
\(\Pp_{N,\varepsilon}^{\mathrm{cut}}
:=\Pp(\,\cdot\mid\mathcal T_{N,\varepsilon})\).  If merely
\(\delta_{N,\varepsilon}=o(1)\), then the event has positive probability
for all sufficiently large \(N\), and
\[
 \norm{\nu-\nu'}_{\mathrm{TV}}
 :=\sup_A\abs{\nu(A)-\nu'(A)}
\]
satisfies
\begin{equation}\label{eq:cutoff-TV}
 \left\|
 \mathcal L_{\Pp}\bigl(Q_N^J(B)\bigr)
 -\mathcal L_{\Pp_{N,\varepsilon}^{\mathrm{cut}}}
 \bigl(Q_N^J(B)\bigr)
 \right\|_{\mathrm{TV}}
 \le\delta_{N,\varepsilon}.
\end{equation}
Consequently, for every deterministic sequence \(a_N\), every
\(\eta>0\), and every bounded measurable
\(\varphi:[0,\infty]\to\R\),
\begin{align}
 &\left|
 \Pp\bigl(\abs{Q_N^J(B)-a_N}>\eta\bigr)
 -\Pp_{N,\varepsilon}^{\mathrm{cut}}
  \bigl(\abs{Q_N^J(B)-a_N}>\eta\bigr)
 \right|
 \le\delta_{N,\varepsilon},\label{eq:cutoff-prob-convergence}\\
 &\left|\E\varphi(Q_N^J(B))
 -\E_{N,\varepsilon}^{\mathrm{cut}}\varphi(Q_N^J(B))\right|
 \le2\norm{\varphi}_\infty\delta_{N,\varepsilon}=o(1).
 \label{eq:cutoff-test}
\end{align}
In particular, convergence of \(Q_N^J(B)\) in probability to a
deterministic limit is equivalent under the original and conditioned
laws.

The hypothesis \(\delta_{N,\varepsilon}=o(1)\) follows, for example,
from the marginal bound
\[
 \sup_{I\in[N]^p}\E\exp(J_I^2/K^2)\le2,
\]
which requires no independence and gives
\begin{equation}\label{eq:cutoff-probability}
 \delta_{N,\varepsilon}
 \le2N^p\exp\{-K^{-2}N^{1-2\varepsilon}\}=o(1).
\end{equation}
Mean comparison needs only a separate integrability input.  In
particular, if \(0\le Q_N^J(B)\le M_N\) almost surely and
\(M_N\delta_{N,\varepsilon}\to0\), then
\begin{equation}\label{eq:cutoff-mean-general}
 \left|\E Q_N^J(B)
 -\E_{N,\varepsilon}^{\mathrm{cut}}Q_N^J(B)\right|
 \le M_N\delta_{N,\varepsilon}=o(1).
\end{equation}
The critical-count ceiling \({\rm(CC)}_p\) is sufficient, with
\[
 M_N=\frac1N\log(1+D_{N,p})=O(1).
\]
Algebraic genericity gives the explicit \(N\)-independent choice
\begin{equation}\label{eq:cutoff-Mp}
 M_N=M_p:=
 \log(p-1)+\log\left(1+\frac2{p-2}\right).
\end{equation}
\end{proposition}

\begin{proof}
If a probability measure is conditioned on an event of probability
\(1-\delta\), its total-variation distance from the original measure
is \(\delta\).  Total variation contracts under measurable maps, which
proves \eqref{eq:cutoff-TV} and
\eqref{eq:cutoff-prob-convergence}.  The bounded-test estimate
\eqref{eq:cutoff-test} is the dual total-variation bound.

Under the displayed marginal exponential-square bound, Markov's
inequality gives
\[
 \Pp(|J_I|>t)
 \le e^{-t^2/K^2}\E e^{J_I^2/K^2}
 \le2e^{-t^2/K^2}.
\]
A union bound over \(N^p\) coordinates proves
\eqref{eq:cutoff-probability}.  If \(Q_N^J(B)\in[0,M_N]\), its
oscillation gives \eqref{eq:cutoff-mean-general}.  Condition
\({\rm(CC)}_p\) gives
\(Q_N^J(B)\le N^{-1}\log(1+D_{N,p})=O(1)\).
Finally, algebraic genericity persists under conditioning and gives the
explicit ceiling
\[
 \Crt_N^J(B)\le2\frac{(p-1)^N-1}{p-2}.
\]
Substitution into \eqref{eq:sample-log-count} proves
\eqref{eq:cutoff-Mp}.
\end{proof}

\begin{remark}[What the cutoff does and does not prove]
\label[remark]{rem:cutoff-quenched-scope}
\Cref{prop:quenched-cutoff} shows that whenever the entry-cutoff event
has probability \(1-o(1)\), removing its complement is invisible to
the normalized sample log-count in total variation.  Uniformly
sub-Gaussian marginals are merely one sufficient route to that
high-probability statement.  Thus the one-entry
annealed obstruction is not, by itself, a quenched obstruction.
This is only a reduction to an \(N\)-dependent bounded triangular
array: it neither identifies a quenched limit nor compares the
non-Gaussian model with the Gaussian one.

Nor does the proposition apply to
\(N^{-1}\log\E\Crt_N^J(B)\).  Rare events can change that annealed
quantity at order one.  Indeed,
\cref{thm:bounded-block-counterexample} shows that even inside every
such cutoff, a coherent event involving order \(N\) fixed-size
coefficients can retain a strictly non-Gaussian annealed lower rate.
\end{remark}

\begin{proof}[Proof of \cref{cor:no-finite-moment}]
Choose \(q\ge\lceil r/2\rceil\), apply \cref{prop:moment-law}, and let
\(\mu_1\) be the resulting wider-tail law and \(\mu_2=\cN(0,1)\).
They are symmetric and centered, have unit variance and strictly positive
smooth densities, and agree through order \(2q\), hence through order
\(r\).  The strict comparison on a sufficiently high compact interval
follows from the argument in the proof of \cref{thm:counterexample}.
\end{proof}

\section{From regularized bulk universality to exact complexity}\label{sec:transfer}

\subsection{One-sided diagonal transfer}

At fixed \(N\), under its stated hard-slice hypotheses,
\cref{prop:finite-N-dereg} approximates an already valid Kac--Rice slice
by the regularized functionals.  It does not establish the Kac--Rice
identity.  Finite-dimensional continuity also does not, by itself,
justify interchanging that limit with \(N\to\infty\).  The defects in
\cref{def:dereg-defects} record exactly the two one-sided exponential
discrepancies, without requiring an absolute or uniform iterated limit.

\begin{proof}[Proof of \cref{thm:conditional-transfer}]
Under \eqref{eq:abstract-pressure-bridge}, for each fixed
\((\eps,\eta)\) one has
\(\widehat P_{N,L}(\eps,\eta)\to F(\eps,\eta)\), and
\(F(\eps,\eta)\to\Phi\) jointly as
\((\eps,\eta)\to(0,0)\).  We first use only these two facts.

If \(\delta_{L,\mathrm{dr}}^{J,\mathrm{up}}(V)=+\infty\), the upper
bound is immediate.  Otherwise fix
\(\gamma>\delta_{L,\mathrm{dr}}^{J,\mathrm{up}}(V)\).
The definition of the joint liminf supplies a sequence
\((\eps_j,\eta_j)\to(0,0)\) such that
\[
 \limsup_N[\widehat Q_{N,L}
 -\widehat P_{N,L}(\eps_j,\eta_j)]_+\le\gamma
\]
for all sufficiently large \(j\).  At fixed \(j\),
\[
 \limsup_N\widehat Q_{N,L}
 \le F(\eps_j,\eta_j)+\gamma.
\]
Let \(j\to\infty\), then
\(\gamma\downarrow\delta_{L,\mathrm{dr}}^{J,\mathrm{up}}(V)\), to
obtain \eqref{eq:weighted-upper-defect}.  If the lower defect is
infinite, its bound is immediate; otherwise the same argument applied
to \([\widehat P_{N,L}-\widehat Q_{N,L}]_+\) proves
\eqref{eq:weighted-lower-defect}.

The exact decomposition
\begin{equation}\label{eq:point-exact-decomposition}
 \mathscr C_N^J(V)
 =\widehat{\mathscr C}_{N,L}^J(V)
 +\mathscr L_{N,L}^J(V)
\end{equation}
and
\[
 \max\left\{\frac1N\log a,\frac1N\log b\right\}
 \le\frac1N\log(a+b)
 \le
 \max\left\{\frac1N\log a,\frac1N\log b\right\}
 +\frac{\log2}{N}
\]
prove \eqref{eq:full-upper-defect} and
\eqref{eq:full-lower-defect}.  If both defects vanish, the weighted
pressure converges to \(\Phi\).  A convergent first branch in
\eqref{eq:point-exact-decomposition} gives the two exact max identities.
Here we use the elementary identities
\[
 \limsup_N\max\{a_N,b_N\}
 =\max\{\limsup_Na_N,\limsup_Nb_N\},
\]
and, whenever \(a_N\to a\),
\[
 \liminf_N\max\{a_N,b_N\}
 =\max\{a,\liminf_Nb_N\}.
\]
For the second identity, the lower bound is immediate; for the reverse
bound, choose a subsequence along which \(b_N\) converges to its
liminf and use \(a_N\to a\) along that subsequence.
Finally, the upper identity proves sufficiency in
\eqref{eq:exact-universality-iff-localized-defect}, while
\(\mathscr L_{N,L}^J(V)\le\mathscr C_N^J(V)\) proves necessity.

It remains to verify the advertised concrete sufficient condition.
Under the disorder-array and \(V\) hypotheses of
\cref{thm:bulk-universality} with \(m=2\),
\cref{cor:point-incoherent,thm:gaussian-variational} and
\eqref{eq:point-Haar-zero-cost} give
\begin{equation}\label{eq:uncut-fixed-pressure-limit}
 \widehat P_{N,L}(\eps,\eta)
 \longrightarrow
 F(\eps,\eta)
 :=
 \frac12\log(2\pi e)+\log a_{\eps,p}
 +\sup_u\left\{-\frac{u^2}{2}-V(u)
 +\ell_{p,\ell_\eta}(u)\right\}.
\end{equation}
Equation \eqref{eq:uncut-dereg-limit} gives
\(F(\eps,\eta)\to\Phi_p^{\rm bulk}(V)\) jointly, proving
\eqref{eq:abstract-pressure-bridge}.
\end{proof}

\begin{proposition}[Exact bulk--sparse--mixed decomposition]
\label[proposition]{prop:bulk-sparse-reduction}
Let \(V:\R\to\R\) be Borel measurable and bounded below, and let
\(0<r_N\le1\) and \(1\le k_N\le N\) be arbitrary deterministic
sequences.  Write
\[
 D_{N,r_N}
 =\{x\in\mathbb S^{N-1}(1):\norm{x}_\infty\le r_N\}
\]
and
\[
 \mathcal A_N(k_N,\rho)
 =\left\{x\in\mathbb S^{N-1}(1):
 \text{some \(S\subset[N]\), \(\abs S\le k_N\), satisfies }
 \norm{x_{S^c}}_2\le\rho\right\}.
\]
For \(\rho>0\), define the residual mixed-profile set
\begin{equation}\label{eq:residual-mixed-set}
 \mathcal R_N(r_N,k_N,\rho)
 :=
 \left\{x\in\mathbb S^{N-1}(1):
 \norm{x}_\infty>r_N,\ 
 \inf_{\substack{S\subset[N]\\|S|\le k_N}}
 \norm{x_{S^c}}_2>\rho\right\}.
\end{equation}
Then the disjoint decomposition
\begin{equation}\label{eq:three-region-decomposition}
 \mathbb S^{N-1}(1)
 =D_{N,r_N}\,\dot\cup\,
  \bigl(\mathcal A_N(k_N,\rho)\setminus D_{N,r_N}\bigr)
  \,\dot\cup\,\mathcal R_N(r_N,k_N,\rho)
\end{equation}
holds.  Suppose only that the exact bulk slice satisfies
\begin{equation}\label{eq:exact-growing-bulk-slice}
 \lim_{N\to\infty}\frac1N
 \log\mathscr C_N^J(V;D_{N,r_N})
 =:\Phi\in\R.
\end{equation}
No independence, moment matching, tail, density, or nondegeneracy
condition is needed for the decomposition itself.
For \(\rho>0\), set
\begin{align*}
 \mathscr S_N(\rho)
 &:=
 \mathscr C_N^J\bigl(
 V;\mathcal A_N(k_N,\rho)\setminus D_{N,r_N}\bigr),\\
 \mathscr R_N(\rho)
 &:=
 \mathscr C_N^J\bigl(
 V;\mathcal R_N(r_N,k_N,\rho)\bigr),
\end{align*}
and define
\begin{align}
 \overline\Phi_{\rm sp}^J(V;r_\bullet,k_\bullet)
 &:=
 \lim_{\rho\downarrow0}\limsup_{N\to\infty}
 \frac1N\log\mathscr S_N(\rho),
 \label{eq:exact-sparse-branch}\\
 \overline\Phi_{\rm mix}^J(V;r_\bullet,k_\bullet)
 &:=
 \lim_{\rho\downarrow0}\limsup_{N\to\infty}
 \frac1N\log\mathscr R_N(\rho).
 \label{eq:exact-mixed-branch}
\end{align}
Then
\begin{equation}\label{eq:exact-three-branch-formula}
 \limsup_{N\to\infty}\frac1N\log\mathscr C_N^J(V)
 =
 \max\left\{
 \Phi,\,
 \overline\Phi_{\rm sp}^J(V;r_\bullet,k_\bullet),\,
 \overline\Phi_{\rm mix}^J(V;r_\bullet,k_\bullet)
 \right\}.
\end{equation}
Consequently, exact universality holds if and only if both nonbulk
upper exponents are at most \(\Phi\).

If \(k_N=o(N)\), \(({\rm PT})_p\), and
\eqref{eq:critical-count-ceiling} also hold, then
\begin{equation}\label{eq:sparse-branch-Xi-bound}
 \overline\Phi_{\rm sp}^J(V;r_\bullet,k_\bullet)
 \le\Xi_p^{\rm sp}(V).
\end{equation}
In particular, whenever the explicit sparse-rate compatibility condition
\begin{equation}\label{eq:sparse-rate-compatible}
 \Xi_p^{\rm sp}(V)\le\Phi
\end{equation}
holds, exact universality is equivalent to the single residual condition
\begin{equation}\label{eq:weakest-mixed-condition}
 \overline\Phi_{\rm mix}^J(V;r_\bullet,k_\bullet)
 \le\Phi.
\end{equation}
When the disorder and \(V\) satisfy the hypotheses of
\cref{thm:growing-point-cutoff}, and
\eqref{eq:growing-Haar-zero-cost-condition} and
\(q_N^{\rm gr}\to0\) hold, that theorem supplies the corresponding
universal regularized bulk pressure.  Controlling its one-sided
de-regularization defects is one sufficient route to
\eqref{eq:exact-growing-bulk-slice}, with
\(\Phi=\Phi_p^{\rm bulk}(V)\).
\end{proposition}

\begin{proof}
The definition of \(\mathcal R_N\) gives
\eqref{eq:three-region-decomposition}.  Under
\eqref{eq:exact-growing-bulk-slice}, the three nonnegative exact counts in
\eqref{eq:three-region-decomposition} sum to
\(\mathscr C_N^J(V)\).  If
\[
 B_N=\mathscr C_N^J(V;D_{N,r_N}),\qquad
 M_N(\rho)=\max\{B_N,\mathscr S_N(\rho),\mathscr R_N(\rho)\},
\]
then, in the extended nonnegative reals,
\[
 M_N(\rho)\le\mathscr C_N^J(V)\le3M_N(\rho).
\]
The bulk-slice hypothesis makes \(B_N\) finite and positive for all
sufficiently large \(N\).  Taking logarithms, dividing by \(N\), and
then taking \(\limsup_N\) gives, for fixed \(\rho\),
\[
 \limsup_N\frac1N\log\mathscr C_N^J(V)
 =
 \max\left\{
 \Phi,\,
 \limsup_N\frac1N\log\mathscr S_N(\rho),\,
 \limsup_N\frac1N\log\mathscr R_N(\rho)
 \right\}.
\]
For \(0<\rho_1<\rho_2\),
\[
 \mathscr S_N(\rho_1)\le\mathscr S_N(\rho_2),
 \qquad
 \mathscr R_N(\rho_1)\ge\mathscr R_N(\rho_2).
\]
Their upper exponential rates inherit these monotonicities.  Hence both
extended-real limits as \(\rho\downarrow0\) exist, and letting
\(\rho\downarrow0\) in the fixed-\(\rho\) identity proves
\eqref{eq:exact-three-branch-formula}.  Necessity follows because each
nonbulk branch is dominated by the full count.  For sufficiency, also
\[
 \mathscr C_N^J(V)\ge\mathscr C_N^J(V;D_{N,r_N}),
\]
so \eqref{eq:exact-growing-bulk-slice} gives
\[
 \liminf_N\frac1N\log\mathscr C_N^J(V)\ge\Phi.
\]
Thus upper bounds by \(\Phi\) for both nonbulk branches force the full
limit to be \(\Phi\).  Finally,
\[
 \mathscr S_N(\rho)
 \le
 \mathscr C_N^J(V;\mathcal A_N(k_N,\rho)),
\]
so \cref{prop:weighted-sharp-sparse} proves
\eqref{eq:sparse-branch-Xi-bound}; the final assertion follows.
\end{proof}

\ifptrfextendedoutlook
\else
\section{Outlook on the remaining localized branch}
\label{sec:conjecture}

The unconditional results above isolate three mechanisms: a delocalized
regularized bulk, finite-support profiles paid for by quadratic-scale
tails, and mesoscopic coherent blocks paid for by speed-\(N\) collective
deviations.  The bounded-disorder theorem shows that the third mechanism
cannot be reduced to a finite list of moments or to a one-entry cutoff.
Meanwhile, \cref{prop:bulk-sparse-reduction} separates the unrestricted
count exactly into bulk, sparse, and residual mixed-profile branches.

A full non-Gaussian variational formula would therefore require two
inputs not proved here: an exponential comparison for the relevant hard
gradient slices and a noncircular description of the localized
complement.  Finite-support deformation formulas alone cannot supply the
second input, because the mesoscopic block already lies outside every
fixed-support ansatz.

\begin{openproblem}[Exact and quenched complexity]
Find verifiable distributional hypotheses under which the two
de-regularization defects vanish and the sparse and mixed-profile
branches are at most the bulk value.  For quenched complexity, a
corresponding theorem would additionally require comparison of a
two-point Kac--Rice functional; the cutoff equivalence in
\cref{prop:quenched-cutoff} does not provide that comparison.
\end{openproblem}
\fi

\section{Discussion}

The results clarify what finite moment matching does and does not
control.

\begin{enumerate}[label=\textbf{(\roman*)},leftmargin=2.2em]
\item Under uniformly subexponential tails and Gaussian moment matching
through \(m\), the delocalized regularized
energy--gradient--Hessian Kac--Rice functional satisfies a quantitative
\(CNq_N^{m-1}\) comparison.  The pressure is universal for every
\(p\ge3,m\ge2\), while this bound yields convergence of the expectation
ratio whenever \((m-1)(p-2)>2\).  In particular, mean and variance
suffice for pressure universality for every \(p\ge3\), and for ratio
universality when \(p\ge5\).  For bounded spectral weights, the
conclusion extends to Weibull upper tails with \(0<\alpha<1\) whenever
\(\alpha(p-2)>2\).
For point cutoffs the estimate is uniform along every sequence
\(r_N\) with \(\sqrt N\,r_N^{p-1}\) small, which permits Haar-full
cutoffs growing far beyond the original logarithmic constant.

\item The resulting universal pressure has an explicit one-dimensional
Gaussian variational formula.  The iterated determinant/delta
de-regularization of the limiting formula, with \(N\to\infty\) taken
first, yields the standard Gaussian pointwise complexity \(\theta_p\).
All bulk
conclusions remain valid when only the point \(x\), rather than its full
tangent-frame completion, is required to be incoherent, and for the
standard independent symmetric-coordinate normalization.
At finite \(N\), under an assumed valid one-point Kac--Rice slice,
\cref{prop:finite-N-dereg} approximates the hard gradient slice by the
regularized functionals.  At exponential scale, the uncut determinant comparison
separates two directional weighted-gradient errors from one near-zero
resolvent defect; no remote spectral-tail estimate remains.  All three
primitive defects vanish for Gaussian disorder, and fixed Gaussian
divisibility already removes the upper slice defect.

\item Unregularized high-energy annealed counts, without a spatial
localization restriction, are sensitive to more than quadratic-scale
tails.  A single \(a\sqrt N\) entry gives the one-spike branch.  More
strongly, a compactly supported smooth law can match any prescribed
finite set of Gaussian moments while a coherent block on
\(k_N\asymp N^{1/p}\) coordinates gives a strictly larger lower rate
than the Gaussian limit.  This bounded-disorder conclusion holds in
both the ordered and standard symmetric-coordinate models.  The strict
one-spike separation also occurs for every subquadratic Weibull index
\(\alpha\in(0,2)\).  In the diagonal-free ordered model we prove only
the corresponding high-energy local-maximum one-spike lower bound, not
a strict Gaussian comparison; for \(p\ge4\), that specialization also
has structural zero-energy critical manifolds.

For the same bounded \(2p\)-moment-matched law, the regularized
expectation ratio tends to one while the unregularized energy-window
count has a strictly larger lower rate than the Gaussian limit;
\cref{thm:bounded-coexistence} makes this separation explicit.
Both the one-entry and coherent-block events are exponentially rare and
therefore give no evidence, by themselves, of quenched
non-universality.

\item Entrywise truncation has opposite implications for annealed and
sample-log observables.  Even an \(N\)-independent bound does not restore
the annealed exponent, by the mesoscopic-block theorem.  On the other
hand, under no structural condition on the joint disorder, conditioning
on any event of probability \(1-o(1)\) changes the law of
\(Q_N^J(B)=N^{-1}\log(1+\Crt_N^J(B))\) by \(o(1)\) in total variation.
Uniformly sub-Gaussian marginals imply this for
\(\max_I|J_I|\le N^{1/2-\varepsilon}\), for every fixed
\(0\le\varepsilon<1/2\), without independence.
Algebraic genericity is one convenient sufficient route to a
deterministic ceiling that also transfers expectations.  This is not a
quenched-universality theorem.

\item Therefore any exact-universality theorem based only on finite
moment information must control the explicit nonbulk branches.  The
uniform profile-tail condition \(({\rm PT})_p\) removes the one-entry advantage
and, more generally,
gives the Gaussian quadratic excursion upper bound uniformly for profiles
asymptotically supported on \(o(N)\) coordinates.  Together with
\({\rm(CC)}_p\), this yields a strict high-energy upper bound for the
corresponding critical points, but does
not settle mixed profiles carrying nonvanishing localized and
delocalized masses; for a general potential, its weighted sparse bound
must also be compared with the bulk value.  The quantitative transfer
theorem identifies separate upper and lower de-regularization defects,
and, once both defects vanish, its exact max formula makes localization
control a necessary-and-sufficient branch criterion.  In the nonnegative-bulk
region, a zero-rate localized upper bound is sufficient once the two
de-regularization defects vanish.
\end{enumerate}

Thermodynamic universality and complexity universality are consequently distinct.  The finite \(2p\)-moment threshold is sharp for the free energy and ground state \cite{SawhneySellke}, but it does not control the annealed count of rare localized critical points.  The counterexample shows that this distinction is intrinsic.

\appendix

\section{Matrix trace derivatives}\label{app:matrix-derivatives}

We prove the estimates used in \cref{lem:spectral-influence}.  Let
\(A,C\in\operatorname{Sym}_n\), and let \(f=c+f_0\in\cF_m\).

\begin{lemma}[Derivatives of a spectral trace]\label[lemma]{lem:trace-derivatives}
For \(r=1\),
\begin{equation}\label{eq:trace-derivative-1}
 \abs{D\Tr f(A)[C]}
 \le C_f\sqrt n\norm C_{\HS}.
\end{equation}
For \(2\le r\le m+1\),
\begin{equation}\label{eq:trace-derivative-r}
 \abs{D^r\Tr f(A)[C,\ldots,C]}
 \le C_{f,r}\norm C_{\HS}^2\norm C_{\op}^{r-2}.
\end{equation}
The constants are uniform in \(n,A,C\).
\end{lemma}

\begin{proof}
The constant part of \(f\) disappears from every derivative.  By Fourier inversion,
\[
 f_0(A)=\int_\R \widehat f_0(t)e^{itA}\dd t.
\]
The Duhamel formula gives
\begin{align*}
 D^r e^{itA}[C,\ldots,C]
 =(it)^r\sum_{\pi}\int_{\Delta_r}
 e^{it s_0A}C e^{it s_1A}C\cdots C e^{it s_rA}\dd s,
\end{align*}
where \(\Delta_r=\{s_j\ge0:\sum_{j=0}^rs_j=1\}\), and the finite sum accounts for the symmetric multilinear ordering.  Since the exponentials are unitary, for \(r=1\),
\[
 \abs{\Tr(e^{its_0A}Ce^{its_1A})}
 \le\sqrt n\norm C_{\HS}.
\]
For \(r\ge2\), cyclicity of the trace and the Hilbert--Schmidt inequality give
\[
 \abs{\Tr(U_0CU_1C\cdots U_{r-1}CU_r)}
 \le\norm C_{\HS}^2\norm C_{\op}^{r-2}
\]
for arbitrary unitary \(U_j\).  Integrate against \(\abs t^r\abs{\widehat f_0(t)}\), which is finite by \eqref{eq:fourier-f}.
\end{proof}

\begin{lemma}[Derivatives of the uncut smoothed logarithm]
\label[lemma]{lem:uncut-logdet-derivatives}
For \(\eta>0\), \(A,C\in\operatorname{Sym}_n\), and \(r\ge1\),
\begin{equation}\label{eq:uncut-trace-derivative-formula}
 \frac{\dd^r}{\dd s^r}
 \Tr\ell_\eta(A+sC)\bigg|_{s=0}
 =
 (-1)^{r-1}(r-1)!\,
 \Re\Tr\!\left\{\bigl((A-i\eta I_n)^{-1}C\bigr)^r\right\}.
\end{equation}
Consequently,
\begin{align}
 \abs{D\Tr\ell_\eta(A)[C]}
 &\le\eta^{-1}\sqrt n\norm C_{\HS},
 \label{eq:uncut-trace-first}\\
 \abs{D^r\Tr\ell_\eta(A)[C,\ldots,C]}
 &\le(r-1)!\eta^{-r}
 \norm C_{\HS}^2\norm C_{\op}^{r-2},
 \qquad r\ge2.
 \label{eq:uncut-trace-higher}
\end{align}
The bounds are uniform in \(n,A,C\).
\end{lemma}

\begin{proof}
Since \(A+sC-i\eta I_n\) is invertible for real \(s\),
\[
 \Tr\ell_\eta(A+sC)
 =\Re\log\det(A+sC-i\eta I_n).
\]
Jacobi's formula and repeated differentiation of the resolvent give
\eqref{eq:uncut-trace-derivative-formula}.  Put
\(R=(A-i\eta I_n)^{-1}\).  Then
\(\norm R_{\op}\le\eta^{-1}\).  For \(r=1\),
\[
 \abs{\Tr(RC)}
 \le\norm R_{\HS}\norm C_{\HS}
 \le\eta^{-1}\sqrt n\norm C_{\HS}.
\]
For \(r\ge2\), the Schatten inequalities give
\[
 \abs{\Tr((RC)^r)}
 \le\norm{RC}_{\HS}^2\norm{RC}_{\op}^{r-2}
 \le\eta^{-r}\norm C_{\HS}^2\norm C_{\op}^{r-2}.
\]
Substitution proves the claim.
\end{proof}

\section{Smooth moment matching with a prescribed Gaussian tail}\label{app:moment-construction}

We prove \cref{prop:moment-law}.

\begin{lemma}[Interior moment representation by smooth bumps]\label[lemma]{lem:moment-interior}
Let
\[
 m_k=\E G^{2k}=(2k-1)!!,
 \qquad k=0,\ldots,q.
\]
There exist symmetric, compactly supported \(C^\infty\) probability densities
\(\phi_0,\ldots,\phi_q\) and positive weights \(w_0,\ldots,w_q\) such that the matrix
\begin{equation}\label{eq:moment-matrix}
 B_{kj}=\int x^{2k}\phi_j(x)\dd x,
 \qquad 0\le k,j\le q,
\end{equation}
is invertible and
\begin{equation}\label{eq:moment-weight-rep}
 m_k=\sum_{j=0}^q w_jB_{kj},
 \qquad k=0,\ldots,q.
\end{equation}
Moreover, all moment vectors sufficiently close to \((m_0,\ldots,m_q)\) in the affine hyperplane with zeroth moment equal to one admit such a representation with strictly positive weights summing to one.
\end{lemma}

\begin{proof}
Apply Gaussian quadrature to the law of \(G^2\) on \([0,\infty)\).  There exist distinct nodes \(0<y_0<\cdots<y_q\) and positive weights \(w_j^{(0)}\) such that
\begin{equation}\label{eq:quadrature}
 m_k=\sum_{j=0}^q w_j^{(0)}y_j^k,
 \qquad k=0,\ldots,q.
\end{equation}
For completeness, one may take the nodes to be the roots of the degree-\(q+1\) orthogonal polynomial for the \(\chi_1^2\) law; positivity of the quadrature weights is standard.

For small \(\delta>0\), choose a symmetric \(C_c^\infty\) probability density \(\phi_{j,\delta}\) supported in small neighborhoods of \(\pm\sqrt{y_j}\).  Its even moments converge to \(y_j^k\) as \(\delta\downarrow0\).  Hence the corresponding matrix \(B(\delta)\) converges to the Vandermonde matrix \((y_j^k)_{k,j=0}^q\), which is invertible.  For small \(\delta\), the solution
\[
 w(\delta)=B(\delta)^{-1}m
\]
is close to the positive vector \(w^{(0)}\), and is therefore positive.  This proves \eqref{eq:moment-weight-rep}.  Since the inverse matrix depends continuously on the target moment vector, all sufficiently close vectors also yield positive weights.
\end{proof}

\begin{proof}[Proof of \cref{prop:moment-law}]
Let \(m_k=(2k-1)!!\), \(0\le k\le q\).  If
\[
 \mu=\varepsilon\cN(0,\sigma^2)+(1-\varepsilon)\nu
\]
is to match the standard Gaussian even moments through order \(2q\), then \(\nu\) must have moments
\begin{equation}\label{eq:target-nu-moments}
 \widetilde m_k(\varepsilon)
 =m_k\frac{1-\varepsilon\sigma^{2k}}{1-\varepsilon},
 \qquad k=0,\ldots,q.
\end{equation}
For \(\varepsilon<\sigma^{-2q}\), these are positive, and as \(\varepsilon\downarrow0\), the vector \(\widetilde m(\varepsilon)\) converges to \(m\).  By \cref{lem:moment-interior}, for all sufficiently small \(\varepsilon\) there are positive weights \(\widetilde w_j\) such that
\[
 \nu(x)=\sum_{j=0}^q\widetilde w_j\phi_j(x)
\]
has the even moments \eqref{eq:target-nu-moments}.  Symmetry makes all odd moments vanish.  Hence \(\mu\) matches the standard Gaussian through order \(2q\).

The compact component has a smooth, compactly supported density, and the
Gaussian component has a strictly positive smooth density.  Their mixture
is therefore strictly positive and \(C^\infty\).  Finally, a mixture of a
compactly supported law and \(\cN(0,\sigma^2)\) is sub-Gaussian.
\end{proof}

\section{Finite-\texorpdfstring{$N$}{N} de-regularization and Kac--Rice}\label{app:finite-N-KR}

We record a conditional finite-dimensional approximation of an already
valid hard gradient slice.  The proposition assumes the one-point
Kac--Rice identity at that slice; it neither proves the identity nor
provides exponential uniformity in \(N\).

\begin{proposition}[Conditional finite-\(N\) approximation of a Kac--Rice slice]\label[proposition]{prop:finite-N-dereg}
Fix \(N\), let \(n=N-1\), and let \(V\) be Borel measurable and bounded
below.  For almost every \(O\in\cO_{N,L}\), suppose that \(g_J(O)\)
has a density \(p_O(g)\) on \(\R^n\) and that
\((h_J(O),A_J(O))\), conditionally on \(g_J(O)=g\), admits a jointly
measurable regular conditional kernel \(K_O(g,\dd h\,\dd A)\).  Assume
that versions at \(g=0\) have been fixed for which the hard-slice identity
\begin{equation}\label{eq:finite-N-hard-slice-assumption}
 \mathscr C_{N,L}^J(V)
 =\omega_N\int_{\cO_{N,L}}p_O(0)
   \int_{\R\times\operatorname{Sym}_n}
   e^{-NV(h)}\abs{\det A}\,
   K_O(0,\dd h\,\dd A)\nu_N(\dd O)
\end{equation}
holds.
For \(0<\eta\le1\) and \(R\ge1\), put
\[
 \Psi_{O,\eta,R}(g)
 =p_O(g)\int_{\R\times\operatorname{Sym}_n}
 e^{-NV(h)+\Tr f_{\eta,R}(A)}
 K_O(g,\dd h\,\dd A).
\]
Assume that \(g\mapsto\Psi_{O,\eta,R}(g)\) is bounded and continuous at zero for almost every \(O\), that
\[
 \int_{\cO_{N,L}}\norm{\Psi_{O,\eta,R}}_\infty\nu_N(\dd O)<\infty,
\]
and that
\begin{equation}\label{eq:finite-N-dominating-condition}
 \int_{\cO_{N,L}}\int_{\R\times\operatorname{Sym}_n}
 e^{-NV(h)}(1+\norm A_{\op})^n p_O(0)
 K_O(0,\dd h\,\dd A)\nu_N(\dd O)<\infty.
\end{equation}
Then
\begin{equation}\label{eq:finite-N-dereg}
 \lim_{R\to\infty}\lim_{\eta\downarrow0}\lim_{\eps\downarrow0}
 \E\cZ_{N,L}^J(V,\eps,f_{\eta,R})
 =\mathscr C_{N,L}^J(V).
\end{equation}
\end{proposition}

\begin{proof}
Disintegration with respect to the gradient gives
\[
 \E\cZ_{N,L}^J(V,\eps,f_{\eta,R})
 =\omega_N\int_{\cO_{N,L}}
 (\kappa_\eps^{\otimes n}*\Psi_{O,\eta,R})(0)\nu_N(\dd O).
\]
The boundedness, continuity, and integrability assumptions allow the product Cauchy approximate identity to pass to the limit \(\eps\downarrow0\).  Hence the last display converges to
\[
 \omega_N\int_{\cO_{N,L}}\int
 e^{-NV(h)+\Tr f_{\eta,R}(A)}p_O(0)
 K_O(0,\dd h\,\dd A)\nu_N(\dd O).
\]
For the fixed cutoff in \eqref{eq:f-eta-R},
\[
 e^{f_{\eta,R}(t)}\le C(1+\abs t),
 \qquad 0<\eta\le1,\quad R\ge1,
\]
with a universal \(C\).  Therefore
\[
 e^{\Tr f_{\eta,R}(A)}\le C_n(1+\norm A_{\op})^n.
\]
Condition \eqref{eq:finite-N-dominating-condition} and dominated convergence now justify first \(\eta\downarrow0\) and then \(R\to\infty\), giving
\[
 \omega_N\int_{\cO_{N,L}}\int
 e^{-NV(h)}\abs{\det A}\,p_O(0)
 K_O(0,\dd h\,\dd A)\nu_N(\dd O).
\]
\Cref{eq:finite-N-hard-slice-assumption} identifies this expression with
the frame-weighted count in \eqref{eq:exact-weighted}.  The gradient
density at zero times the
conditional determinant expectation is invariant under orthogonal
changes of tangent basis, and the fraction of admissible completions of
\(x\) is \(w_{N,L}(x)\).  If
\(\cO_{N,L}\) is replaced by \(\widehat\cO_{N,L}\), the same argument
gives the point-incoherent count in \eqref{eq:exact-point-count}; its
exact weight is \(\1_{\{x\in D_{N,L}\}}\).
\end{proof}

\subsection{Uncut one-sided frame-slice criteria}

Retain a hard gradient slice for every sufficiently large \(N\), but
do not assume the full finite-\(N\) approximation package above.
On \(\cO_{N,L}\times\R\times\operatorname{Sym}_n\), define
\begin{equation}\label{eq:frame-slice-measure}
 \Lambda_{N,L}(\dd O\,\dd h\,\dd A)
 :=
 \omega_Np_O(0)K_O(0,\dd h\,\dd A)\nu_N(\dd O).
\end{equation}
For \(0<\eta\le1\) and \(\eps>0\), put
\begin{align}
 \mathfrak K_{N,L}(\eps,\eta)
 &:=
 \E\cZ_{N,L}^J(V,\eps,\ell_\eta),
 \label{eq:frame-K}\\
 \mathfrak D_{N,L}(\eta)
 &:=
 \int e^{-NV(h)}
 \det(A^2+\eta^2I_n)^{1/2}
 \Lambda_{N,L}(\dd O\,\dd h\,\dd A),
 \label{eq:uncut-smoothed-determinant}\\
 \mathfrak D^0_{N,L}
 &:=
 \int e^{-NV(h)}\abs{\det A}\,
 \Lambda_{N,L}(\dd O\,\dd h\,\dd A).
 \label{eq:frame-D0}
\end{align}
Suppose only that these quantities are finite and positive and that the
hard-slice Kac--Rice identity
\begin{equation}\label{eq:hard-slice-is-count}
 \mathfrak D^0_{N,L}=\mathscr C_{N,L}^J(V)
\end{equation}
holds.  The domination in
\eqref{eq:finite-N-dominating-condition} is one sufficient condition
for finiteness of the two determinant slices.  The uncut comparison
proof gives finiteness of \(\mathfrak K\) under the uniform
subexponential hypothesis.

All definitions and results in this subsection remain valid after
replacing \(\cO_{N,L}\), \(\cZ_{N,L}\), and
\(\mathscr C_{N,L}\) by their point-incoherent hatted versions.  In that
version \(\mathfrak D^0_{N,L}=\widehat{\mathscr C}_{N,L}^J(V)\), and
the resulting direct defects are exactly those in
\cref{def:dereg-defects}.

Set
\[
 P_N(\eps,\eta)=\frac1N\log\mathfrak K_{N,L}(\eps,\eta),\quad
 S_N(\eta)=\frac1N\log\mathfrak D_{N,L}(\eta),\quad
 Q_N=\frac1N\log\mathfrak D^0_{N,L},
\]
and define the nonnegative primitive defects
\begin{align}
 \Gamma_{L,+}(\eps,\eta)
 &:=
 \limsup_{N\to\infty}[S_N(\eta)-P_N(\eps,\eta)]_+,
 \label{eq:Gamma-plus}\\
 \Gamma_{L,-}(\eps,\eta)
 &:=
 \limsup_{N\to\infty}[P_N(\eps,\eta)-S_N(\eta)]_+,
 \label{eq:Gamma-minus}\\
 \Sigma_{L,0}(\eta)
 &:=
 \limsup_{N\to\infty}\{S_N(\eta)-Q_N\}\ge0.
 \label{eq:Sigma-zero}
\end{align}

\begin{proposition}[Primitive one-sided de-regularization criterion]
\label[proposition]{prop:primitive-dereg}
For the point-incoherent version,
\begin{align}
 \delta_{L,\mathrm{dr}}^{J,\mathrm{up}}(V)
 &\le
 \liminf_{(\eps,\eta)\to(0,0)}
 \Gamma_{L,+}(\eps,\eta),\label{eq:primitive-up}\\
 \delta_{L,\mathrm{dr}}^{J,\mathrm{low}}(V)
 &\le
 \liminf_{(\eps,\eta)\to(0,0)}
 \{\Gamma_{L,-}(\eps,\eta)+\Sigma_{L,0}(\eta)\}.
 \label{eq:primitive-low}
\end{align}
The identical statement holds for the frame-weighted direct defects.
Thus the exact upper bound uses only one side of the gradient-slice
comparison.  The lower bound uses the reverse side and the near-zero
determinant defect; no remote spectral-tail condition is present.
\end{proposition}

\begin{proof}
Since \(\mathfrak D^0_{N,L}\le\mathfrak D_{N,L}(\eta)\),
\[
 Q_N\le P_N(\eps,\eta)
 +[S_N(\eta)-P_N(\eps,\eta)]_+,
\]
and
\[
 Q_N\ge P_N(\eps,\eta)
 -[P_N(\eps,\eta)-S_N(\eta)]_+
 -\{S_N(\eta)-Q_N\}.
\]
Taking the relevant positive parts, then \(\limsup_N\), and finally the
joint liminf over the two regularizers proves
\eqref{eq:primitive-up}--\eqref{eq:primitive-low}.
\end{proof}

Define the uncut determinant tilt
\[
 \widehat{\mathbb Q}_{N,L}^{\eta}
 (\dd O\,\dd h\,\dd A)
 :=
 \frac{e^{-NV(h)}
 \det(A^2+\eta^2I_n)^{1/2}}
 {\mathfrak D_{N,L}(\eta)}
 \Lambda_{N,L}(\dd O\,\dd h\,\dd A),
\]
write \(\widehat{\mathbb E}_{N,L}^{\eta}\) for expectation under it,
and put
\begin{equation}\label{eq:small-singular-ratio}
 \mathfrak r_\eta(A)
 :=
 \frac{\abs{\det A}}
 {\det(A^2+\eta^2I_n)^{1/2}}\in[0,1].
\end{equation}
The exact change of density gives
\begin{equation}\label{eq:D0-ratio}
 -\frac1N\log
 \widehat{\mathbb E}_{N,L}^{\eta}\mathfrak r_\eta(A)
 =
 S_N(\eta)-Q_N.
\end{equation}
Moreover, differentiation under the determinant slice yields
\begin{align}
 S_N(\eta)-Q_N
 &=
 \frac1N\int_0^\eta
 \widehat{\mathbb E}_{N,L}^{s}
 \Tr\!\left[s(A^2+s^2I_n)^{-1}\right]\dd s.
 \label{eq:resolvent-exact-identity}
\end{align}
Indeed, for each \(A\) the function
\(\det(A^2+s^2I_n)^{1/2}\) is nondecreasing in \(s\).  Finiteness of
\(\mathfrak D_{N,L}(\eta)\) therefore gives finiteness of every slice
with \(0\le s\le\eta\).  Its pointwise fundamental theorem of calculus
has a nonnegative derivative, so Tonelli yields the corresponding
identity after integration against the hard-slice measure.  For \(s>0\),
\[
 \Tr\!\left[s(A^2+s^2I_n)^{-1}\right]\le\frac ns,
\]
which makes the derivative finite at each positive \(s\).  Hence
\(\mathfrak D_{N,L}(s)\) is absolutely continuous on \([0,\eta]\), and
integrating its logarithmic derivative, as an improper integral at
zero, proves \eqref{eq:resolvent-exact-identity}.  Thus
\(\Sigma_{L,0}\) is exactly the exponential near-zero resolvent defect,
rather than an additional tail assumption.

\begin{proposition}[Gaussian divisibility removes the upper slice defect]
\label[proposition]{prop:gaussian-divisible-upper-slice}
In the preceding frame-slice setting, suppose
\[
 J_{I,N}=\sqrt{1-\tau_N}\,X_{I,N}+\sqrt{\tau_N}\,G_{I,N},
 \qquad \tau_N\in[\tau_0,1],
\]
for a deterministic sequence \((\tau_N)\) and some fixed
\(\tau_0>0\), where \(X\) is an arbitrary random tensor and \(G\) is an
independent standard Gaussian tensor.  Suppose only
that the hard and smoothed slices appearing above are finite and
positive.  Gaussian convolution then supplies a global strictly
positive gradient density and a compatible conditional kernel.  For
every \(\eta>0\),
\begin{equation}\label{eq:Gaussian-divisible-slice}
 \mathfrak K_{N,L}(\eps,\eta)
 \ge c_{\eps,\sqrt{\tau_N p}}^{\,n}\mathfrak D_{N,L}(\eta),
 \qquad
0<c_{\eps,\sigma}\le1,\qquad
 \inf_{\sigma\in[\sqrt{\tau_0p},\sqrt p]}c_{\eps,\sigma}
 \longrightarrow1
 \quad(\eps\downarrow0).
\end{equation}
Consequently, the upper primitive defect in \eqref{eq:primitive-up}
vanishes, and hence so does the direct upper de-regularization defect.
The same conclusion holds for the point-incoherent slice.
\end{proposition}

\begin{proof}
At a fixed frame, the Gaussian gradient
\(g_G\sim\cN(0,pI_n)\) is independent of the Gaussian energy--Hessian
pair.  Condition on \(X\) and on that pair.  With
\(\sigma^2=\tau_N p\), put
\[
 \varphi_\sigma(t)
 =\frac1{\sqrt{2\pi}\sigma}e^{-t^2/(2\sigma^2)},
 \qquad
 y=\sqrt{1-\tau_N}\,g_X.
\]
The exact and mollified gradient slices contribute,
coordinatewise,
\[
 \varphi_\sigma(y),
 \qquad
 (\kappa_\eps*\varphi_\sigma)(y),
\]
respectively, while the energy times smoothed-determinant factor is the
same nonnegative weight.

For the pointwise comparison, let \(U\sim\chi_1^2\) and
\(S_\eps=\eps^2/U\).  A centered Cauchy variable of scale \(\eps\) is a
normal variance mixture with conditional variance \(S_\eps\).  Hence
\begin{align*}
 \frac{(\kappa_\eps*\varphi_\sigma)(y)}
 {\varphi_\sigma(y)}
 &=
 \E\left[
 \frac{\sigma}{\sqrt{\sigma^2+S_\eps}}
 \exp\left\{
 \frac{y^2S_\eps}{2\sigma^2(\sigma^2+S_\eps)}
 \right\}\right]\\
 &\ge
 \E\frac{\sigma}{\sqrt{\sigma^2+S_\eps}}
 =:c_{\eps,\sigma}.
\end{align*}
Plainly \(0<c_{\eps,\sigma}\le1\), and it is nondecreasing in
\(\sigma\).  Under the coupling \(S_\eps\to0\) almost surely, dominated
convergence at \(\sigma=\sqrt{\tau_0p}\) gives the stated uniform
convergence.  Multiply the inequality over
the \(n\) gradient coordinates and integrate all remaining variables
and frames to prove \eqref{eq:Gaussian-divisible-slice}.
\end{proof}

\begin{proposition}[Gaussian benchmark]
\label[proposition]{prop:Gaussian-dereg-benchmark}
Let \(J=G\) be standard Gaussian disorder, let \(V\) be continuous and
bounded below, and suppose the chosen frame or point slice has positive
geometric measure.  Then all three primitive defects in
\eqref{eq:Gamma-plus}--\eqref{eq:Sigma-zero} vanish in the required
joint-liminf sense.  For the point-incoherent slice, both direct
de-regularization defects therefore vanish; the same holds for the
analogously defined frame-weighted direct defects.
\end{proposition}

\begin{proof}
Let
\[
 \varphi_p(t)=\frac1{\sqrt{2\pi p}}e^{-t^2/(2p)}.
\]
Gradient independence in \cref{prop:Gaussian-jet} gives the exact ratio
\begin{equation}\label{eq:Gaussian-KD-ratio}
 \frac{\mathfrak K^G_{N,L}(\eps,\eta)}
 {\mathfrak D^G_{N,L}(\eta)}
 =
 \left(\frac{a_{\eps,p}}{\varphi_p(0)}\right)^n.
\end{equation}
Since \(a_{\eps,p}\to\varphi_p(0)\), both gradient defects vanish as
\(\eps\downarrow0\), uniformly in \(\eta\).

Put \(\Phi=\Phi_p^{\rm bulk}(V)\) and
\begin{align}
 F_0(\eta)
 :={}&\frac12\log(2\pi e)+\log\varphi_p(0)\notag\\
 &+\sup_u\left\{-\frac{u^2}{2}-V(u)
 +\ell_{p,\ell_\eta}(u)\right\}.
 \label{eq:Gaussian-exact-gradient-limit}
\end{align}
\Cref{thm:gaussian-variational}, the exact Gaussian Kac--Rice formula,
and isotropy give
\[
 \lim_N\frac1N
 \log\frac{\mathfrak D^G_{N,L}(\eta)}
               {\mathfrak D^{G,0}_{N,L}}
 =F_0(\eta)-\Phi.
\]
The same geometric frame or point factor appears in numerator and
denominator and cancels, so no zero-cost condition is needed.  Since
\(\mathfrak D^G_{N,L}(\eta)\ge\mathfrak D^{G,0}_{N,L}\), the left side
is nonnegative.  Equation \eqref{eq:uncut-dereg-limit} makes its limit
tend to zero as \(\eta\downarrow0\).  Thus all three primitive
joint-liminf defects
vanish; \eqref{eq:primitive-up}--\eqref{eq:primitive-low} imply that
both direct defects vanish.
\end{proof}

\begin{remark}
The uncut comparison is proved here only under the uniform
subexponential hypothesis.  The current Weibull-\(\alpha<1\) truncation
argument uses a bounded spectral weight and does not justify an uncut
determinant claim.  Also, smoothness of a one-dimensional coefficient
density alone does not imply an exponential gradient-slice comparison;
finite-\(N\) continuity does not prevent a slice value from being
exponentially smaller than its fixed-width convolution.
\end{remark}

\section{Morse regularity for smooth coefficient laws}\label{app:Morse}

\begin{lemma}[Generic Morse property]\label[lemma]{lem:Morse}
For fixed \(N,p\), the set of coefficient tensors for which the restriction of \(H_N\) to \(\SN\) is not Morse is contained in a proper real algebraic subset of \(\R^{N^p}\).  Consequently, if the joint coefficient law is absolutely continuous, the field is almost surely Morse.
The same conclusion holds when the coefficient law is absolutely
continuous with respect to Lebesgue measure on the space of symmetric
degree-\(p\) tensors.
\end{lemma}

\begin{proof}
A non-Morse critical point is a solution of the Lagrange multiplier equations
\[
 \nabla H_N(\sigma)=\lambda\sigma,
 \qquad \norm\sigma^2=N,
\]
together with
\[
 \det\begin{pmatrix}
 D^2H_N(\sigma)-\lambda I_N&\sigma\\
 \sigma^{\mathsf T}&0
 \end{pmatrix}=0.
\]
The standard discriminant construction for critical points of a polynomial restricted to the smooth complex quadric \(\{\sigma^{\mathsf T}\sigma=N\}\) gives a polynomial in the coefficients that vanishes whenever this degeneracy system has a solution.  Equivalently, one may homogenize and saturate the displayed incidence system before eliminating \((\sigma,\lambda)\).  It remains to show that this discriminant is not identically zero.

Choose \(c_1,\ldots,c_N>0\) and consider
\[
 H_0(\sigma)=\sum_{i=1}^Nc_i\sigma_i^p.
\]
The same calculation may be made after complexifying the equations.  The constraint \(\sigma^{\mathsf T}\sigma=N\) implies \(\lambda\ne0\), and every active coordinate satisfies \(pc_i\sigma_i^{p-2}=\lambda\).  If \(A=\{i:\sigma_i\ne0\}\), then
\[
 T_\sigma\{z:z^{\mathsf T}z=N\}
 =\{v_A:\sigma_A^{\mathsf T}v_A=0\}\oplus\mathbb C^{A^c}.
\]
Because \(\sigma_A^{\mathsf T}\sigma_A=N\ne0\), the standard complex bilinear form is nondegenerate on \(\sigma_A^\perp\).  On the first summand the constrained Hessian \(D^2H_0-\lambda I_N\) is \((p-2)\lambda\) times this form, and on the second it is \(-\lambda\) times the standard form.  Its restriction to the complex tangent space is therefore nonsingular.  Thus the complexified degeneracy system has no solution for \(H_0\), the discriminant polynomial is nonzero, and its real zero set is a proper algebraic subset of Lebesgue measure zero.
Because \(H_0\) itself is symmetric, the discriminant remains nonzero
when restricted to the symmetric coefficient space.  Equivalently, the
symmetrization map from ordered coefficients onto polynomial
coefficients is surjective, and the inverse image of the proper
discriminant is proper.  This proves both absolute-continuity
statements.
\end{proof}

\section*{Acknowledgments}

The author would like to thank G\'erard Ben Arous, Ji Oon Lee, and
Michel Talagrand for many helpful discussions.

\section*{Statements and Declarations}

\noindent\textbf{Funding.}
This work was supported by the KAIST Jang Young Sil Fellow Program.

\smallskip
\noindent\textbf{Competing interests.}
The author declares no competing interests relevant to this article.

\smallskip
\noindent\textbf{Data availability.}
No datasets were generated or analyzed for this theoretical study.

\end{document}